\documentclass[aap]{imsart}

\startlocaldefs
\RequirePackage{amsthm,amsmath,amsfonts,amssymb}
\usepackage{silence}
\usepackage{ifdraft}
\usepackage{amsfonts}
\usepackage{amsthm}
\usepackage{color}
\usepackage{exmath}
\usepackage{graphicx}
\usepackage{hyperref}
\usepackage{comment}
\usepackage{lmodern}
\usepackage{xstring}

\usepackage{pgfplots}
\usepackage{pgfplotstable}
\usepackage{pgfconfig}
\usepackage[edges]{forest}
\usepackage{fmtcount}   %
\usepackage{subcaption}

\usepackage{amsmath}
\usepackage{xcolor}

\newtheoremstyle{mythmstyle}
  {\topsep}
  {\topsep}
  {\itshape}
  {0pt}
  {}
  {\bfseries{.}}
  {5pt plus 1pt minus 1pt}
  {\def\temp{#3}%
   \ifx\temp\empty\def\mynote{#3}\else\def\mynote{\ (#3)}\fi%
   {\bfseries{\thmname{#1}\thmnumber{ #2}}}\thmnote{\mynote}%
   \belowpdfbookmark{#1 #2\mynote}{#1-#2}}
\theoremstyle{mythmstyle}

\pgfplotsset{%
  every axis post/.code={%
    \if\thesubfigure\empty\else%
    \pgfkeyssetvalue{/pgfplots/title}{(\thesubfigure)}%
    \fi},
}

\def\subfigtag#1{}
\DeclareCaptionLabelFormat{subreffig}{({\emph{#2}})}
\hypersetup{
    colorlinks,
    linkcolor={blue!50!black},
    citecolor={green!50!black},
    urlcolor={red!80!black}
}

\usepgfplotslibrary{polar}
\usepgfplotslibrary{fillbetween}
\usetikzlibrary{decorations.text}  %
\usetikzlibrary{shapes,arrows,shapes.multipart}  %

\pgfkeys{/pgf/images/include external/.code=\includegraphics{#1}}

\pgfplotsset{%
  width = 6.1cm,
  title style={yshift=-1.5ex},
  legend entry/.initial=,
  every axis plot post/.code={%
    \pgfkeysgetvalue{/pgfplots/legend entry}\tempValue
    \ifx\tempValue\empty
    \pgfkeysalso{/pgfplots/forget plot}%
    \else
    \expandafter\addlegendentry\expandafter{\tempValue}%
    \fi
  },
}

\def\addref[#1]#2#3{
  \addplot[
  forget plot,
  mark=none,
  dotted,
  decoration={
    text along path,
    text={{#2}{}},
    #1
  },
  postaction={decorate},
  ][domain=\pgfkeysvalueof{/pgfplots/xmin}:\pgfkeysvalueof{/pgfplots/xmax}]
  {#3};
}

\def\I#1{\mathbb I_{#1}}
\def\tol{\varepsilon}
\def\Diff#1{\mathrm{D}_{#1}}

\xParseDeclareExpectation{\maxp}{\max}\{\}
\xParseDeclareExpectation{\minp}{\min}\{\}

\NewDocumentCommand{\dist}{e{^} o m}
{d_{\IfValueTF{#2}{#2}{K}}
  \IfValueT{#1}{^{#1}}\p{#3}}
\NewDocumentCommand{\Fass}{e{^} m}
{F\IfValueT{#1}{^{#1}}\p{#2}}
\def\diagbeta{\beta_{\textnormal{d}}}
\def\crossbeta{\beta_{\textnormal{c}}}
\def\Zest{\Delta \mathcal{P}}

\NewDocumentCommand{\X}{e{^_} o m}{%
  \IfValueT{#3}{\overline} X%
  \IfValueT{#1}{^{#1}} %
  _{\IfValueT{#2}{{#2},} %
    \IfValueT{#3}{{#3},} %
    #4} %
}
\NewDocumentCommand{\Xa}{e{^_} o m}{%
  \IfValueT{#3}{\overline} X%
  ^{\p{a}\IfValueT{#1}{{,#1}}} %
  _{\IfValueT{#2}{{#2},} %
    \IfValueT{#3}{{#3},} %
    #4} %
}

\NewDocumentCommand{\g}{e{^_} o m}
{%
  \IfValueT{#3}{\overline} g%
  \IfValueT{#1}{^{#1}} %
  _{\IfValueT{#2}{{#2},} %
    \IfValueT{#3}{{#3},} %
    #4} %
}
\def\parent#1{\langle #1 \rangle}
\def\niceset{\(\p{\textnormal{Si}}\)}
\def\half{{\textstyle{\frac{1}{2}}}}
\def\N{2}
\def\eps{\varepsilon}

\def\spellout#1{%
  \if!\ifnum9<1#1!\fi%
  \text{\numberstringnum{#1}}\else#1\fi}
\usepackage{environ}
\NewEnviron{customlegend}[1][]{
  \begin{tikzpicture}
    \begingroup
    \def\addlegendimage{\csname pgfplots@addlegendimage\endcsname}
    \def\addlegendentry{\csname pgfplots@addlegendentry\endcsname}
  \csname pgfplots@init@cleared@structures\endcsname
  \pgfplotsset{#1}
  \BODY
  \csname pgfplots@createlegend\endcsname
  \endgroup
\end{tikzpicture}}

\makeatletter
\def\bdoi#1{\
  \check@doiurl@prefix#1\end
\ifnum\@doiurlfull>\z@
  \ims@href{#1}{#1}%
\else
  \ims@href{\doi@base#1}{%
    \doi@base%
    \saveexpandmode\noexpandarg%
    \StrSubstitute{#1}{_}{\_}%
    \restoreexpandmode}%
\fi%
\ignorespaces}

\makeatother
 \usepackage{todonotes}

\usepackage[numbers]{natbib}
\usepackage[capitalise]{cleveref}

\crefname{assumption}{Assumption}{Assumptions}
\crefname{equation}{}{}

\AtBeginDocument{%
  \DeclareFontShape{T1}{lmr}{m}{scit}{<->ssub*lmr/m/scsl}{}%
}

\newtheorem{theorem}{Theorem}[section]
\newtheorem{lemma}[theorem]{Lemma}
\newtheorem{corollary}[theorem]{Corollary}
\newtheorem{assumption}[theorem]{Assumption}
\newtheorem{definition}[theorem]{Definition}
\newtheorem{remark}[theorem]{Remark}

\renewcommand{\bdoi}[1]{\href{http://doi.org/#1}{doi:\detokenize{#1}}}
\newcommand{\beprint}[1]{\href{https://arxiv.org/abs/#1}{arxiv:\detokenize{#1}}}

\endlocaldefs

\showdetails
\begin{document}

\begin{frontmatter}
  \title{Multilevel Path Branching for Digital Options}
  \runtitle{Multilevel Path Branching for Digital Options}

  \begin{aug}
    \author[A]{\fnms{Michael
        B.}~\snm{Giles}\ead[label=e1]{mike.giles@maths.ox.ac.uk}\orcid{0000-0002-5445-3721}}
    \and
    \author[B]{\fnms{Abdul-Lateef}~\snm{Haji-Ali}\ead[label=e2]{a.hajiali@hw.ac.uk}\orcid{0000-0002-6243-0335}}
    \address[A]{Mathematical Institute, University of Oxford%
      \printead[presep={,\ }]{e1}}

    \address[B]{Maxwell Institute,
      School of Mathematical and Computer Sciences,
      Heriot-Watt University%
      \printead[presep={,\ }]{e2}}
  \end{aug}

\begin{abstract}
  We propose a new Monte Carlo-based estimator for digital options with assets
  modelled by a stochastic differential equation (SDE). The new estimator is
  based on repeated path splitting and relies on the correlation of
  approximate paths of the underlying SDE that share parts of a Brownian path.
  Combining this new estimator with Multilevel Monte Carlo (MLMC) leads to an
  estimator with a computational complexity that is similar to the complexity
  of a MLMC estimator when applied to options with Lipschitz payoffs.
\end{abstract}

  \begin{keyword}[class=MSC]
    \kwd[Primary ]{65C05}
    \kwd{65C30}
    \kwd[; secondary ]{65B99}
    \kwd{60J85}
  \end{keyword}

  \begin{keyword}
    \kwd{Monte Carlo}
    \kwd{Multilevel}
    \kwd{Path splitting}
    \kwd{Computational complexity}
    \kwd{Branching processes}
  \end{keyword}
\end{frontmatter}

\section{Introduction} %
In its simplest form, the Multilevel Monte Carlo (MLMC) path simulation method
\cite{giles:MLMC} considers a scalar SDE
\begin{equation}\label{eq:sde-scalar}
  \D \X{t} = a\p{\X{t},t}\, \D t + \sigma\p{\X{t},t}\, \D W_t,
\end{equation}
for \(t \in [0,1]\) with a sequence of approximate paths \(\br{\p{\X[\ell]{t}}_{t \in
    [0,1]}}_{\ell \in \br{0,1,\ldots}}\) using uniform timesteps of size \(h_\ell=h_{0}
M^{-\ell}\) for some \(h_{0} \in \rset_{+}\) and \(M\in\zset_+\). If we are
interested in estimating \(\E{f\p{\X{1}}} \approx \E{f\p{\X[L]{1}}}\) for some
function \(f\) and we define \(\Delta P_{\ell} \defeq f\p{\X[\ell]{1}} -
f\p{\X[\ell-1]{1}}\) with \(\Delta P_{0} \defeq f\p{\X[0]{1}}\), we have the
telescoping summation
\[
  \E{P_L} = \sum_{\ell=0}^L \E{\Delta P_\ell}.
\]
The MLMC estimator is then
\begin{equation}\label{eq:MLMC}
\sum_{\ell=0}^L \frac{1}{N_\ell} \sum_{n=1}^{N_\ell} \Delta P_\ell^{\p{i}},
\end{equation}
with the coarse and fine paths within \(\Delta P_\ell^{\p{i}}\) based on the same
driving Brownian path. If there are constants \(\alpha, \beta, \gamma\) such that the cost
of a level \(\ell\) sample \(\Delta P_\ell\) is \(W_\ell\sim 2^{\gamma\ell}\), its variance is \(V_\ell
\defeq \var{\Delta P_{\ell}} \sim 2^{-\beta \ell}\), and the weak error is
\(\abs{\E{f\p{\X[L]{1}}-f\p{\X{1}}}} \sim 2^{-\alpha L}\), then an optimal number of
levels, \(L\), and an optimal number of samples per level,
\(\br{N_{\ell}}_{\ell=0}^{L}\), can be chosen to achieve a root-mean-square accuracy
of \(\eps\) with a computational complexity which is \(\Order{\eps^{-2}}\) if
\(\beta>\gamma\), \(\Order{\eps^{-2}\abs{\log \eps}^2}\) if \(\beta=\gamma\), and
\(\Order{\eps^{-2-\p{\gamma-\beta}/\alpha}}\) if \(\beta<\gamma\)~\cite{giles:acta}.

If the function \(f\) is globally Lipschitz, with constant \(L_f\), then
\[
V_\ell \leq L_f^2\ \E{\p{\X[\ell]{1}-\X[\ell-1]{1}}^2}.
\]
In the case of the Euler-Maruyama discretization when the SDE coefficients,
\(a\) and \(\sigma\), are Lipschitz and grow linearly in \(x\) and are
\(1/2\)-H\"older continuous in \(t\), this results in \(V_\ell=\Order{h_\ell}\)
\cite[Theorem 10.2.2]{kloden:numsde} along with \(W_\ell=\Order{h_\ell^{-1}}\), so
\(\beta=\gamma\) and the computational complexity is \(\Order{\eps^{-2}\abs{\log
    \eps}^2}\). When using a first-order Milstein discretization, and under
additional differentiablity assumptions on \(\sigma\), the variance is reduced to
\(V_\ell=\Order{h_\ell^{2}}\) and the complexity is improved to
\(\Order{\eps^{-2}}\). A limitation of the first-order Milstein discretization
is that it often requires the simulation of L\'evy areas for multi-dimensional
SDEs. To avoid this, Giles \& Szpruch \cite{giles:antithetic} developed an
antithetic, truncated Milstein estimator which omits these L\'evy area terms
and still achieves an MLMC variance \(V_\ell\) which is \(\Order{h_\ell^{2}}\) when
\(f\) is smooth, and \(\Order{h_\ell^{3/2}}\) when \(f\) is Lipschitz and
piecewise smooth; both are sufficient for the computational complexity of MLMC
to be \(\Order{\eps^{-2}}\).

In this article, we are concerned with the more difficult case in which \(f\)
is a discontinuous function such as \(f\p{x}=\I{x>K}\); in computational
finance this is referred to as a digital option.
In this case \(\Delta P_\ell\) is nonzero only if the final values of the fine and
coarse path approximations \(\X[\ell]{1}\) and \(\X[\ell-1]{1}\) within \(\Delta P_{\ell}\)
are on opposite sides of \(K\). Speaking loosely (we will be precise later),
in the case of using Euler-Maruyama discretization, this only happens if
\(\X{1}, \X[\ell]{1}, \X[\ell-1]{1}\) are all within \(\Order{h_\ell^{1/2}}\) of \(K\),
and the probability of that is \(\Order{h_\ell^{1/2}}\). Hence \(V_\ell \approx
\Order{h_\ell^{1/2}}\) and so \(\beta \approx \gamma/2\), resulting in a computational
complexity which is approximately \(\Order{\eps^{-5/2}}\) since standard weak
convergence results give \(\alpha = 2\beta\). With the Milstein discretization, \(V_\ell \approx
\Order{h_\ell}\) and the complexity is improved to
\(\Order{\eps^{-2}\abs{\log{\eps}}^{2}}\), but with the antithetic Milstein
estimator \(V_\ell\) remains \(\Order{h_\ell^{1/2}}\).

The challenge of discontinuous functions such as this has been tackled in
previous research. In the context of the first-order Milstein approximation, a
conditional expectation with respect to the Brownian increment for the final
timestep, conditional on the Brownian path up to that point, has been used to
decrease the variance \(V_\ell\) from \(\Order{h_\ell}\) to \(\Order{h_\ell^{3/2}}\)
\cite{giles:milstein-analysis}. In simple cases the conditional expectation
can be evaluated analytically \cite{giles:milstein-analysis}, while in harder
cases a change of measure or path splitting can be used \cite{giles:acta}.
Unfortunately, none of these approaches work with the Euler-Maruyama
discretization. One method which is effective for a subset of cases with
particularly simple functions \(f\) is ``pre-integration'', a variant of
conditional expectation or conditional sampling based on the final value of
the driving Brownian path. Originally developed to improve the effectiveness
of Quasi-Monte Carlo integration
\cite{acn:conditonal-sampling,gks:smoothing-anova,gkls:smoothing-kinks}, it
also works well with MLMC \cite{bayer:smoothing}. Another effective method
uses adaptive refinement of paths which lie close to the discontinuity
\cite{hajiali:adaptive}. When used for Euler-Maruyama or Milstein schemes,
adaptive refinement methods recover the convergence rates of the variance,
\(V_{\ell}\), that are observed for Lipschitz functions without substantially
increasing the cost per sample. However, these methods lead to estimators with
high kurtosis which can cause difficulties for MLMC algorithms that rely on
variance estimates. Additionally, adaptive refinement does not recover the
improved variance convergence rates of antithetic Milstein. See also
\cite{giles:mlmc-discont} for a more thorough discussion of existing
methodologies.

Inspired in part by the literature on dyadic Branching Brownian Motion
\cite{alison:superprocesses,mckean:application-bm}, the idea that we develop
in the current article, as illustrated in \cref{fig:branching-estimator},
builds on path splitting where each MLMC sample, instead of corresponding to a
single pair of fine and coarse paths, is an average of the difference \(\Delta
P_\ell\) from many particles, each of which is a pair of fine and coarse paths.
The branching process to generate the particles is similar to the process of
dyadic Branching Brownian Motion, except that the branching times are
deterministic not exponential random times.
\cref{fig:branching-est-tree} illustrates the logical structure of the
particle generation. If there are \(2^\ell\) timesteps for the fine path
approximation \(\X[\ell]{\cdot}\) and \(2^{\ell-1}\) timesteps for the coarse path
\(\X[\ell-1]{\cdot}\), then in the simplest version of the algorithm the first
branching from 1 to 2 particles is after \(2^{\ell-1}\) fine timesteps, the
second branch from 2 to 4 particles is after another \(2^{\ell-2}\) fine
timesteps, and so on, until there is only one coarse timestep left, at which
there is a final branching into \(2^{\ell-1}\) particles. This gives the
following number of particles at different stages of the calculation:

\begin{figure}\centering
  \pgfdeclarelayer{markers}
  \pgfsetlayers{main, markers}
\makeatletter
\catcode`\!=4
\def\createlist#1{\gdef#1{}}
\def\push#1#2{\xdef#1{#1#2!}}
\def\qtop#1#2{\edef#2{\expandafter\@q@topitem#1!\@q@xpop!\relax}}
\def\@q@topitem#1!#2\relax{#1}
\def\qpop#1#2{\qtop#1#2\pop#1}
\def\pop#1{\xdef#1{\expandafter\@q@xpop#1}}
\def\@q@xpop#1!{}
\def\print#1{\expandafter\@q@xprint#1!\@q@xpop!\relax}
\def\@q@xprint#1!#2\relax{\ifx\@q@xpop#1\else\@q@xprint#2\relax#1\fi}
\catcode`\!=12
\makeatother

\usetikzlibrary{calc}  %
\newenvironment{tikzaxis}[1][\unskip]
{\begin{axis}[#1]%
    \coordinate (O) at (axis cs:0,0);%
    \coordinate (X) at (axis cs:1,0);%
    \coordinate (Y) at (axis cs:0,1);%
    \coordinate (SW) at (\pgfkeysvalueof{/pgfplots/xmin},\pgfkeysvalueof{/pgfplots/ymin});%
    \coordinate (NE) at (\pgfkeysvalueof{/pgfplots/xmax},\pgfkeysvalueof{/pgfplots/ymax});%
  \end{axis}%
  \scope[x={($(X)-(O)$)}, y={($(Y)-(O)$)}, shift={(O)}]%
  \clip  (SW) rectangle (NE);}
{\endscope}

\makeatletter
\newcommand{\gettikzxy}[3]{%
  \pgfpointanchor{#1}{center}%
  \edef#2{\the\pgf@x}%
  \edef#3{\the\pgf@y}%
}
\makeatother

\pgfmathdeclarefunction{randn}{2}{%
  \pgfmathparse{ #1 + sqrt(#2) * sqrt(-2*ln(rnd))*cos(deg(2*pi*rnd)) }}

\newcommand{\brownian}[4][]{%
  \draw[#1] (#2) \foreach \x in {1,...,#4}{ -- ++(#3/#4,{randn(0,#3/#4)})}}

\createlist{\coordlist}
\def\pushCoord#1{\gettikzxy{#1}{\XCoord}{\YCoord}%
  \push{\coordlist}{\XCoord}%
  \push{\coordlist}{\YCoord}}
\def\popCoord#1{
  \qpop{\coordlist}{\XCoord}%
  \qpop{\coordlist}{\YCoord}%
  \coordinate (#1) at (canvas cs: x=\XCoord, y=\YCoord);}

\pgfmathsetmacro\vL{2}

\gdef\treeXticks{} %
\gdef\treeXticksLabels{} %
\foreach\x in {0,...,\vL} {%
  \pgfmathparse{1-2^-(\x+1)}%
  \xdef\treeXticks{\treeXticks\pgfmathresult\ifnum\x<\vL,\fi}%
  \xdef\treeXticksLabels{\treeXticksLabels $1{-}\tau_{\x}$\ifnum\x<\vL,\fi }
}

\makeatletter
\pgfkeys{
  /pgfplots/xticklabels/.code={%
    \expandafter\pgfplotslistnew\expandafter\pgfplots@xticklabels\expandafter{#1}%
    \let\pgfplots@xticklabel=\pgfplots@user@ticklabel@list@x%
  },
  /pgfplots/extra x tick labels/.code={%
    \expandafter\pgfplotslistnew\expandafter\pgfplots@extra@xticklabels\expandafter{#1}%
    \let\pgfplots@extra@xticklabel=\pgfplots@user@extra@ticklabel@list@x%
  }
}
\makeatother

\pgfdeclarelayer{markers}
\pgfsetlayers{main, markers}

  \begin{subfigure}[t]{0.5\textwidth}\centering
    \subfigtag{{left}}\phantomsubcaption%
\begin{tikzpicture}[
  branch/.style={mark=*,mark options={color=white,draw=black}},
  final/.style={mark=+,mark options={color=white,draw=black}}]
  \begin{tikzaxis} [%
    axis x line=bottom,hide y axis,
    xmin=0, xmax=1.1, ymin=-1.05, ymax=1.05,
    xtick = {0,1},
    extra x ticks = {\treeXticks},
    extra x tick labels = {\treeXticksLabels},
    extra x tick style={
      tick label style={font=\small,rotate=50, anchor=east}},
    xlabel={\(t\)}]

    \draw [thick] (0,0) -- (0.5,0) coordinate (A) ;
    \begin{pgfonlayer}{markers}
      \draw plot[mark=*,mark options={color=white,draw=black}] (A);
    \end{pgfonlayer}

    \foreach \l in {0,...,\vL}
    {
      \pgfmathsetmacro\vxstart{1-2^(-(1+\l))}
      \pgfmathsetmacro\vyDelta{2^(-(1+\l))}
      \pgfmathsetmacro\vCount{2^(\l)}
      \pgfmathsetmacro\vxDelta{2^(-(1+\l)) - (\l < \vL? 2^(-(1+\l)-1) : 0)}

      \foreach \b in {1,...,\vCount}%
      {
        \coordinate (S) at (\vxstart, { 2*(2*\b-1)*\vyDelta - 1});

        \draw [thick] (S) -- ++(\vxDelta,+\vyDelta) coordinate (B);
        \draw [thick] (S) -- ++(\vxDelta,-\vyDelta) coordinate (A);

        \def\style{branch}
        \ifnum\l=\vL
        \def\style{final}
        \fi
        \begin{pgfonlayer}{markers}
          \draw plot[\style] (A) plot[\style] (B);
        \end{pgfonlayer}

        \ifnum\l=\vL
        \node[right] at (A) {\pgfmathparse{int(\b*2-1)}\pgfmathresult};
        \node[right] at (B) {\pgfmathparse{int(\b*2)}\pgfmathresult};
        \fi
      }%
    };
  \end{tikzaxis}
\end{tikzpicture}
     \label{fig:branching-est-tree}
  \end{subfigure}\hfill
  \begin{subfigure}[t]{0.5\textwidth}\centering
    \subfigtag{{right}}\phantomsubcaption%
\pgfmathsetseed{1}
\begin{tikzpicture}[
  branch/.style={mark=*,mark options={color=white,draw=black}},
  final/.style={mark=+,mark options={color=white,draw=black}}]
  \begin{tikzaxis}[
    axis x line=bottom,hide y axis,
    ymin=-1.1,ymax=1.8,xmin=0,xmax=1.04,
    xtick = {0,1},
    extra x ticks = {\treeXticks},
    extra x tick labels = {\treeXticksLabels},
    extra x tick style={
      tick label style={font=\small,rotate=50, anchor=east}},
    xlabel={\(t\)}]

    \pgfmathsetmacro\vN{300}  %
    \pgfmathsetmacro\vNi{int(round(\vN*0.5))}
    \brownian{0,0}{0.5}{\vNi} coordinate (A);
    \pushCoord{A};

    \foreach \l in {0,...,\vL}
    {
      \pgfmathsetmacro\vCount{2^((\l+1)-1)}
      \pgfmathsetmacro\vxDelta{2^(-(1+\l)) - (\l < \vL? 2^(-(1+\l)-1) : 0)}
      \pgfmathsetmacro\vNi{int(round(\vN*\vxDelta))}

      \foreach \b in {1,...,\vCount}%
      {
        \popCoord{S};
        \begin{pgfonlayer}{markers}
          \draw plot[branch] (S);
        \end{pgfonlayer}

        \brownian{S}{\vxDelta}{\vNi} coordinate (B);
        \pushCoord{B};

        \brownian{S}{\vxDelta}{\vNi} coordinate (C);
        \pushCoord{C};
      }%
    }

    \pgfmathsetmacro\vCount{2^(1+\vL)}
    \begin{pgfonlayer}{markers}
      \foreach \b in {1,...,\vCount}{
        \popCoord{S};
        \draw plot[final] (S);
      }
    \end{pgfonlayer}
  \end{tikzaxis}
\end{tikzpicture}
     \label{fig:branching-est-tree-bm}
  \end{subfigure}\hfill

  \caption{An illustration of the branching estimator \(\Zest_{4}\) defined in
    \cref{def:est-main} for \(\tau_{\ell'}=2^{-\ell'-1}\) and \(h_{\ell}=2^{-\ell}\).
    \subref{fig:branching-est-tree} shows the logical structure ending up in
    the eight correlated samples of \(\Delta P_{4}\).
    \subref{fig:branching-est-tree-bm} shows the eight underlying, correlated
    Brownian paths.}
  \label{fig:branching-estimator}
\end{figure}

\vspace{0.2in}

\begin{tabular}{rcl}
  1          & particle for first & \(2^{\ell-1}\) fine timesteps \\
  2          & particles for next & \(2^{\ell-2}\) fine timesteps \\
  \(4=2^2\)    & particles for next & \(2^{\ell-3}\) fine timesteps \\
  \vdots     & \vdots & \vdots\\
  \(2^{\ell-2}\) & particles for next  & \(2\) fine timesteps \\
  \(2^{\ell-1}\) & particles for final  & \(2\) fine timesteps \\
\end{tabular}

\vspace{0.2in}

\noindent
so that the total cost (i.e.~total number of particle-timesteps) is
\(\p{\ell{+}1}\,2^{\ell-1}\) which is not much more than the usual \(2^\ell\) cost per
sample.
The MLMC sample value would be an average of the outputs from the \(2^{\ell-1}\)
particles:
\[
  \Zest_\ell \defeq \frac{1}{2^{\ell-1}} \sum_{i=1}^{2^{\ell-1}} \Delta P^{\p{i}}_\ell,
\]
i.e.~this \(\Zest_\ell\) counts as a single sample within an MLMC estimator
similar to \cref{eq:MLMC}.

The claim is that with the Euler-Maruyama discretization we obtain
\(\var{\Zest_\ell} \approx \Order{h_\ell}\) so that approximately we have \(\beta\approx 1, \gamma \approx 1\).
We present here a heuristic analysis which we make rigorous later. Suppose two
particles \(\p{i}\) and \(\p{j}\) share a common driving Brownian path up to
time \(1{-}\tau\). Conditional on \(\X{1-\tau}\), the distribution of \(\X{1}\) is
approximately Normal with a standard deviation of \(\Order{\tau^{1/2}}\) and peak
probability density of \(\Order{\tau^{-1/2}}\). For both particles to finish
within \(\Order{h_\ell^{1/2}}\) of \(K\) (by which we mean that both the coarse
and fine path approximations end within \(\Order{h_\ell^{1/2}}\) of \(K\))
requires that \(\X{1-\tau}\) lies within \(\Order{\tau^{1/2}}\) of \(K\), which
occurs with probability \(\Order{\tau^{1/2}}\), and conditional on this the
probability that each particle finishes within \(\Order{h_\ell^{1/2}}\) of \(K\)
is \(\Order{h_\ell^{1/2}\tau^{-1/2}}\). Hence, the probability that both particles
finish within \(\Order{h_\ell^{1/2}}\) of \(K\) is
\[
\Order*{ \tau^{1/2} \ \p{ h_\ell^{1/2}\tau^{-1/2}}^2} = \Order{h_\ell\,\tau^{-1/2}},
\]
and therefore
\[
\E*{\, \abs{\Delta P_\ell^{\p{i}}} \ \abs{\Delta P_\ell^{\p{j}}} \, }
= \Order{h_\ell\,\tau^{-1/2}}.
\]

There are \(2^{2\p{\ell-1}}\) possible pairs \(\p{i,j}\), and for each \(i\) the
number of particle pairs \(\p{i,j}\) with \(j\neq i\) and different \(\tau\) values
are:

\vspace{0.2in}
\begin{tabular}{rl}
  \(2^{\ell-2}\) & with \(\tau=2^{-1}\) \\
  \(2^{\ell-3}\) & with \(\tau=2^{-2}\) \\
  \(2^{\ell-4}\) & with \(\tau=2^{-3}\) \\
  \vdots     & \vdots \\
  \(2\) & with \(\tau=2^{-(\ell-2)}\) \\
  \(1\) & with \(\tau=2^{-(\ell-1)}\) \\
\end{tabular}
\vspace{0.1in}

\noindent
In addition there is the particle pair \(\p{i,i}\) for which \(\E*{\p*{\Delta
    P_\ell^{\p{i}}}^2} = \Order{h_\ell^{1/2}}\), as discussed previously. Together,
these give
\[
V_\ell
\ \leq\ \E{ \Zest_\ell^2 }
\ =\  \Order*{ 2^{-(\ell-1)}\, h_\ell^{1/2} + 2^{-(\ell-1)}\,  \sum_{\ell'=1}^{\ell-1}  2^{\ell-1-\ell'} h_\ell\, 2^{\ell'/2}}
\ =\ \Order*{ h_\ell},
\]
with the largest contribution coming from the \(\tau{=}1/2\) branch, the most common branching point.
This last observation suggests that the variance is well modelled by
\begin{eqnarray*}
  V_\ell &\approx& \var*{
          \E*{ f\p{\X{1}} \given \X{1/2}=\X[\ell]{1/2}} -
          \E*{ f\p{\X{1}} \given \X{1/2}=\X[\ell-1]{1/2}}} \\
      &\approx& \var*{
          \p{\X[\ell]{1/2} - \X[\ell-1]{1/2}} \cdot \nabla_x \E{ f\p{\X{1}} \given \X{1/2}=x} \given_{x=\X[\ell-1]{1/2}} },
\end{eqnarray*}
and we will later follow a similar approach in analyzing the branching
estimator based on the antithetic Milstein approximation.

One final point for this introduction concerns optimization of the
branching times. If \(1\!-\!\tau_{\ell'}\) is the \(\ell'\)-th branching time, then the
total cost of \(\Zest_\ell\) is of order
\[
2^\ell \p*{ 1 + \sum_{\ell'=1}^\ell 2^{\ell'} \tau_{\ell'}},
\]
and the variance bound is of order
\[
h_\ell \sum_{\ell'=1}^\ell 2^{-\ell'} \, \tau_{\ell'}^{-1/2}.
\]
Optimizing \(\tau_{\ell'}\) to minimize the cost for a fixed variance gives \(
\tau_{\ell'} \propto 2^{-4\ell'/3} \) which is slightly different to the initial choice of
\(\tau_{\ell'} {=} 2^{-\ell'}\) and eliminates the additional linear factor in the
cost. Hence, our main analysis will consider branching times \(\tau_{\ell'}{=}2^{-\eta
  \ell'}\) for some constant \(\eta\). These branching times may not coincide with
discretization timesteps. There are two ways to handle this in an
implementation. One is to round the times to the nearest coarse path timestep,
and the other is to keep the times as specified in which case there is a
common Brownian increment for the first part of the timestep, and then
independent Brownian increments for the branched paths for the second part of
the timestep.

In the remainder of the article, we consider the SDE \cref{eq:sde-scalar} in
\(d\)-dimensions for \(t \in [0,1]\):
\begin{equation}\label{eq:sde}
  \D \X{t} = a\p{\X{t}, t}\D t + \sigma\p{\X{t}, t}\D W_{t},
\end{equation}
where \(a: \rset^{d} \times [0, 1] \to \rset^{d}\) and \(\sigma: \rset^{d} \times [0, 1] \to
\rset^{d} \times \rset^{d'}\) are Borel-measurable functions and \(W\) is a
\(d'\)-dimensional Wiener process and denote its natural filtration by
\(\p{\mathcal F_{t}}_{0 \leq t \leq 1}\). We will again refer to a corresponding
sequence of approximations \(\br{\X[\ell]{t}}, {\ell=0,1,\ldots}\) using uniform
timesteps \(h_\ell=h_{0} M^{-\ell}\) for some \(h_{0} \in \rset_{+}\) and
\(M\in\zset_+\). We assume that \(a\) and \(\sigma\) satisfy at least the necessary
conditions (measurability, linear growth and global Lipschitzness in \(x\))
for existence and uniqueness of \(\X{t}\) in the strong sense
\cite{kloden:numsde}. Our goal is to estimate \(\prob{X_{1} \in S} = \E{\I{\X{1}
    \in S}}\) for some closed set \(S \subset \rset^{d}\) with boundary \(\partial S \eqdef
K\).

The article is organized as follows; see \cref{fig:outline} for an outline of
the assumptions/analysis carried out in the current work. In
\cref{sec:method}, we present the new branching estimator for a given
underlying estimator \(\Delta P_{\ell}\). We also bound the work and variance of the
branching estimator in \cref{lem:estimator-workvar} under the main
\cref{ass:est-main} on the underlying estimator \(\Delta P_{\ell}\) and show the
improved computational complexity of MLMC when using the branching estimator
in \cref{thm:mlmc-complexity-optimal-eta}. In \cref{sec:strong}, we consider
\(\Delta P_{\ell} \equiv \I{\X[\ell]{1}\in S} - \I{\X[\ell-1]{1}\in S}\), and prove that under
\cref{ass:cond-density}, ``strong'' approximations such as the Euler-Maruyama
or Milstein numerical schemes satisfy \cref{ass:est-main}. We conclude
\cref{sec:strong} with a numerical verification of the results in that
section. In \cref{sec:antithetic}, we consider the antithetic estimator that
was proposed in \cite{giles:antithetic}. We again show in
\cref{thm:Zest-antithetic-rates} that an antithetic estimator under
\cref{ass:g-derv-bounds,ass:g-weak-conv} satisfies \cref{ass:est-main} and
conclude the section with a numerical verification of the presented
theoretical results. In \cref{sec:sde-bounds}, we consider elliptic SDEs,
i.e., SDEs whose coefficients are bounded and their diffusion coefficient is
elliptic. In \cref{thm:cond-density-sde} we prove that the solutions to such
SDEs satisfy \cref{ass:cond-density,ass:g-derv-bounds} under mild assumptions
on \(K\). Then in \cref{thm:cond-density-sde-gbm} we prove that exponentials
of those solutions also satisfy \cref{ass:cond-density} under different
assumptions on \(K\).

In what follows, for \(\ell \in \nset \equiv \br{0, 1, \ldots }\), we use the notation
\(A_{\ell} \lesssim B_{\ell}\) to denote \(A_{\ell} \leq c B_{\ell}\), and \(A_{\ell} \simeq B_{\ell}\) to
denote \(c' B_{\ell} \leq A_{\ell} \leq c B_{\ell}\) for some constant, deterministic
\(c,c'>0\) that are independent of the index, \(\ell\), the accuracy tolerance,
\(\varepsilon\), and other parameters which will be specified. For \(u \in
\br{-1,1}^{\ell}\), let \(\abs{u}_{0} \defeq \ell\) and, for \(\ell \geq 1\), let
\(\parent{u} \defeq \p{u_{1}, u_{2}, \ldots, u_{\ell-1}} \in \br{-1,1}^{\ell-1}\), using
the convention \(\br{-1,1}^{0} \defeq \br{\varnothing}\). For \(k \in \nset, l \in
\nset\), let \(C^{k,l}_{b}\) be the space of continuously differentiable
bounded functions \(\p{x, t} \mapsto f\p{x, t}\) for \(\p{x,t} \in \rset^{d} \times [0,
T]\) with uniformly bounded derivatives with respect to \(x\)
(resp. with respect to \(t\)) up to order \(k\) (resp. \(l\)). When \(f\) is a
vector- or a matrix-valued function, \(f \in C_{b}^{k,l}\) means that all
function components are in \(C_{b}^{k,l}\). In addition, all vector and matrix
norms are Euclidean \(\ell^{2}\) norms.

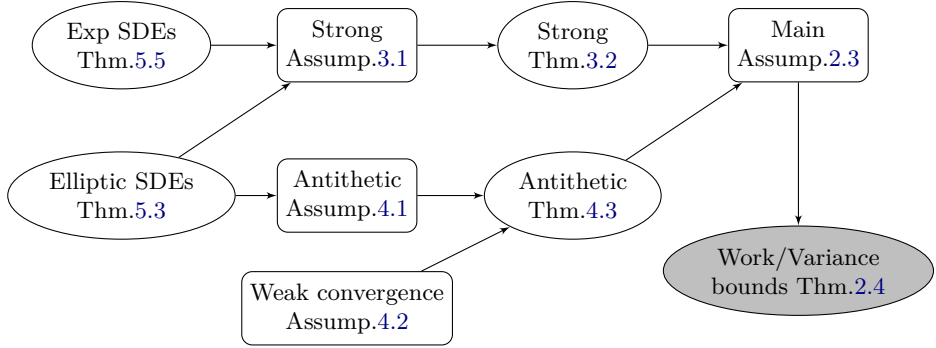
\begin{figure*}
  \centering

\tikzstyle{decision} = [diamond, draw, fill=none,
text width=4.5em, text badly centered, node distance=3cm, inner sep=0pt]%
\tikzstyle{assumption} = [rectangle, draw, fill=none,
text centered, rounded corners, minimum height=3em]%
\tikzstyle{line} = [draw, -latex']%
\tikzstyle{result} = [draw, ellipse,fill=none, node distance=3cm,
minimum height=2em]
\tikzstyle{final} = [draw, ellipse,fill=lightgray, node distance=3cm,
minimum height=2em]%

\begingroup
\Crefformat{lemma}{Lem.#2#1#3}%
\Crefformat{assumption}{Assump.#2#1#3}%
\Crefformat{theorem}{Thm.#2#1#3}%

\begin{tikzpicture}[node distance = 3cm,
  every text node part/.style={align=center},
  auto]

  \node [assumption] (main-ass) {Main\\
    \Cref{ass:est-main}};
  \node [final, below of=main-ass] (workvar) {Work/Variance\\
    bounds \Cref{lem:estimator-workvar}
  };
  \node [result, left of=main-ass] (strong-thm) {Strong\\
    \Cref{thm:strong-conv-var-rates}
  };
  \node [assumption, left of=strong-thm] (main-strong-ass) {Strong\\
    \Cref{ass:cond-density}
  };
  \node [result, left of=main-strong-ass] (exp-lemma) {Exp SDEs\\
    \Cref{thm:cond-density-sde-gbm}
  };
  \node [result, below of=exp-lemma, node distance=2cm] (elliptic-lemma) {Elliptic SDEs\\
    \Cref{thm:cond-density-sde}
  };
  \node [result, below of=strong-thm, node distance=2cm] (antithetic-lemma)
  {Antithetic\\
    \Cref{thm:Zest-antithetic-rates}
  };
  \node [assumption, left of=antithetic-lemma] (antithetic-ass) {Antithetic\\
    \Cref{ass:g-derv-bounds}
  };

  \node [assumption, below of=antithetic-ass, node distance=1.5cm] (weak-conv-ass) {Weak convergence\\
    \Cref{ass:g-weak-conv}
  };

  \path [line] (main-ass) -- (workvar);

  \path [line] (main-strong-ass) -- (strong-thm);

  \path [line] (strong-thm) -- (main-ass);

  \path [line] (exp-lemma) -- (main-strong-ass);
  \path [line] (elliptic-lemma) -- (main-strong-ass);

  \path [line] (antithetic-ass) -- (antithetic-lemma);
  \path [line] (antithetic-lemma) -- (main-ass);

  \path [line] (elliptic-lemma) -- (antithetic-ass);

  \path [line] (weak-conv-ass) -- (antithetic-lemma);

\end{tikzpicture}
\endgroup
   \caption{Outline of the analysis presented in the current work. Rectangles
    are assumptions while ellipses are lemmas and theorems. An arrow indicates
    implication under sufficient but not necessary conditions.}
  \label{fig:outline}
\end{figure*}
 \section{Branching Estimator}\label{sec:method}%

We begin by giving a formal definition to our branching estimator, using
  random discrete trees and branching processes \cite{gall:random-trees}.
  \begin{definition}[Branching Brownian Motion]\label{def:branching-BM}
    Given \(\tau_{0} \in \p{0,1}, \eta \in \rset_{+}\), let \(\tau_{\ell'} \defeq \tau_{0} \N^{-\eta
      \ell'}\) for all \(\ell'\in \nset\). Let \(\br{W^{u}}_{u \in \bigcup_{\ell'=0}^{\infty}
      \br{-1,1}^{\ell'}}\) be mutually independent Wiener processes and let
    \(\widetilde B_{t}^{\varnothing} \defeq W^{\varnothing}_{t}\) for \(t \in
    [0, 1]\). Then for any \(u \in \bigcup_{\ell'=1}^{\infty} \br{-1,1}^{\ell'}\), define the
    \(\nth{u}\) branch of a Branching Brownian Motion as
  \begin{equation*}
    \widetilde B_{t}^{u} \defeq
    \begin{cases}
      \widetilde B_{t}^{\parent{u}}
      & t \in [0, 1-\tau_{\abs{\parent{u}}_{0}}],\\
        \widetilde B_{1-\tau_{\abs{\parent{u}}_{0}}}^{\parent{u}} + W^{u}_{t - 1 + \tau_{\abs{\parent{u}}_{0}}}
      & t \in (1-\tau_{\abs{\parent{u}}_{0}}, 1].
    \end{cases}
  \end{equation*}
  \end{definition}

  \begin{definition}[Branching estimator]\label{def:est-main}%
    Given \(\ell \in \nset\), let \(\hat \ell \defeq
    \en*\lfloor\rfloor{\log_{\N}\p{\p{h_{\ell}/\tau_{0}}^{-1}}/\eta}, \) such that \(\tau_{\hat \ell-1} \simeq
    h_{\ell}\). Let \(\br{\overline X_{\ell,1}^{u}, \overline X_{\ell-1,1}^{u}}_{u \in
      \br{-1,1}^{\hat \ell}}\) be approximations of the SDE path in \cref{eq:sde}
    with timesteps of sizes \(\br{h_{\ell}, h_{\ell-1}}\), respectively, and for a
    given a Branching Brownian Motion, \(\br{\widetilde B^{u}}_{u \in
      \br{-1,1}^{\hat \ell}}\), as in \cref{def:branching-BM}, and let \({\Delta
      P_{\ell}^{u}} \defeq {\I{\X^{u}[\ell]{1} \in S} - \I{\X^{u}[\ell-1]{1} \in S}}\).
    Finally, define the branching estimator as
  \begin{equation}\label{eq:est-main}
    \Zest_{\ell} \defeq \frac{1}{\N^{\hat \ell}}
    \sum_{u \in \br{-1,1}^{\hat \ell}} \Delta P_{\ell}^{u}.
  \end{equation}
\end{definition}

See \cref{fig:branching-estimator} for an illustration of the path branching
involved in \(\Zest_{\ell}\). Note that a single branch of a Branching Brownian
motion is itself a Brownian Motion. Hence the distribution of \(\Delta P^{u}_{\ell}\)
for \(u \in \br{-1,1}^{\hat \ell}\) is independent of \(u\). We will refer to a
generic \(\Delta P_{\ell}\) and the filtration of its underlying Brownian motion,
\(\p{\mathcal F_{t}}_{0 \leq t \leq 1}\), when the dependence on \(u\) is not
relevant. We will also refer to the cost of \(\Zest_{\ell}\), which we define as
the total number of Brownian increments needed to compute \(\Zest_{\ell}\) based
on an Euler-Maruyama approximation, a Milstein approximation which does not
require simulations of L\'evy areas or a truncated Milstein scheme as
described in \cref{sec:antithetic}. We denote the cost of \(\Zest_{\ell}\) by
\(\textnormal{Work}\p*{\Zest_{\ell}}\). Note that \( \E{\Zest_{\ell}} = \E{\Delta P_{\ell}}
= \E{\I{\X[\ell]{1} \in S} - \I{\X[\ell-1]{1} \in S}} \) and hence we can use
\(\Zest_{\ell}\) instead of \(\Delta P_{\ell}\) in the MLMC setup \cref{eq:MLMC}.
However, under some conditions, we will see in this section that the cost of
\(\Zest_{\ell}\) is not significantly larger than the cost of \(\Delta P_{\ell}\) while
the variance is significantly smaller, leading to a better computational
complexity of MLMC. In particular, even though the cost of each sample of \(\Delta
P_{\ell}\) is \(h_{\ell}^{-1}\), many of the samples share an underlying Brownian
path up to some branching point and hence the total cost for the
\(2^{\hat{\ell}}\) samples is greatly reduced as we will show in
\cref{lem:estimator-workvar}. We make the following general assumptions which
we will relate, in \cref{sec:strong,sec:antithetic}, to assumptions on the SDE
\cref{eq:sde} and \(K \equiv \partial S\):

\begin{assumption}[Estimator assumptions]\label{ass:est-main}
  There exist \(\crossbeta \geq \diagbeta >0, p \geq 0, \tau_{0} \in \p{0,1}\)
  such that
  \begin{subequations}
    \label{eq:est-assumpt}
    \begin{align}
      \label{eq:est-assumpt-diag}\E{ \p{\Delta P_{\ell}}^{2} }
      &\lesssim h_{\ell}^{\diagbeta},\\
      \label{eq:est-assumpt-cross} \textrm{and}\qquad \E*{ \p*{\E{\Delta P_{\ell} \given
      \mathcal{F}_{1-\tau}}}^{2}}
      & \lesssim h_{\ell}^{\crossbeta} / \tau^{p},
    \end{align}
    for all \(\ell \in \nset\) and \(\tau \in [h_{\ell},\tau_{0}]\).
  \end{subequations}
  We also assume that the estimator satisfies the following bias constraint
  for some \(\alpha \geq \diagbeta/2\)
\begin{equation}\label{ass:bias}
  \abs*{\E{\I{\X[\ell]{1} \in S} - \I{\X{1} \in S}}} \lesssim h_{\ell}^{\alpha}.
\end{equation}
\end{assumption}
Typical approximate values for \(\crossbeta,\diagbeta\) and \(p\) are in
\cref{tbl:assumption-values}. The assumption \eqref{ass:bias} is shown to hold
for the Euler-Maruyama scheme for \(\alpha=1\) in \cite[Theorem
2.5]{gobet:density}, when the SDE \cref{eq:sde} is uniformly elliptic, \(a, \sigma
\in C_{b}^{3,1}\) and \(\frac{\partial \sigma}{\partial t} \in C^{1,0}_{b}\).

\begin{table}
  \centering
  \begin{tabular}{c|ccc}
    &\(\diagbeta\)&\(\crossbeta\)&\(p\)\\
    \hline
    Euler-Maruyama & 1/2 & 1 & 1/2\\
    Milstein & 1 & 2 & 1/2 \\
    Antithetic Milstein & 1/2 & 2 & 3/2 \\

  \end{tabular}
  \caption{Limiting values for \(\diagbeta, \crossbeta\) and \(p\) in
    \cref{ass:est-main} for estimators \(\Delta P_{\ell}\) based on Euler-Maruyama,
    Milstein or antithetic Milstein discretizations. These are proved later in
    \cref{sec:strong,sec:antithetic}.}
  \label{tbl:assumption-values}
\end{table}

\begin{theorem}[Work and variance]\label{lem:estimator-workvar}
  For any \(\ell \in \nset\), the estimator \(\Zest_{\ell}\) in \cref{def:est-main}
  satisfies %
  \[
    \textnormal{Work}\p{\Zest_{\ell}} \lesssim \begin{cases}
       {h_{\ell}^{-\maxp{1, 1/\eta}}} & \eta \neq 1,\\
      {\ell \, h_{\ell}^{-1} } & \eta = 1,
    \end{cases}
  \]
  and, under \cref{ass:est-main},
  \[
    \var{\Zest_{\ell}} \lesssim \:
    \begin{cases}
{h_{\ell}^{\diagbeta+1/\eta} + h_{\ell}^{\crossbeta -\maxp{0,p-1/\eta}}} & \eta p
      \neq 1,\\
      {h_{\ell}^{\diagbeta+1/\eta} + \ell \, h_{\ell}^{\crossbeta} } & \eta p = 1.\\
    \end{cases}
  \]
\end{theorem}
\begin{proof}
  The proof is a slight generalization of the argument in the Introduction.

  \textbf{Work:} {If the branching points coincide with the discretization
    grid specified by \(h_{\ell}\), each path on the time interval
    \([1{-}\tau_{\ell'-1},1{-} \tau_{\ell'}]\) contains \(h_\ell^{-1}(\tau_{\ell'-1} - \tau_{\ell'})\)
    fine timesteps; if they do not coincide then at worst each path segment
    requires \(\lfloor h_\ell^{-1}(\tau_{\ell'-1} - \tau_{\ell'})\rfloor + 2\) Brownian increments.
    Accordingly, the total work is bounded by}
  \[
    \begin{aligned}
      &h_{\ell}^{-1} \p*{\p{1-\tau_{0}} + \sum_{\ell'=1}^{\hat \ell-1} \N^{\ell'} \p{\tau_{\ell'-1} -
        \tau_{\ell'}} + \N^{\hat \ell} \tau_{\hat \ell-1}
        } + 2 \sum_{\ell'=0}^{\hat \ell-1} 2^{\ell'}\\
      &= h_{\ell}^{-1} \p*{\p{1-\tau_{0}} +
        \tau_{0}\,\sum_{\ell'=1}^{\hat \ell-1} \N^{\ell'} \p{\N^{-\eta \, \p{\ell'-1}} - \N^{-\eta \ell'}}
        + \tau_{0}\, \N^{\hat \ell} \, \N^{-\eta \p{\hat \ell-1}}
        } + 2 \sum_{\ell'=0}^{\hat \ell-1} 2^{\ell'}\\
      &\lesssim
      \begin{cases}
        {h_{\ell}^{-1} \N^{\,\maxp{1-\eta,0}\hat \ell}} + 2^{\hat \ell}& \eta \neq 1,\\
        {\hat \ell \, h_{\ell}^{-1}} + 2^{\hat \ell}& \eta = 1.
      \end{cases}%
    \end{aligned}
  \]
  and noting that \(\N^{\hat \ell} \lesssim h_{\ell}^{-1/\eta}\) and \(\hat \ell \lesssim \ell\) we obtain the desired result.

  \vskip 0.3cm \noindent \textbf{Variance: }
  \[
      \var*{\Zest_{\ell}}
      = \frac{1}{\N^{2\hat \ell}} \sum_{u \in \br{-1,1}^{\hat \ell}} \E*{\p*{\Delta
                         P_{\ell}^{u}}^{2}}
                     + \frac{1}{\N^{2\hat \ell}}
                     \sum_{u \in \br{-1,1}^{\hat \ell}}
                     \sum_{\substack{v \in \br{-1,1}^{\hat \ell} \\ u \neq v}}
                     \E{\Delta P_{\ell}^{u} \: \Delta P_{\ell}^{v}}.
  \]
  Using \cref{eq:est-assumpt-diag}, we have that
  \[
    {\N^{-2\hat \ell}} \sum_{u \in \br{-1,1}^{\hat \ell}}
    \E*{\p*{\Delta P_{\ell}^u}^{2}}
      ~=~      {\N^{-\hat \ell}} \E*{\p*{\Delta P_{\ell}}^{2}}
      ~\lesssim~ \N^{-\hat \ell} h_{\ell}^{\diagbeta}
      ~\lesssim~ h_{\ell}^{\diagbeta+1/\eta}.
    \]
  Let \(\abs{u \wedge v}_{0} = \maxp{\ell' \leq \minp{\abs{u}_{0}, \abs{v}_{0}} \; : \;
    u_{i} = v_{i} \text{ for all } i \in \br{1, \ldots, \ell'} }\), and note that two
  payoff differences \(\Delta P_{\ell}^{u}\) and \(\Delta P_{\ell}^{v}\) share a path up to
  time \(1{-}\tau_{\abs{u \wedge v}_{0}}\) and then the paths are independent and
  identically distributed after that. Hence, using
  \cref{eq:est-assumpt-cross}, we have
  \[
    \E{\Delta P_{\ell}^{u} \Delta P_{\ell}^{v}}
    = \E*{\p*{\E{\Delta P_{\ell} \given \mathcal F_{1-\tau_{\abs{u \wedge v}_{0}}}}}^{2}}%
                      \lesssim h_{\ell}^{\crossbeta} \p*{\tau_{\abs{u \wedge v}_{0}}}^{-p}.
  \]
  We can then evaluate the double sum as
  \[
    \begin{aligned}
      \frac{1}{\N^{2\hat \ell}}
        \sum_{u \in \br{-1,1}^{\hat \ell}}
        \sum_{\substack{v \in \br{-1,1}^{\hat \ell} \\ u \neq v}}
      \E{\Delta P_{\ell}^{u} \: \Delta P_{\ell}^{v}}
      &\lesssim
        \frac{h_{\ell}^{\crossbeta}}{\N^{2\hat \ell}}
        \sum_{u \in \br{-1,1}^{\hat \ell}}
        \sum_{\substack{v \in \br{-1,1}^{\hat \ell} \\ u \neq v}}
      \tau_{\abs{u \wedge v}_{0}}^{-p} \\
      & = \frac{h_{\ell}^{\crossbeta}}{\N^{2\hat \ell}} \sum_{\ell'=0}^{\hat \ell-1} \N^{2 \hat \ell-\ell'-1} \tau_{\ell'}^{-p} \\
      &\lesssim %
      h_{\ell}^{\crossbeta} \, \p*{ \sum_{\ell'=0}^{\hat \ell-1} \N^{\p*{\eta p-1}\ell'}}. \\
   \end{aligned}
  \]
 \detailed{Here, we evaluated the double sum by noting that \(\abs{u \wedge v}_{0}
   \in \br{0,\ldots, \hat \ell-1}\) when \(u,v \in \br{-1,1}^{\hat \ell}\) and \(u \neq v\).
   Then summing over these possible values, indexing by \(\ell'\), we count the
   number of possibilities of having \(u=v\) up to the \(\ell'\) index (this is
   \(2^{\ell'}\)) the \(\ell'\) being not equal (this is \(2\)), and then the rest
   of the indices \(\p{\hat \ell - \ell-1}\) being arbitrary (this is \(2^{\hat \ell +
     \ell' - 1}\) for \(u\) and similar for \(v\)). The result is
   \[
     \begin{aligned}
       \sum_{u \in \br{-1,1}^{\hat \ell}}
       \sum_{\substack{v \in \br{-1,1}^{\hat \ell} \\ u \neq v}}
       \p*{\tau_{\abs{u \wedge v}_{0}}}^{-p}
       &=
         \sum_{\ell'=0}^{\hat \ell-1} 2^{\ell' + 1} {2^{\hat \ell - \ell'-1}} {2^{\hat \ell - \ell'-1}} \tau_{\ell'}^{-p}
         = \sum_{\ell'=0}^{\hat \ell-1} {2^{2 \hat \ell - \ell'-1}} \tau_{\ell'}^{-p}.
     \end{aligned}
  \]
 }
  Bounding the sum based on the sign of \(\eta p{-}1\) concludes the proof.
  \begin{details}
    \[
      \begin{aligned}
        \sum_{\ell'=1}^{\hat \ell} \N^{\p*{\eta p-1}\ell'} &\lesssim
        \begin{cases}
          \Order{1} & \eta p < 1 \\
          \hat \ell & \eta p = 1 \\
          \N^{\p{\eta p - 1}\hat \ell} & \eta p > 1 \\
        \end{cases}\\
        &\lesssim
        \begin{cases}
          \Order{1} & \eta p < 1 \\
          \frac{1}{\eta}\log_2\p{h_{\ell}^{-1}} & \eta p = 1 \\
          h_{\ell}^{-p + 1/\eta} & \eta p > 1. \\
        \end{cases}
      \end{aligned}
    \]
  \end{details}
\end{proof}

\begin{remark}[Optimal \texorpdfstring{\(\eta\)}{eta}]\label{rem:optimal-eta}
  \cref{lem:estimator-workvar} shows that the choice of \(\eta\) in the estimator
  \(\mathcal P_{\ell}\) in \cref{def:est-main} compared to \(p\) is crucial to
  improving the variance convergence rate of the new estimator compared to
  \(\diagbeta\), the variance convergence rate of the simple estimator \(\Delta
  P\), without substantially increasing the cost of the new estimator. We can
  optimize the value of \(\eta\), by noting that work has the term \(
  \sum_{\ell'=0}^{\hat \ell-1} \N^{\ell'} \tau_{\ell'}, \) and the variance has the term \(
  \sum_{\ell'=0}^{\hat \ell-1} \N^{-\ell'} \tau_{\ell'}^{-p}. \) Hence, minimizing the work
  subject to a constrained variance yields the optimal value of \(\tau_{\ell'} \propto
  \N^{-\frac{2}{p+1} \ell'}\) and the optimal value of \(\eta\) is
  \(\frac{2}{p+1}\).
\end{remark}

\begin{remark}[Number of branches] In the estimator outlined above,
  \(\spellout{\N}\) branches are created at each branching point, \(1-\tau_{\ell'}\)
  for \(\ell'=0,\ldots, \hat\ell-1\). The method and analysis can be easily extended to
  allow for a different number of branches at every branching point. However,
  after adjusting $\eta$ to keep the total work constant this would not improve
  the variance convergence rate or the subsequent computational complexities
  that we later derive.
\end{remark}

\begin{corollary}\label{thm:mlmc-complexity-optimal-eta}
  Under \cref{ass:est-main}, an MLMC estimator with an MSE \(\varepsilon^2\) based on
  the branching estimator \(\Zest_{\ell}\) in \cref{def:est-main} with \(p <
  \maxp{2\alpha,1}\) and \(\eta=\frac{2}{p+1}\) and \(h_{\ell} = h_{0} M^{-\ell}\) for \(M \in
  \nset_{+}\) has the following computational complexity \[
      \begin{cases}
        \Order{\varepsilon^{-2-\maxp{0, \p*{1-p}/2 -\diagbeta}/\alpha} \, \abs{\log \varepsilon}^{2\I{2 \diagbeta=1-p}}}
        & p < \minp{1, 2\p*{\crossbeta-\diagbeta} - 1}, \\
        \Order{\varepsilon^{-2-\maxp{0, 1-\crossbeta}/\alpha} \, \abs{\log \varepsilon}^{2\I{\crossbeta=1}}}
        & 2\p*{\crossbeta-\diagbeta} - 1 \leq p < 1, \\
        \Order{\varepsilon^{-2-\maxp{0, 1-\crossbeta}/\alpha} \: \abs{\log \varepsilon}^{2\I{\crossbeta = 1} + 2\I{\crossbeta \leq 1}}}
        & p = 1,\\
        \Order{\varepsilon^{-2-\maxp{0, p-\crossbeta}/\alpha} \: \abs{\log \varepsilon}^{2\I{\crossbeta=p}}}
        & p > 1.
      \end{cases}
    \]
\end{corollary}
Since \(\crossbeta \geq \diagbeta\) as in \cref{ass:est-main}, the computational
complexity of an MLMC estimator based on \(\Zest_{\ell}\) is lower than that of
an MLMC estimator based on \(\Delta P_{\ell}\), the latter being
\(\Order{\varepsilon^{-2-\maxp{1-\diagbeta, 0}/\alpha} \abs{\log\p{\varepsilon}}^{2
    \I{\diagbeta=1}}}\), whenever \(p < 1+\crossbeta-\diagbeta\).
\begin{proof}
  Recall that the MLMC estimator is defined as
  \[
    \sum_{\ell=0}^{L} \frac{1}{N_{\ell}} \sum_{m=1}^{N_{\ell}} \Zest_{\ell}^{\p{m}},
  \]
  where \(\Zest_{\ell}^{\p{m}}\) are independent samples of \(\Zest_{\ell}\). For
  \(V_{\ell} \defeq \var{\Zest_{\ell}}\) and \(W_{\ell} \defeq
  \textnormal{Work}\p{\Zest_{\ell}}\) the total work of the MLMC estimator for a
  fixed \(L\) is \( \sum_{\ell=0}^{L} W_{\ell} N_{\ell} \) while the total variance is \(
  \sum_{\ell=0}^{L} V_{\ell} / N_{\ell} \). %
  Minimizing the work while constraining the variance to be less than
  \(\varepsilon^{2}/2\) leads to the optimal choice of number of samples on level \(\ell\)
  \cite{giles:acta},
  \[N_{\ell} = \en*\lceil \rceil{2 \varepsilon^{-2} \p{{V_{\ell}}/{W_{\ell}}}^{1/2} \: \p*{\sum_{\ell=0}^{L}
        \p{W_{\ell} V_{\ell}}^{1/2}}}, \]
  so that the total work is bounded by
  \[
    2 \, \varepsilon^{-2}\p*{\sum_{\ell=0}^{L} \p{W_{\ell} V_{\ell}}^{1/2}}^{2} + \sum_{\ell=0}^{L} W_{\ell} \ .
  \]
  The sum \(\sum_{\ell=0}^{L} W_{\ell} \lesssim M^{L\p{1+\maxp{p-1,0}/2}} L^{\I{p=1}}\) is
  dominated by the first term when setting \(L \propto \frac{1}{\alpha}\log{\varepsilon^{-1}}\),
  such that the bias in \cref{ass:bias} is \(\Order{\tol}\), and when \(p <
  \maxp{2\alpha, 1}\) and for sufficiently small \(\varepsilon\). By
  \cref{lem:estimator-workvar}, for \(\eta=\frac{2}{p+1}\), we have
  \begin{details}
    \[
      W_{\ell} \lesssim \begin{cases}
                {h_{\ell}^{-1-\maxp{p-1,0}/2}} & p \neq 1\\
                {\ell \, h_{\ell}^{-1} } & p=1\\
              \end{cases}\]
    while
    \[
      V_{\ell} \lesssim \
      \begin{cases}
        {h_{\ell}^{\diagbeta+\p{p+1}/2} +
        h_{\ell}^{\crossbeta-\maxp{p-1, 0}/2}} & p \neq 1\\
        {h_{\ell}^{\diagbeta+1} + \ell \, h_{\ell}^{\crossbeta} } &p = 1
      \end{cases}
    \]
    Hence,
  \end{details}
  \[
    W_{\ell}V_{\ell} \lesssim \begin{cases}
                   {h_{\ell}^{\diagbeta
                   +\p*{p-1}/2
                   } + h_{\ell}^
                   {\crossbeta
                   -1}
                   }
                   & p < 1,\\
                   \ell {h_{\ell}^{\diagbeta} + \ell^{2} \, h_{\ell}^{\crossbeta-1} }
                   & p=1,%
                   \\
                   {h_{\ell}^{\diagbeta} + h_{\ell}^{\crossbeta -p}}
                   & p > 1.\\
    \end{cases}
  \]
  Evaluating the sum for \(h_L \approx \eps^{1/\alpha}\), which implies \(L \propto \frac{1}{\alpha}
  \log \eps^{-1}\), yields the result.
  \begin{details}
    Using that \(h_{L} < h_{0}\) and \(\diagbeta>0\),
    \[
    \p*{\sum_{\ell=0}^{L} \sqrt{W_{\ell} V_{\ell}}}^{2} \lesssim
    \begin{cases}
      L^{2\I{2 \diagbeta=1-p}} \, h_{L}^{\diagbeta+\p*{p-1}/2} +
      L^{2\I{\crossbeta=1}} \, h_{L}^{\crossbeta-1} & p < 1 \\
      L^{2 \I{\crossbeta=1} + 2 \I{p=1, \crossbeta \leq 1}} \, h_{L}^{\crossbeta-1} & p = 1\\
      L^{2\I{\crossbeta=p}} h_{L}^{\crossbeta-p} & p > 1
    \end{cases}
  \]
  where the hidden constant is independent of \(L\). Setting \(L \propto -\log\p*{\varepsilon}/\alpha\)
  \[
    \p*{\sum_{\ell=0}^{L} \sqrt{W_{\ell} V_{\ell}}}^{2} \lesssim
    \begin{cases}
      \varepsilon^{-\maxp{0, 1-\crossbeta}} \, \abs{\log \varepsilon}^{2\I{\crossbeta=1}}
      + \varepsilon^{-\maxp{0, \p*{1-p}/2 -\diagbeta}} \, \abs{\log \varepsilon}^{2\I{2 \diagbeta=1-p}}  \,& p < 1 \\
      \varepsilon^{-\maxp{0, p-\crossbeta}} \: \abs{\log \varepsilon}^{2\I{p=1} + 2\I{\crossbeta=p}}  & p \geq 1
    \end{cases}
  \]

    The result is
    \[
      \begin{cases}
        \varepsilon^{-\maxp{0, \p*{1-p}/2 -\diagbeta}} \, \abs{\log \varepsilon}^{2\I{2 \diagbeta=1-p}}
        \,& p < \minp{1, 2\p*{\crossbeta-\diagbeta} - 1} \\
        \varepsilon^{-\maxp{0, 1-\crossbeta}} \, \abs{\log \varepsilon}^{2\I{\crossbeta=1}}
        \,& 2\p*{\crossbeta-\diagbeta} - 1 \leq p < 1 \\
        \varepsilon^{-\maxp{0, 1-\crossbeta}} \: \abs{\log \varepsilon}^{2\I{\crossbeta=1}  + 2 \I{\crossbeta \leq 1}}  & p = 1 \\
        \varepsilon^{-\maxp{0, p-\crossbeta}} \: \abs{\log \varepsilon}^{2\I{\crossbeta=p}}  & p > 1
      \end{cases}
    \]
    Note: When \(p=1\), we get an extra 2 in the log only when \(\crossbeta=1\)
  \end{details}
\end{proof}

\begin{remark}\label{rem:mlmc-complexity-eta=1}
  The simple case \(\eta = 1\) is optimal only when \(p=1\). In other cases, the
  computational complexity increases slightly compared to
  \cref{thm:mlmc-complexity-optimal-eta}:
  \[
    \begin{cases}
      \Order{\varepsilon^{-2 - \maxp{1-\crossbeta, 0}/\alpha} \, \abs{\log\p{\varepsilon}}^{2 \I{\crossbeta=1} + \I{\crossbeta \leq 1}} }
      & p < 1, \\
      \Order{\varepsilon^{-2 - \maxp{1-\crossbeta, 0}/\alpha} \, \abs{\log\p{\varepsilon}}^{2 \I{\crossbeta=1} + 2 \I{\crossbeta \leq 1}}}
      & p = 1, \\
      \Order{\varepsilon^{-2 - \maxp{p-\crossbeta, 0}/\alpha} \, \abs{\log\p{\varepsilon}}^{2 \I{\crossbeta=p} +  \I{\crossbeta \leq p}} }
      &  p > 1.
    \end{cases}
  \]
\end{remark}
\begin{details}
  \begin{proof}

    \( \textnormal{Work}\p{\Zest_{\ell}} \lesssim {\ell \, h_{\ell}^{-1} } \) and
    \[
      \var{\Zest_{\ell}} \lesssim \:
      \begin{cases}
        {h_{\ell}^{\diagbeta+1} + h_{\ell}^{\crossbeta -\maxp{0,p-1}}} & p \neq 1\\
        {h_{\ell}^{\diagbeta+1} + \ell \, h_{\ell}^{\crossbeta} } & p = 1\\
      \end{cases}
    \]
    Hence,
    \[
      W_{\ell}V_{\ell} \lesssim
      \begin{cases}
        \ell {h_{\ell}^{\diagbeta} + \ell h_{\ell}^{\crossbeta-1}} & p < 1\\
        \ell h_{\ell}^{\diagbeta} + \ell^{2} \, h_{\ell}^{\crossbeta-1} & p = 1\\
        \ell {h_{\ell}^{\diagbeta} + \ell h_{\ell}^{\crossbeta -p}} & p > 1\\
      \end{cases}.
    \]
    which leads to
    \[
      \p*{\sum_{\ell=0}^{L} \sqrt{W_{\ell}V_{\ell}}}^{2} \lesssim
      \begin{cases}
        {\varepsilon^{- \maxp{1-\crossbeta, 0}/\alpha} \, \abs{\log\p{\varepsilon}}^{2 \I{\crossbeta=1} + \I{\crossbeta \leq 1}} }
        & p < 1 \\
        {\varepsilon^{- \maxp{1-\crossbeta, 0}/\alpha} \, \abs{\log\p{\varepsilon}}^{2 \I{\crossbeta=1} + 2 \I{\crossbeta \leq 1}}}
        & p = 1 \\
        {\varepsilon^{- \maxp{p-\crossbeta, 0}/\alpha} \, \abs{\log\p{\varepsilon}}^{2 \I{\crossbeta=p} +  \I{\crossbeta \leq p}} }
        &  p > 1
      \end{cases}
    \]
  \end{proof}
\end{details}
 \section{Strong Analysis}\label{sec:strong}%
\def\tradestlabel{Without branching}
\def\newestlabel{With branching}

In this section, we consider
\begin{equation}\label{eq:strong-estimator}
  \Delta P_{\ell} \equiv \I{\X[\ell]{1} \in S} - \I{\X[\ell-1]{1}\in S},
\end{equation}
and make well-motivated assumptions on the solution of the SDE in
\cref{eq:sde} and its numerical approximation \(\br{\X[\ell]{t}}_{0 \le t \le 1}\),
and then show that our main \cref{ass:est-main} follows from these. We then
present the results of several numerical
experiments. %
For a set \(J \subset \rset^{d}\), define the distance of \(x\) to \(J\) as
\begin{equation}\label{eq:dist-def}
  \dist[J]{x} \defeq \inf_{y\in{J}} \norm{y-x}.
\end{equation}

\begin{assumption}\label{ass:cond-density}
  Assume that for some \(\delta_{0}>0\) and all \(0<\delta\le\delta_{0}\) and \(\tau \in (0, 1]\),
  there is a constant \(C\) independent of \(\delta, \tau\) and \(\mathcal F_{1-\tau}\)
  such that
  \begin{equation}\label{eq:ass-cond-density}
    \E*{\p[\big]{\prob{\dist{\X{1}} \leq \delta \given \mathcal F_{1-\tau}}}^{2}} \leq C \: \frac{\delta^{2}}{\tau^{1/2}}.
  \end{equation}
\end{assumption}

\cref{ass:cond-density} is fairly generic and depends on the set \(K \equiv \partial S\)
and the conditional density of \(\X{1}\) given the filtration \(\mathcal
F_{1-\tau}\). It is motivated by the case when \(\p{X_{t}}_{t \geq 0}\) is a
\(d\)-dimensional Wiener process and we prove it in
\cref{thm:cond-density-sde} (and \cref{thm:cond-density-sde-gbm}) for
solutions (and exponentials of solutions) to uniformly elliptic SDEs.

\begin{theorem}\label{thm:strong-conv-var-rates}
  Let \cref{ass:cond-density} hold and assume that there is \(q \geq 1\) and
  \(\beta>0\) such that
  \begin{equation}\label{ass:strong-conv}
    \E*{\norm{\X{1} - \X[\ell]{1}}^{q}}^{1/q} \lesssim h_{\ell}^{\beta/2}.
  \end{equation}
  Then for \(\Delta P_{\ell}\) in \cref{eq:strong-estimator} and all \(\tau \in \p{0,1}\),
  \begin{subequations}
    \begin{align}
      \label{eq:strong-diag}\E{ \p{\Delta P_{\ell}}^{2} } &\lesssim h_{\ell}^{\beta\p{1-1/\p{q+1}}/2}\\
      \label{eq:strong-cross} \textrm{and}\qquad \E*{ \p*{\E{\Delta P_{\ell} \given
      \mathcal F_{1-\tau}}}^{2}} & \lesssim h_{\ell}^{\beta \p{1-2/\p{q+2}}} / \tau^{1/2}.
    \end{align}
  \end{subequations}
\end{theorem}
This theorem shows that \cref{ass:est-main} is satisfied with
\(\crossbeta {=} \beta \p{1{-}2/\p{q{+}2}}
\),\\ \(\diagbeta{=} \beta\p{1{-}1/\p{q{+}1}}/2
\) and \(p{=}1/2\). Under {standard conditions} on the coefficients of the
SDE, assumption \cref{ass:strong-conv} is satisfied for the Euler-Maruyama and
Milstein numerical schemes for \(\beta=1\) and \(\beta=2\), respectively, and any \(q
\geq 1\) \cite{kloden:numsde}.
\begin{details}
  Note that \(2\p{\crossbeta-\diagbeta} = \beta\p{1- \frac{4}{q+2} +
    \frac{1}{q+1}}\). Hence \(\crossbeta > \diagbeta\) whenever \({1-
    \frac{4}{q+2} + \frac{1}{q+1}} > 0\), i.e., when \(q \geq 1\).
\end{details}
\begin{proof}
  The proof of \eqref{eq:strong-diag} is similar to the proof in
  \cite{giles:discont-payoff}, and is included here for completeness. We start
  by noting that
  \[
      \E{\p{\Delta P_{\ell}}^{2}}
      \leq \E*{\abs*{ \I{\X{1} \in S} - \I{\X[\ell-1]{1} \in S}}}
    + \E*{\abs*{ \I{\X{1} \in S} - \I{\X[\ell]{1} \in S}}},
  \]
  and for any \(\delta>0\)
  \[
    \begin{aligned}
      \abs{\I{\X{1} \in S} - \I{\X[\ell]{1} \in S}}
      &\le \I{\norm{\X{1} -
        \X[\ell]{1}} > \dist{\X{1}} }\\
      &\le \I{\norm{\X{1} - \X[\ell]{1}} > \dist{\X{1}} } \I{\dist{\X{1}}
        \leq
        \delta} + \I{\norm{\X{1} - \X[\ell]{1}} > \dist{\X{1}}> \delta}\\
      &\le \I{\dist{\X{1}} \leq \delta} %
        + \I{\norm{\X{1} - \X[\ell]{1}} > \delta}\\
      &\le \I{\dist{\X{1}} \leq \delta} + %
        \p{h_{\ell}^{-\beta/2}\delta}^{-q} \p*{h_{\ell}^{-\beta/2}\norm{\X{1}-\X[\ell]{1}}}^{q}.
    \end{aligned}
  \]
  Hence
  \[
    \begin{aligned}
      \E*{\abs*{ \I{\X{1} \in S} - \I{\X[\ell]{1} \in S}}}
      &\leq \prob{\dist{\X{1}} \leq
        \delta}
        + \p{h_{\ell}^{\beta/2}\delta}^{-q} \E*{\p*{h_{\ell}^{\beta/2}\norm{\X{1}-\X[\ell]{1}}}^{q}}\\
      &\leq C\, \delta + \p{h_{\ell}^{-\beta/2}\delta}^{-q}
        \E*{\p*{h_{\ell}^{-\beta/2}\norm{\X{1}-\X[\ell]{1}}}^{q}},
    \end{aligned}
  \]
  where we used~\cref{ass:cond-density} for \(\tau=1\) and the fact that
  \(\E*{\p*{h_{\ell}^{-\beta/2}\norm{\X{1}-\X[\ell]{1}}}^{q}}\) is bounded by
  \cref{ass:strong-conv}. A similar bound is obtained for \(\E*{\abs*{
      \I{\X{1} \in S} - \I{\X[\ell-1]{1} \in S}}}\) and then we select \(\delta =
  h_{\ell}^{\frac{q}{q+1}\beta/2}\) to obtain \eqref{eq:strong-diag}.

  Proving \eqref{eq:strong-cross}
  follows the same steps by similarly noting that
  \begin{multline*}
      \E*{ \p*{\E{\Delta P_{\ell} \given \mathcal F_{1-\tau}}}^{2}}
      \leq
      2 \, \E*{ \p*{\E*{ \abs*{ \I{\X{1} \in S} - \I{\X[\ell]{1} \in S}}\given
            \mathcal F_{1-\tau}}}^{2}} \\
      + 2 \, \E*{ \p*{\E*{ \abs*{ \I{\X{1} \in S} - \I{\X[\ell-1]{1} \in S}}\given
            \mathcal F_{1-\tau}}}^{2}},
  \end{multline*}
  and for any \(\delta>0\)
  \[
      \abs{\I{\X{1} \in S} - \I{\X[\ell]{1} \in S}}
      \le \I{\dist{\X{1}} \leq \delta}
      + \p{h_{\ell}^{-\beta/2}\delta}^{-q/2} \p*{h_{\ell}^{-\beta/2}\norm{\X{1}-\X[\ell]{1}}}^{q/2},
  \]
  so that
  \begin{multline*}
      \E*{ \p*{\E{\I{\X{1} \in S} - \I{\X[\ell]{1} \in S}\given
        \mathcal F_{1-\tau}}}^{2}}
      \leq
      2\, \E{\p{\E{\I{\dist{\X{1}} \leq \delta} \given \mathcal F_{1-\tau} }}^{2}}\\
      + 2\,\p{h_{\ell}^{-\beta}\delta^{2}}^{-q/2} \, \E*{\p*{h_{\ell}^{-\beta/2}
        \norm{\X{1}-\X[\ell]{1}}}^{q} }.
  \end{multline*}
  Here \(\E*{\p*{h_{\ell}^{-\beta/2}\norm{\X{1}-\X[\ell]{1}}}^{q}}\) is bounded by
  \cref{ass:strong-conv}, and using \cref{ass:cond-density} we have
  \[
    \begin{aligned}
      \begin{aligned}
        \E*{ \p*{\E{\I{\X{1} \in S} - \I{\X[\ell]{1} \in S}\given \mathcal
        F_{1-\tau}}}^{2}}
        &\lesssim \delta^{2}/\tau^{1/2} + \p{h_{\ell}^{-\beta}\delta^{2}}^{-q/2}\\
        &\lesssim \delta^{2}/\tau^{1/2} + \p{h_{\ell}^{-\beta}\delta^{2}}^{-q/2} /\tau^{1/2},
      \end{aligned}
    \end{aligned}
  \]
  for \(\tau {<} 1\) and we choose \( \delta^{2} {=} h_{\ell}^{\frac{q}{q+2} \beta} \) to
  obtain the result.
\end{proof}

As a direct implication of
\cref{thm:mlmc-complexity-optimal-eta,thm:strong-conv-var-rates}, we have the
following result
\begin{corollary}[MLMC Computational
  Complexity]\label{thm:mlmc-complexity-eta1/2}
  Let \cref{ass:cond-density} and \cref{ass:strong-conv} be satisfied for all
  \(q \geq 2\). When \(\beta \leq 1\), assume further that the bias is bounded according
  to \cref{ass:bias} for some \(\alpha\). Then, the computational complexity of a
  MLMC estimator with MSE \(\varepsilon^2\) based on \(\Zest_{\ell}\) and \(h_{\ell} = h_{0}
  M^{-\ell}\) for \(M \in \nset_{+}\) and \(\eta = 4/3\) is
  \[
    \begin{cases}
      \Order{\varepsilon^{-2 -\frac{1-\beta}{\alpha} - \nu}} & \beta \leq 1,\\
      \Order{\varepsilon^{-2}} & \beta > 1,
    \end{cases}
  \]
  for any \(\nu > 0\).
\end{corollary}
\begin{details}
  \begin{proof}%
    Substitute the value \(p=1/2\) in \cref{thm:mlmc-complexity-optimal-eta}
    \[
      \begin{cases}
        \Order{\varepsilon^{-2-\maxp{0, 1/4 -\diagbeta}/\alpha} \, \abs{\log \varepsilon}^{2\I{4 \diagbeta=1}}}
        & 3/2 < 2\p{\crossbeta-\diagbeta}  \\
        \Order{\varepsilon^{-2-\maxp{0, 1-\crossbeta}/\alpha} \, \abs{\log \varepsilon}^{2\I{\crossbeta=1}}}
        & 2 \p{\crossbeta-\diagbeta} \leq 3/2 \\
      \end{cases}
    \]

    Note that \(2\p{\crossbeta-\diagbeta} = \beta\p{1- \frac{4}{q+2} +
      \frac{1}{q+1}}\). Hence if \(\beta>3/2\), then there is \(q\) large enough
    such that the first case applies. In this case, for a sufficiently large
    \(q\), \(\diagbeta > 1/4\) and we arrive at \(\varepsilon^{-2}\) computational
    complexity. For \(1 < \beta<3/2\), we are in the second case but we can find a
    sufficiently large \(q\) for which \(\crossbeta > 1\), yielding
    \(\varepsilon^{-2}\). When \(\beta\leq1\), we are in the second case with increased
    complexity.
  \end{proof}
\end{details}

\subsection*{Numerical Experiments}\label{sec:strong-num}

In this section, we consider the SDE for the Geometric Brownian Motion (GBM)
\begin{equation} \label{eq:gbm}%
  \D X_{i, t} = \mu_{i} \, X_{i,t} \D t + \sigma_{i} \, X_{i,t} \p*{\rho \D W_{i, t} +
    \p*{1-\rho^{2}}^{1/2} \D W_{0, t}},%
\end{equation}%
for \(i=1, \ldots, d\). Here \(\br*{\p{W_{{i}, t}}_{t \geq 0}}_{i=0}^{d}\) are
independent Wiener processes. The processes \(\br*{\p{W_{{i}, t}}_{t \geq
    0}}_{i=1}^{d}\) model the idiosyncratic noise in the \(d\)--dimensional
system while \({\p{W_{0, t}}_{t \geq 0}}\) models the systematic noise in the
system. As an example, we compute \(\prob{\frac{1}{d}\sum_{i=1}^{d}X_{i,1} \le
  1}\)%
. We set \(\rho{=}0.7\) and \(X_{i, 0}{=}1,\, \mu_{i}{=}0.05,\, \sigma_{i}{=}0.2\) for
all \(i\in \br{1, \ldots, d}\). We approximate the path of \(\br{X_{i,t}}_{i=1}^{d}\)
using the Euler-Maruyama or the Milstein numerical schemes
\cite{kloden:numsde} and set the time step size at level \(\ell\) as \(h_{\ell} =
2^{-\ell-1}\) and use the new branching estimator in \cref{def:est-main} with
\(\eta=1\) along with a traditional estimator without branching. Note that this
sequence of time steps sizes is not optimal and other choices such as \(h_{\ell}
\propto 4^{-\ell}\) would lead to better computational performance for both the
branching and non-branching estimators, see
\cite{giles:MLMC,abdullatif:meshMLMC} for further analysis. We choose this
sub-optimal sequence as it produces more data points in the plots below and
makes inferring the computational complexity and convergence trends easier.

\cref{fig:tau-conv-gbm-1d} shows the convergence of \(\E{\p{\E*{\Delta P_{\ell} \given
      \mathcal F_{1-\tau}}}^{2}}\) for an Euler-Maruyama approximation
for \(d{=}1, 2, 3\), verifying \cref{ass:est-main} for \(p=1/2\) as shown by
\cref{thm:strong-conv-var-rates}.
\Cref{fig:Vl-gbm-1d,fig:Wl-gbm-1d} confirm the claims of
\cref{lem:estimator-workvar}. We only show the results for \(d{=}1\) as the
numerical results for \(d{>}1\) show similar rates for the work and variance
convergence when using Euler-Maruyama. Recall that for the example in
\cref{eq:gbm}, by \cref{thm:strong-conv-var-rates}, when
\cref{ass:strong-conv,ass:cond-density} are satisfied as we argued above, we
have \(\diagbeta {\approx} \beta/2\) and \(\crossbeta {\approx} \beta\), hence \(\var{ \Zest_{\ell}}
\approx \Order{h_{\ell}^{\minp{\beta, \beta/2 + 1}}}\) where \(\beta{=}1\) for Euler-Maruyama and
\(\beta{=}2\) for the Milstein numerical scheme.
\Cref{fig:total-work-gbm-1d} shows the total work estimate of a MLMC sampler
based on \(\Zest_{\ell}\). %
For the previous values of \(\crossbeta, \diagbeta\) and \(\beta\) and \(p=1/2,
\eta=1\), the computational complexity for the MLMC sampler based on
\(\Zest_{\ell}\) is \(\Order{\varepsilon^{-2}\abs{\log\varepsilon}^{3}}\) when using Euler-Maruyama
and \(\Order{\varepsilon^{-2}}\) when using Milstein, c.f.,
\cref{rem:mlmc-complexity-eta=1}. As discussed in \cref{rem:optimal-eta} and
\cref{lem:estimator-workvar}, in theory choosing \(\eta{=}4/3\) when \(p=1/2\)
leads to smaller computational complexity than \(\eta{=}1\); in particular the
computational complexity of MLMC when using the branching estimator with an
Euler-Maruyama scheme is
\(\Order{\varepsilon^{-2}\abs{\log\varepsilon}^{2}}\). %
In practice, we observed that the difference in computational complexity
is not significant for reasonable tolerances due to the additional branching
cost when using \(\eta{=}4/3\) where the branching times do not align with the
time-stepping scheme.
In any case, these results are an improvement over computational
complexity for the MLMC sampler based on \(\Delta P_{\ell}\), labelled
``\tradestlabel'', which is approximately \(\Order{\varepsilon^{-5/2}}\) for
Euler-Maruyama, and \(\Order{\varepsilon^{-2}\abs{\log\p*{\varepsilon}}^{2}}\) for
Milstein. %

The kurtosis of \(\Delta P_{\ell}\) grows approximately in proportion to
\(h_{\ell}^{-1/2}\) when using Euler-Maruyama or \(h_{\ell}^{-1}\) when using
Milstein; recall that in both cases \(\E{\p{\Delta P_{\ell}}^{2}} = \E{\p{\Delta
    P_{\ell}}^{4}}\). %
This leads to difficulties when estimating \(\var{\Delta P_{\ell}}\) for sufficiently
large \(\ell\) and determining the optimal number of samples in MLMC for these
levels becomes difficult. On the other hand, \Cref{fig:kurt-gbm-1d}
illustrates another benefit of our branching estimator: \(\Zest_{\ell}\) has a
bounded kurtosis and hence an MLMC algorithm that relies on variance
estimates is more stable when using \(\Zest_{\ell}\) than when using \(\Delta P_{\ell}\).
See also \cref{sec:kurtosis} for a proof of the boundedness of the kurtosis of
\(\Zest_{\ell}\).
\def\excludeTriangle{1}
\pgfplotstableread[col sep=comma,header=false]{
  gbm1-Wl-x0,0.5,0.25,0.125,0.0625,0.03125,0.015625,0.0078125,0.00390625,0.001953125,0.0009765625,0.00048828125,0.000244140625,0.0001220703125,,
gbm1-Wl-y0,2,4,8,16,32,64,128,256,512,1024,2048,4096,8192,,
gbm1-Wl-y1,2,6,16,40,96,224,512,1152,2560,5632,12288,26624,57344,,
gbm1-Wl-y2,2,4,12,34,80,152,326,720,1526,2978,6124,12780,26242,,
gbm1-Tl-y0,5.17784932236763e-06,9.44562802079377e-06,1.68298552655111e-05,4.82884386352673e-05,6.12279898516691e-05,0.000138024197784579,0.000295459565463339,0.000546377723811159,0.00101866780952284,0.00204514244393368,0.00394231858075994,0.00787600251187564,0.0153151258842865,,
gbm1-Tl-y1,4.84129425826346e-06,1.03491827086279e-05,1.96944639846018e-05,4.39257262523767e-05,7.87255172706713e-05,0.000153053256024601,0.000306723492255636,0.000604528814886406,0.00121174644702559,0.00238416060734137,0.00483648980783809,0.00968662133452239,0.0190355887230794,,
gbm1-Tl-y2,5.0786590784978e-06,9.42905664823617e-06,1.80740134826132e-05,4.70572116838139e-05,8.06885775012575e-05,0.000156810286270965,0.00029969631201902,0.000614826142503198,0.0011390105236298,0.00219864997087391,0.00426399700604616,0.00846149403245965,0.0167920343656164,,
gbm1-Tl-y3,5.63306271270582e-06,1.07999278861246e-05,1.96069206117065e-05,4.45998279721874e-05,7.21809673745921e-05,0.000150970416558776,0.000326631004406009,0.000644658637132235,0.00123810227367149,0.00236820360467692,0.00461228606664831,0.00915741593903227,0.0180462901393889,,
gbm1-Tl-y4,5.51705812193026e-06,1.1813458601001e-05,2.79235673748004e-05,4.60579875073615e-05,9.43018562474828e-05,0.000190112517707667,0.00037184704237493,0.000726371778376922,0.00140746869265464,0.00286053190968788,0.00572585295624794,0.0117689156603471,0.0232223721311256,,
gbm1-Tl-y5,5.65444303166335e-06,1.08397986953426e-05,2.14123037780166e-05,5.56189493293975e-05,9.27260634341058e-05,0.000183711987059967,0.000358157718115172,0.000711264942482019,0.00137455455365644,0.00261376219356705,0.00513855521845969,0.0101911628768323,0.0199043598671438,,
gbm1-Vl-y0,0.243581036929387,0.0182011646554738,0.0146579684542461,0.0109926264123291,0.0077469812644133,0.00571714363412309,0.00393907457183227,0.00277467747170271,0.00191182174190478,0.00136243322557298,0.000969322407697828,0.000658837100230993,0.000480692317600042,,
gbm1-Vl-y1,0.243471977565013,0.0091638866521905,0.00385367459763745,0.00148369110228445,0.000575840731757017,0.000229627440482483,9.10772406163576e-05,3.71221508765043e-05,1.56195378979639e-05,6.78605408956325e-06,3.01887425386318e-06,1.39407678322295e-06,6.36115696896646e-07,,
gbm1-Vl-y2,0.24350297901185,0.0180249900953132,0.00743903387322659,0.00281525224497403,0.00105957352890775,0.000750574976416446,0.000270840589822855,0.000104398104480867,4.03570162557899e-05,2.63294377882266e-05,1.0168991251511e-05,4.10559720595246e-06,1.66321848000192e-06,,
gbm1-Vl-y3,0.243366377371895,0.00926715879918905,0.00495064411267794,0.00252753771942447,0.00125639928343949,0.000670332974968559,0.000289525998733151,0.000192041892189782,0.000108467531604667,1.49281210935915e-05,1.39329268872469e-05,5.47344789199739e-05,5.97130192274534e-06,,
gbm1-Vl-y4,0.243342849734293,0.00477289204394543,0.00123505294702615,0.000305506235509815,7.91654071206616e-05,1.91549615659957e-05,5.05032702833698e-06,1.25359008605207e-06,3.19996742999642e-07,7.69524766737752e-08,1.92943967295368e-08,4.87882584594541e-09,1.21675476845263e-09,,
gbm1-Vl-y5,0.243407006075258,0.00912393282546325,0.00248741810012577,0.000648276699824025,0.000159555695023932,8.30484681286344e-05,2.16027974124815e-05,4.9658099218645e-06,1.29567855506829e-06,6.1140235304431e-07,1.48409050487109e-07,4.12264865640936e-08,9.95389587898816e-09,,
gbm1-tau-y0,2.21184247238621e-07,4.28750067000177e-07,6.87044137602399e-07,1.01800746978468e-06,1.46829778221762e-06,2.10705076812939e-06,3.03300323000379e-06,4.30325793612535e-06,6.02655350023015e-06,8.66470822862759e-06,1.23059218096885e-05,1.71763426179339e-05,,,
gbm1-tau-x0,0.5,0.198425131496025,0.0787450656184296,0.03125,0.0124015707185016,0.00492156660115185,0.001953125,0.000775098169906348,0.000307597912571991,,,,,,
gbm1-tau-y1,2.25556124547485e-07,5.09638695200537e-07,9.32474804531996e-07,1.4906476257713e-06,2.47760942787122e-06,3.85940454568073e-06,6.04908937101911e-06,9.43712368132962e-06,1.58613654458599e-05,,,,,,
gbm1-tau-y2,2.37279636844708e-11,3.98629789899109e-11,6.16927055796241e-11,1.34774833727794e-10,1.44266019201582e-10,1.2148717406449e-10,1.06301277306429e-10,9.41525598999801e-10,2.42974348128981e-10,2.42974348128981e-10,4.85948696257962e-10,1.94379478503185e-09,,,
gbm1-tau-y3,6.07435870322452e-11,0,1.2148717406449e-10,1.2148717406449e-10,7.28923044386943e-10,9.71897392515923e-10,1.94379478503185e-09,0,0,,,,,,
gbm1-kurt-y0,1.1054098981026,53.0513366182093,66.9481615316345,90.230879588822,128.626761440513,174.126850011883,253.732485020911,360.310961988347,523.142942675885,733.513448638952,1032.32386453712,1517.6940558902,2080.17701214964,,
gbm1-kurt-y1,1.1072488505704,30.3386698718963,22.0700629067063,18.8076280993433,16.3897209842241,14.3655204423388,12.7895881849227,11.8768704414819,11.206037736851,10.7535217630717,10.0679933736058,10.2530071151078,10.2446991405208,,
gbm1-kurt-y2,1.106725938459,53.5726157569503,36.6901991751966,27.0636282354189,23.3436469810663,29.8161348379147,24.6821851760094,21.685601704599,19.5740495348969,22.9568312913711,21.4232314854377,19.2303386680088,18.0496383777297,,
gbm1-kurt-y3,1.10903104528639,113.039247701791,207.039307023681,400.625146208055,800.907200099424,1496.78878478687,3458.9258876648,5212.13810144521,9224.35046829554,66992.6669204541,71777.4288082965,18274.8034831047,167472.66676818,,
gbm1-kurt-y4,1.10942832753008,61.0211350821892,61.7224194858056,64.2849860824216,64.49922227011,67.3802785220503,66.1426608719502,65.166520918766,65.866005785747,68.0877796956148,81.0994621132855,67.6909804926982,66.6885057666751,,
gbm1-kurt-y5,1.10834517922962,114.716616983161,108.510830579723,110.643412031108,113.925222270982,219.367198230092,190.144214841833,230.369313007537,216.328586510112,417.403364420966,420.265464528709,380.022818773229,412.574860576363,,
gbm1-work-x0,0.0290709672551173,0.0145354836275586,0.00726774181377932,0.00363387090688966,0.00181693545344483,0.000908467726722416,0.000454233863361208,0.000227116931680604,,,,,,,
gbm1-work-y0,6.78100992,8.61896704,6.9490688,22.44464128,62.46688256,45.70584096,80.48598464,123.92591596,,,,,,,
gbm1-work-x1,0.0290692784868164,0.0145346392434082,0.00726731962170411,0.00363365981085205,0.00181682990542603,0.000908414952713013,0.000454207476356507,0.000227103738178253,,,,,,,
gbm1-work-y1,9.6681984,8.67180544,9.14352128,10.87940608,15.78709888,16.35810304,24.61850048,23.43021796,,,,,,,
gbm1-work-x2,0.0290643137809927,0.0145321568904963,0.00726607844524817,0.00363303922262408,0.00181651961131204,0.000908259805656021,0.00045412990282801,0.000227064951414005,,,,,,,
gbm1-work-y2,7.94525696,9.32855808,19.13929728,20.00897024,14.99583872,23.00689984,27.6701952,32.80099056,,,,,,,
gbm1-work-x3,0.0318907564064499,0.015945378203225,0.00797268910161249,0.00398634455080624,0.00199317227540312,0.000996586137701561,0.00049829306885078,0.00024914653442539,,,,,,,
gbm1-work-y3,2.78757376,6.24320512,7.37181696,9.88161024,12.56622464,13.7745184,23.24793464,20.2800665,,,,,,,
gbm1-work-x4,0.0318903965398135,0.0159451982699068,0.00797259913495338,0.00398629956747669,0.00199314978373835,0.000996574891869173,0.000498287445934586,0.000249143722967293,,,,,,,
gbm1-work-y4,9.4404608,7.94877952,6.81693184,6.2564352,6.01391744,6.15581632,6.23439184,6.29983038,,,,,,,
gbm1-work-x5,0.0319007032794226,0.0159503516397113,0.00797517581985565,0.00398758790992783,0.00199379395496391,0.000996896977481957,0.000498448488740978,0.000249224244370489,,,,,,,
gbm1-work-y5,7.51730688,7.26286336,6.756864,6.68463616,7.1218816,7.79457952,8.17500056,8.56337864,,,,,,,
gbm1-time-y0,4.66572460032e-06,4.3108620288e-06,2.45720219648e-06,5.93225851136e-06,1.361722226112e-05,9.75785841056e-06,1.560442668996e-05,2.137718804596e-05,,,,,,,
gbm1-time-y1,5.040791552e-06,3.95334144e-06,3.00446678016e-06,2.73622813696e-06,3.22733063488e-06,3.08193846688e-06,4.13919545632e-06,3.66780179811e-06,,,,,,,
gbm1-time-y2,5.47813720064e-06,4.43874291712e-06,5.64158620672e-06,4.49151763456e-06,3.11974249792e-06,4.27410615664e-06,4.84372695464e-06,5.10711428506e-06,,,,,,,
gbm1-time-y3,2.88186007552e-06,3.9294898176e-06,3.36733014016e-06,3.50889551616e-06,3.78697758464e-06,4.01037271088e-06,6.07767592772e-06,4.95032789384e-06,,,,,,,
gbm1-time-y4,4.77480992768e-06,3.0809524224e-06,2.24432247808e-06,1.7746828672e-06,1.62192003584e-06,1.58358268512e-06,1.55509202028e-06,1.54762237921e-06,,,,,,,
gbm1-time-y5,5.22360963072e-06,4.01791909888e-06,2.84655591424e-06,2.14430717696e-06,2.1021575072e-06,2.10170623888e-06,2.08082078392e-06,2.09285059386e-06,,,,,,,
gbm1-levels-y0,3,4,4,6,8,8,9,10,,,,,,,
gbm1-levels-y1,3,4,5,6,7,8,10,10,,,,,,,
gbm1-levels-y2,3,4,6,7,7,8,9,10,,,,,,,
gbm1-levels-y3,2,4,5,6,7,8,10,10,,,,,,,
gbm1-levels-y4,3,4,5,6,7,8,9,10,,,,,,,
gbm2-Wl-y0,6,12,24,48,96,192,384,768,1536,3072,6144,12288,24576,,
gbm2-Wl-y1,6,18,48,120,288,672,1536,3456,7680,16896,36864,79872,172032,,
gbm2-Wl-y2,6,12,36,102,240,456,978,2160,4578,8934,18372,38340,78726,,
gbm2-Tl-y0,5.89407838074265e-06,1.14216169069527e-05,2.08957661773748e-05,5.00321065544323e-05,7.88931256741475e-05,0.000168218084723706,0.000369381834367278,0.000662166536756002,0.00129434059047775,0.00259506326761975,0.00499196816757796,0.00983053596037778,0.0192428590784407,,
gbm2-Tl-y1,5.94787180993208e-06,1.38706633239795e-05,3.98364587194601e-05,5.75888987369598e-05,0.000132692189543111,0.00028994794958716,0.000590232798627987,0.00120902567105308,0.0025793082803298,0.00552380504003566,0.0124429047250064,0.0285919609812984,0.0636192820140511,,
gbm2-Tl-y2,6.08329750170374e-06,1.13945247925771e-05,2.64393277229018e-05,6.42997531849108e-05,0.000116911766350649,0.000233582589229581,0.000462341925283526,0.000940750502524482,0.00183751530802933,0.00355338148511709,0.00712976191453873,0.014599354844659,0.0290013633101324,,
gbm2-Vl-y0,0.217198925650508,0.0133947846356597,0.0100139124202237,0.00662952469765349,0.00531817606301853,0.00359270328854619,0.00240839907536183,0.00172770242644225,0.0011066721734084,0.00109969100447628,0.000659805179366747,0.000392116942973092,0.000306525330483742,,
gbm2-Vl-y1,0.21694826786441,0.0065163398796842,0.00260513075943423,0.000936166061156226,0.000366510411737279,0.000134855057818009,5.37395150534565e-05,2.12673920393716e-05,8.63264851299915e-06,3.61748455535737e-06,1.58050453618001e-06,6.95526483905818e-07,3.24764577719256e-07,,
gbm2-Vl-y2,0.217298705654538,0.0131507854654347,0.00486091259466503,0.00183109036814374,0.000670869418609321,0.000476864658550557,0.000175868534955723,6.33416211669394e-05,2.40563645015223e-05,1.66679309700327e-05,6.21782896057956e-06,2.42939883904047e-06,9.70666908319255e-07,,
gbm2-tau-y0,1.00873742893243e-07,1.99836435591339e-07,3.27690771430921e-07,4.99919721275378e-07,7.11026465057567e-07,1.0140914066582e-06,1.46952404338083e-06,2.06373299762702e-06,2.94673215052125e-06,4.10675243207604e-06,6.04641665318969e-06,8.13186548318073e-06,,,
gbm2-tau-y1,9.79945917797696e-08,2.60772219129429e-07,4.2854600651249e-07,6.84337251505275e-07,1.28436240420979e-06,1.76593756220143e-06,2.9477647915008e-06,4.55819877089968e-06,7.03264953224522e-06,,,,,,
gbm2-kurt-y0,1.6040743387888,73.3153443118687,99.1040051274209,150.105100014217,187.636436643905,278.093375318315,414.886751944353,578.961610751023,904.076136412904,908.738078208429,1514.6549707513,2549.96084612212,3261.6553881695,,
gbm2-kurt-y1,1.60939379624357,41.9716937496591,30.6559646034234,24.7449542407404,20.3213318265332,17.8328618962611,16.6726359267187,15.1614796962746,14.4588790097316,12.7264940058278,12.0464131164093,11.5356442379179,11.0154757126602,,
gbm2-kurt-y2,1.60196022331491,74.6986124452935,54.1477530338818,41.5450042637005,33.0907803485043,41.6670112180719,34.2562549109257,28.1053557508269,25.7110165247156,30.0567681153579,25.050452253465,23.768326182092,21.3044213320078,,
gbm2-work-x0,0.0389222260478364,0.0194611130239182,0.0097305565119591,0.00486527825597955,0.00243263912798977,0.00121631956399489,0.000608159781997444,0.000304079890998722,,,,,,,
gbm2-work-y0,8.40597504,12.05280768,9.35104512,39.20572416,80.02174464,84.54847968,116.36764008,171.35032698,,,,,,,
gbm2-work-x1,0.0389204508826492,0.0194602254413246,0.00973011272066231,0.00486505636033116,0.00243252818016558,0.00121626409008279,0.000608132045041394,0.000304066022520697,,,,,,,
gbm2-work-y1,10.20985344,12.42906624,15.22335744,11.29852416,25.54718208,29.04433344,34.19908584,40.21023492,,,,,,,
gbm2-work-x2,0.0389294870963943,0.0194647435481971,0.00973237177409856,0.00486618588704928,0.00243309294352464,0.00121654647176232,0.00060827323588116,0.00030413661794058,,,,,,,
gbm2-work-y2,8.43350016,24.72296448,14.89539072,35.52978432,40.16178816,42.61547136,55.49048592,59.12611824,,,,,,,
gbm2-time-y0,2.5279782912e-06,2.95077289984e-06,1.9121645568e-06,5.42398046976e-06,1.023070302208e-05,1.044530078864e-05,1.355141861756e-05,1.78260343805e-05,,,,,,,
gbm2-time-y1,3.96341444608e-06,4.067791872e-06,3.34791625728e-06,1.96099589888e-06,3.76699515968e-06,4.13549754144e-06,4.80005141952e-06,5.24707189151e-06,,,,,,,
gbm2-time-y2,1.98198296576e-06,4.69598322688e-06,2.76695216128e-06,5.39213292288e-06,5.5497408832e-06,5.781177608e-06,7.3228619502e-06,7.23585493474e-06,,,,,,,
gbm2-levels-y0,2,3,3,5,7,7,8,9,,,,,,,
gbm2-levels-y1,2,3,4,4,6,7,8,9,,,,,,,
gbm2-levels-y2,2,4,4,6,7,7,8,9,,,,,,,
gbm3-Wl-y0,8,16,32,64,128,256,512,1024,2048,4096,8192,16384,32768,,
gbm3-Wl-y1,8,24,64,160,384,896,2048,4608,10240,22528,49152,106496,229376,,
gbm3-Wl-y2,8,16,48,136,320,608,1304,2880,6104,11912,24496,51120,104968,,
gbm3-Tl-y0,5.89998783009827e-06,1.17952101359701e-05,2.18524871643182e-05,5.89685400675057e-05,8.53350122643125e-05,0.000194552435189675,0.000358914635173834,0.000705338751387065,0.00136249793539192,0.00263896382704472,0.00515167608975795,0.0101314477641491,0.019714657241704,,
gbm3-Tl-y1,8.06555294307174e-06,1.43431701288102e-05,3.7146270939499e-05,6.74842383451523e-05,0.000139506052538847,0.000294822358354262,0.000610734487937135,0.0013349538367645,0.00286792834805455,0.00634220062025413,0.014551977518068,0.0347666323680881,0.0798603792732973,,
gbm3-Tl-y2,6.38286944977037e-06,1.17893201435447e-05,2.99950793480418e-05,5.95081725697608e-05,0.00012289035424685,0.000247601496803153,0.000541264773554103,0.00102897036322363,0.00204528155623917,0.0039727918056261,0.00810356182372494,0.0165742456222511,0.0331141788199259,,
gbm3-Vl-y0,0.201498354332326,0.0126746501337767,0.00952195805575118,0.00691404538073933,0.00479074150684992,0.00328220319140102,0.00239847346986059,0.00172173108886263,0.00123903804631875,0.000850897553803843,0.000535420808146375,0.000511466933242728,0.000277667188538011,,
gbm3-Vl-y1,0.201242894674828,0.00646377909496796,0.00246927711201747,0.000913595869404882,0.000340389099080082,0.000130513411683675,5.00500414481661e-05,2.00318790970721e-05,8.24058818406781e-06,3.42716905770703e-06,1.494973508614e-06,6.62887736425288e-07,3.02506222708207e-07,,
gbm3-Vl-y2,0.20141459188471,0.0126416365617013,0.00467911394707276,0.00172002286714852,0.000644620263172679,0.000445870585520711,0.000164443101484039,6.1140776854789e-05,2.30700467688505e-05,1.5456674585159e-05,5.92529991344437e-06,2.29269483996587e-06,9.12959638572986e-07,,
gbm3-tau-y0,9.26605455435006e-08,1.87944454751956e-07,3.080183912994e-07,4.60339579612586e-07,6.69565170433871e-07,9.44828531544679e-07,1.36495395830482e-06,1.91466823504989e-06,2.74245146733181e-06,3.87118880156499e-06,5.41832796327627e-06,7.73727514181927e-06,,,
gbm3-tau-y1,8.83059896481265e-08,2.2848700262179e-07,4.03398161481141e-07,7.13251198932623e-07,1.10553328398686e-06,1.73726658912221e-06,2.77476705563296e-06,4.61068123009554e-06,7.64300109474522e-06,,,,,,
gbm3-kurt-y0,1.96281968809892,77.9052206160356,104.293370979738,144.144180481721,208.320549157615,304.427521885723,416.738436931107,580.969580777298,806.698470835474,1175.88627447773,1867.0079982138,1953.2787081531,3601.39091243406,,
gbm3-kurt-y1,1.96911953893238,41.7927682285873,31.4110477004874,26.258439803988,21.4843240412755,18.5789323877045,16.805657015951,15.3128121294031,14.2735585051267,13.294129263419,12.7066542957646,12.3868569549071,11.7679800111169,,
gbm3-kurt-y2,1.96488357989675,78.0264447179577,56.9835043221835,43.0106198136184,34.6816238184753,45.8857802310285,35.9527301273138,30.9233465799312,26.9855094229786,30.3936861555395,27.1348618546865,23.9372112544523,23.1624666289629,,
gbm3-work-x0,0.0446455327741155,0.0223227663870577,0.0111613831935289,0.00558069159676443,0.00279034579838222,0.00139517289919111,0.000697586449595554,0.000348793224797777,,,,,,,
gbm3-work-y0,10.94975488,9.81598208,25.0793984,40.78845952,67.74245376,69.22742912,95.24115648,243.89252944,,,,,,,
gbm3-work-x1,0.044644899211809,0.0223224496059045,0.0111612248029522,0.00558061240147612,0.00279030620073806,0.00139515310036903,0.000697576550184515,0.000348788275092258,,,,,,,
gbm3-work-y1,13.42308352,10.0777984,20.310016,35.62627072,24.46022144,23.01052032,47.15609472,54.1152432,,,,,,,
gbm3-work-x2,0.0446204599800844,0.0223102299900422,0.0111551149950211,0.00557755749751055,0.00278877874875527,0.00139438937437764,0.000697194687188818,0.000348597343594409,,,,,,,
gbm3-work-y2,10.70989312,15.73126144,11.38868224,20.64726016,55.41524992,40.84873728,62.30011776,95.2678832,,,,,,,
gbm3-time-y0,2.60754522112e-06,2.08678621184e-06,3.80754735104e-06,4.877497344e-06,7.2928272672e-06,7.2876006832e-06,9.6668200958e-06,2.105205708257e-05,,,,,,,
gbm3-time-y1,5.37413779456e-06,2.530430976e-06,3.1273622528e-06,4.84349839616e-06,3.27033720192e-06,2.98140398256e-06,5.97605346732e-06,6.33467098326e-06,,,,,,,
gbm3-time-y2,3.54589851648e-06,4.03909574656e-06,2.16487796736e-06,3.02331486464e-06,6.66147190784e-06,4.88011707696e-06,7.2890170406e-06,1.020351012436e-05,,,,,,,
gbm3-levels-y0,2,2,4,5,6,6,7,9,,,,,,,
gbm3-levels-y1,2,2,4,6,6,6,8,8,,,,,,,
gbm3-levels-y2,2,3,3,4,7,7,8,10,,,,,,,
CC_el1-Wl-y0,4,8,16,32,64,128,256,512,1024,2048,4096,8192,16384,,
CC_el1-Wl-y1,4,12,32,80,192,448,1024,2304,5120,11264,24576,53248,114688,,
CC_el1-Wl-y2,4,8,24,68,160,304,652,1440,3052,5956,12248,25560,52484,,
CC_el1-Wl-y3,4,12,40,104,372,784,1968,4620,14404,30016,72752,167280,499728,,
CC_el1-Tl-y0,5.50642063853088e-06,9.57372414458329e-06,1.67225744978637e-05,4.23866835464338e-05,6.82203191670643e-05,0.000122950863876161,0.000249343612201654,0.000498135409489938,0.00100510894231925,0.00191157912159232,0.00379810697571108,0.00746845810893615,0.0145966906940481,,
CC_el1-Tl-y1,5.21921546785695e-06,1.0571346683487e-05,2.16363733456393e-05,4.24296183476023e-05,8.31321664866368e-05,0.000172991355655679,0.000311455277215903,0.000645616946706347,0.00131859670802476,0.00258678387803067,0.00531549314262381,0.0109184314495629,0.0216278351587095,,
CC_el1-Tl-y2,5.59584896086128e-06,9.54105263682687e-06,1.81480945580325e-05,4.6661112955801e-05,7.42455237325589e-05,0.000143402210751157,0.000304761421006576,0.000592095908370747,0.00114009888546102,0.00221985988247736,0.00436522587421973,0.00874359875371692,0.0168503860004578,,
CC_el1-Tl-y3,5.84506305160037e-06,1.36014380177874e-05,2.55368514710171e-05,6.14226424390343e-05,9.54994803686051e-05,0.000193196502128604,0.000398696866716929,0.000737738402880681,0.00140843027787414,0.0028235859689629,0.00556851808978304,0.0110081791120825,0.0219101407803643,,
CC_el1-Tl-y4,5.66984580200949e-06,1.47862794103136e-05,3.69974415583216e-05,5.34379534470807e-05,0.000109635036984447,0.000228896291250256,0.000478920358096718,0.000902016700310692,0.00182760919450195,0.00360627479210591,0.00729719500632802,0.0151067980898509,0.0304124707206609,,
CC_el1-Tl-y5,5.60297280739827e-06,1.48956885763035e-05,3.95258853010311e-05,5.77056792321478e-05,0.000137925087503947,0.000280692091413364,0.000547764020835518,0.00114525781837618,0.00287668044779711,0.00620543687017101,0.0147123424165473,0.0336680522250237,0.093962241882921,,
CC_el1-Vl-y0,0.0096570842905883,0.00880230458191604,0.00668137719103558,0.00459136247920808,0.00321182648129742,0.00223520578928151,0.00148778872459171,0.00114649168069699,0.000843947760938578,0.000524482483085964,0.000425945509325936,0.0003294187650472,0.000193073212750771,,
CC_el1-Vl-y1,0.00958974856595983,0.00487982019482242,0.00184190090386358,0.00070662201211256,0.000292028417509513,0.000113951492060986,5.15219557516966e-05,2.29196549061007e-05,1.03462890628419e-05,4.69879590425183e-06,2.3589823137095e-06,1.07923361386432e-06,5.27849298058056e-07,,
CC_el1-Vl-y2,0.00970684811234102,0.00878204286621668,0.00340154758750355,0.00122427770170538,0.000503356068403882,0.000339321176727876,0.000130133973222663,4.98235544426561e-05,2.18571846620747e-05,1.33233019534635e-05,6.21784721200832e-06,2.53921024425152e-06,1.12746855448978e-06,,
CC_el1-Vl-y3,0.0183486239141294,0.00373149348790614,0.002451084129132,0.00169564662866648,0.00117236642994993,0.000784979560044509,0.000562050008040625,0.000347828907509876,0.000280400314309388,0.000182623240254105,0.000168687068591462,0.000108224847711479,7.86224916182578e-05,,
CC_el1-Vl-y4,0.0185210327371293,0.00193011222528905,0.000635835875297948,0.000215440556903297,7.34574332092736e-05,2.60632451579094e-05,9.22757412938589e-06,3.27174268358244e-06,1.16630924786712e-06,4.04596442893116e-07,1.42813486672659e-07,5.03636926022261e-08,1.76760734727418e-08,,
CC_el1-Vl-y5,0.01860242570435,0.00192286952563999,0.000656490775987417,0.000212919661204536,3.85645132380821e-05,1.29754772601812e-05,4.50841636118519e-06,1.58832837197315e-06,2.89795748625567e-07,1.04248005312591e-07,3.66828417713359e-08,1.28369920185067e-08,2.26724165792324e-09,,
CC_el1-tau-y0,1.14014051902066e-07,3.13477721183923e-07,6.36356461579633e-07,1.11883802778402e-06,1.79799878673189e-06,2.70241385052918e-06,3.95337487481962e-06,5.6968069380256e-06,8.31300285971089e-06,1.17391841426776e-05,1.67683886874254e-05,2.38775751393312e-05,,,
CC_el1-tau-y1,1.22383141973216e-07,4.09988840674139e-07,9.22269881910579e-07,1.91621719651921e-06,2.82214705351811e-06,5.28323422571656e-06,8.29514424512341e-06,1.40497487062102e-05,1.95584631269904e-05,,,,,,
CC_el1-tau-y2,6.08622268506676e-11,3.26734059935163e-10,2.70498786002967e-11,2.14548247634985e-09,7.64989549187338e-10,2.01213132044312e-09,1.75397107555608e-09,4.81392927230543e-09,2.47833835091561e-08,3.70232162961535e-08,1.14562405142814e-07,3.23520344533738e-07,,,
CC_el1-tau-x0,0.5,0.287174588749259,0.164938488846612,0.0947322854068999,0.0544094102060078,0.03125,0.0179484117968287,0.0103086555529132,0.00592076783793124,0.00340058813787548,0.001953125,0.00112177573730179,0.000644290972057077,0.000370047989870703,0.000212536758617218
CC_el1-tau-y3,1.67903149275074e-11,1.17193895635332e-11,2.39504133440127e-12,3.1498871791135e-11,2.46177623226384e-11,9.73439710155414e-11,1.4085512442194e-09,2.42250645236605e-09,4.69291665751463e-09,2.44587849659525e-09,2.37137269062601e-08,1.64861891679703e-08,7.98588345764549e-08,1.7781166514014e-07,2.96610935478453e-07
CC_el1-kurt-y0,100.55092385127,111.401473571277,147.904171789163,216.524427718921,310.342685738314,446.943318072487,671.461052393724,872.480063885171,1184.73946215987,1906.63378330365,2346.83825166916,3035.71042661543,5179.25964608364,,
CC_el1-kurt-y1,101.278020755376,65.6222884873008,59.9908170747483,58.9407729689405,54.7191144617347,52.4341664158323,58.4360565572995,54.8213985158879,55.4677128670706,62.8946035046747,88.7101078171614,50.9212145928796,58.7869871863474,,
CC_el1-kurt-y2,100.020052279239,111.676006814791,89.3373978374221,72.9412175916963,86.0292638642656,86.7578826205684,76.0453573004207,75.0738934789413,80.5685815180325,85.0516347254205,95.5463767221333,88.0002280838037,83.9359999048736,,
CC_el1-kurt-y3,51.4999998190572,66.2636762111522,101.497911673325,146.910937822471,213.110937693986,318.352622791161,444.637320153045,718.524132722755,891.157875779466,1369.08153432107,1481.38996694335,2308.73773225166,3179.47805670885,,
CC_el1-kurt-y4,50.9926695337723,38.7118017052861,37.8021474633128,38.5155494328041,37.6297814244088,36.9391038891352,43.5192860110761,38.1974518380538,37.9923909790333,38.0546096872141,38.3229206761265,37.971305604056,36.558863091877,,
CC_el1-kurt-y5,50.7564302577034,39.0776385160232,36.2017719123079,33.8749967456349,29.125296454782,28.4657338011978,28.9847808948286,29.6777645202317,30.7279118758576,28.9265886469669,34.6372872044046,30.0672445306104,30.4528934744843,,
CC_el1-work-x0,0.596326888068752,0.298163444034376,0.149081722017188,0.074540861008594,0.037270430504297,0.0186352152521485,0.00931760762607426,0.00465880381303713,,,,,,,
CC_el1-work-y0,3.93216,6.66025984,8.6370304,8.63794176,41.7605376,32.0211936,40.5043776,94.38835772,,,,,,,
CC_el1-work-x1,0.593754953308936,0.296877476654468,0.148438738327234,0.074219369163617,0.0371096845818085,0.0185548422909043,0.00927742114545213,0.00463871057272606,,,,,,,
CC_el1-work-y1,5.24288,1.31072,8.58701824,3.84884736,9.24264704,10.31257024,8.03100192,17.01622608,,,,,,,
CC_el1-work-x2,0.592197899329409,0.296098949664705,0.148049474832352,0.0740247374161762,0.0370123687080881,0.018506184354044,0.00925309217702202,0.00462654608851101,,,,,,,
CC_el1-work-y2,3.93216,3.36379904,1.655808,6.93799936,8.39166208,22.7155616,19.7361496,22.61906524,,,,,,,
CC_el1-work-x3,0.595604889006662,0.297802444503331,0.148901222251666,0.0744506111258328,0.0372253055629164,0.0186126527814582,0.0093063263907291,0.00465316319536455,,,,,,,
CC_el1-work-y3,3.93216,1.07905024,1.49286912,1.04316928,6.69325056,5.34679808,24.95850752,19.57951248,,,,,,,
CC_el1-work-x4,0.599668905053513,0.299834452526757,0.149917226263378,0.0749586131316892,0.0374793065658446,0.0187396532829223,0.00936982664146114,0.00468491332073057,,,,,,,
CC_el1-work-y4,5.24288,1.363968,1.39874304,1.67176192,1.1883264,2.8358528,3.18687616,3.31173196,,,,,,,
CC_el1-work-x5,0.59494246996627,0.297471234983135,0.148735617491568,0.0743678087457838,0.0371839043728919,0.0185919521864459,0.00929597609322297,0.00464798804661149,,,,,,,
CC_el1-work-y5,5.24288,1.31072,0.56049664,1.9198976,1.9626624,2.15975424,3.12451056,4.17704372,,,,,,,
CC_el1-time-y0,1.5852969984e-06,1.86862997504e-06,1.92939649024e-06,1.4830174464e-06,4.737248128e-06,3.49489580848e-06,4.35342511288e-06,8.78209911832e-06,,,,,,,
CC_el1-time-y1,1.66049234944e-06,4.1512308736e-07,1.11161716736e-06,5.2588625152e-07,9.5596298112e-07,9.4564109872e-07,7.0619383272e-07,1.37423242968e-06,,,,,,,
CC_el1-time-y2,1.62418491392e-06,9.7372909568e-07,4.9504182272e-07,1.0523716992e-06,9.1024657408e-07,2.01045865776e-06,1.69906849928e-06,1.85320254168e-06,,,,,,,
CC_el1-time-y3,2.40436887552e-06,6.6657378304e-07,6.8520745984e-07,3.7624620288e-07,1.19938912768e-06,8.8466346512e-07,3.69572207368e-06,2.79353487346e-06,,,,,,,
CC_el1-time-y4,2.3595966464e-06,6.3988068352e-07,4.9529560064e-07,4.4408793344e-07,2.2904080768e-07,4.1963691072e-07,4.2570096452e-07,4.2849091277e-07,,,,,,,
CC_el1-time-y5,2.67425513472e-06,6.6856378368e-07,2.6098393088e-07,4.10410432e-07,3.372551808e-07,3.0807341264e-07,4.0629972528e-07,5.0640570586e-07,,,,,,,
CC_el1-levels-y0,2,4,5,5,8,8,8,10,,,,,,,
CC_el1-levels-y1,2,2,5,5,7,8,8,10,,,,,,,
CC_el1-levels-y2,2,3,3,5,6,9,9,10,,,,,,,
CC_el1-levels-y3,2,2,3,3,6,6,9,9,,,,,,,
CC_el1-levels-y4,2,2,3,4,4,6,7,7,,,,,,,
CC_el1-levels-y5,2,2,2,4,4,5,6,8,,,,,,,
CC1-Tl-y0,5.42057595055574e-06,9.29222673557367e-06,1.6449437874138e-05,4.16183144233789e-05,7.01274962466993e-05,0.000139820862822472,0.000267522397219755,0.000498377328180963,0.000986948025406926,0.00189789833892493,0.00376821454733041,0.00744420111914919,0.0147506371173699,,
CC1-Tl-y1,5.10398132406223e-06,1.04298321589543e-05,2.11395007694603e-05,4.1782438375388e-05,7.97344803525384e-05,0.000158734312814892,0.000329006433866586,0.000665125373015358,0.00134745870117739,0.00264982576321834,0.00531744352571524,0.0108600750598748,0.0216726366886098,,
CC1-Tl-y2,5.28611812242277e-06,9.31963680466269e-06,1.78640048689903e-05,4.57354127221806e-05,7.2829744476042e-05,0.000143871047314565,0.000302847419764586,0.00057605884637043,0.00115395659189315,0.00218579579927739,0.00427276153829257,0.0086454500863032,0.0167755298195466,,
CC1-Tl-y3,5.85909624388263e-06,1.32866163447404e-05,2.5499671412881e-05,6.38954507507336e-05,9.35481284644194e-05,0.000191878459303622,0.000377101450446685,0.000743213284072603,0.00141194784954475,0.00279505677924604,0.00556072057005326,0.0110086790961065,0.0218677435091631,,
CC1-Tl-y4,5.60154106206955e-06,1.47170855835745e-05,3.66881757879713e-05,5.32781871356023e-05,0.000111686943490414,0.000242665787553711,0.000461061574328857,0.00093467756678724,0.00190313744056187,0.00375285918450659,0.00770764257806312,0.0158699070123635,0.0314610558793328,,
CC1-Tl-y5,5.35273105855201e-06,1.46006873458814e-05,3.38988250512986e-05,5.77757657048809e-05,0.000130956110301291,0.000273908599024745,0.000552745315536952,0.00114081114933939,0.00291991398782487,0.00618656190337649,0.0147831794230422,0.0338082598655182,0.093484755786029,,
CC1-Vl-y0,0.0300345190854599,0.0313310167317497,0.0245110630349719,0.0162909874154615,0.0124001089267931,0.0087933228772196,0.00635512507717732,0.00487254032554566,0.00317860200971642,0.00193768405978532,0.00159528858513371,0.00110171042748787,0.000849919367925677,,
CC1-Vl-y1,0.0301241287045113,0.0171228458653242,0.00760839248224516,0.00331501309135883,0.00157445380659953,0.000692709536599631,0.000341437462449698,0.000164079251740145,7.7139857579566e-05,3.79618535210823e-05,1.9619816631293e-05,1.01540215646342e-05,4.79918897268767e-06,,
CC1-Vl-y2,0.0298954013321645,0.0312668268386585,0.0124739693021962,0.00505298720178771,0.00224722612459832,0.00145725835549226,0.000615840672260801,0.000279696071268342,0.000131455647531661,7.54116224988227e-05,3.77208363785454e-05,1.74576060206171e-05,8.55779124843182e-06,,
CC1-Vl-y3,0.0457397221828456,0.0126268767180174,0.00878827725246841,0.00618271275188015,0.00444032995696097,0.00307029691968004,0.00223076046184524,0.00155826925636523,0.00108677101505132,0.00074989702988553,0.00063345838066453,0.000362012340516714,0.000318471084019636,,
CC1-Vl-y4,0.0459987205828821,0.00637587983308107,0.00235096832525174,0.000849320643707155,0.000306595147900409,0.000110569731807548,4.3104079979816e-05,1.58476420089344e-05,6.03684963784422e-06,2.31716908198385e-06,8.52591168544967e-07,3.22338110165017e-07,1.12307931280335e-07,,
CC1-Vl-y5,0.0460625456368349,0.00635614810231549,0.00231729774732,0.000821186574233067,0.000166527727013055,5.96235334578187e-05,2.21733055481813e-05,8.16921392425697e-06,1.93028621171658e-06,7.1289427776381e-07,2.7998827876235e-07,9.31667941865926e-08,3.23056932448994e-08,,
CC1-tau-y0,1.02878732096617e-06,3.10713176135045e-06,6.02367577279449e-06,1.09279649272846e-05,1.72686159231101e-05,2.78727947526677e-05,4.05686979840516e-05,5.76953827195866e-05,7.83885664241329e-05,0.00010481548916762,0.000137067661163913,0.000177237152293989,,,
CC1-tau-y1,1.01072773052629e-06,4.16345657057064e-06,9.52240767752289e-06,1.85306816344049e-05,3.09891913347184e-05,4.99501805396596e-05,7.75272831036027e-05,0.000114382604125199,0.000163508129727309,,,,,,
CC1-tau-y2,1.31269027093414e-09,6.6750321038969e-09,5.53573392758704e-09,2.87364622589889e-08,6.35283008502547e-08,1.51412881863345e-07,8.26974583279555e-07,2.44636444529151e-06,4.55504010437401e-06,7.3244009807611e-06,1.09097304617523e-05,1.59909922605867e-05,,,
CC1-tau-y3,2.88656239463076e-10,1.20000098090453e-09,3.77676514967991e-09,3.743538585533e-09,2.31971987968038e-08,1.47591345012188e-07,4.14689053680487e-07,1.01096892527714e-06,1.95891020973777e-06,3.37174980883386e-06,4.97138519195994e-06,7.20390088998588e-06,9.98911963906258e-06,1.3838922901518e-05,1.8776784277266e-05
CC1-kurt-y0,30.2950228753326,30.1002761110143,39.5165610310611,60.4219566151726,79.9357103018126,113.19241655935,156.98065853805,205.015877880678,314.146914082633,515.831342446886,626.285713685188,907.814346788664,1176.73349479029,,
CC1-kurt-y1,30.1959808633483,20.8682475472668,23.9035818988045,28.9991248201961,34.6897589703589,42.1430998682759,57.0293724451324,56.8310750846842,66.516671619309,78.5983079069899,91.9435841958246,211.32917521788,120.752022759651,,
CC1-kurt-y2,30.4499607109839,30.1639729275758,25.8949297480149,27.5885116417082,31.8146456317167,38.2091002924962,44.3292935246018,58.9751857328642,67.6637854730792,73.8753162057401,116.425900125776,131.241668765921,166.052976699607,,
CC1-kurt-y3,18.8628350212202,18.7101435931694,27.5627280545721,39.7434141635675,55.7548289931525,81.1296603884289,111.91136672805,160.393002406969,229.836825524013,333.22406052346,394.559010700252,690.586942529586,784.901254619333,,
CC1-kurt-y4,18.7397350910699,13.2276133751831,15.1102517857034,17.2152132759468,20.9201292558649,26.0698782952278,45.2008615590181,56.246552547517,89.2364286843573,119.119455941856,149.969530990584,199.099650738948,262.297096513868,,
CC1-kurt-y5,18.7096121409393,13.3038816352119,14.5083752206255,15.9708918696773,23.2926808693649,31.3188004579173,32.7901507957199,44.411481799966,128.013062133915,132.869611964967,251.982837392384,238.407787222194,658.010714897847,,
CC1-work-x0,0.202148511687815,0.101074255843908,0.0505371279219539,0.0252685639609769,0.0126342819804885,0.00631714099024423,0.00315857049512212,0.00157928524756106,,,,,,,
CC1-work-y0,3.93216,36.77192192,31.8416896,83.0411776,66.41674752,387.94459968,284.37605536,510.82648708,,,,,,,
CC1-work-x1,0.202131095159926,0.101065547579963,0.0505327737899814,0.0252663868949907,0.0126331934474954,0.00631659672374768,0.00315829836187384,0.00157914918093692,,,,,,,
CC1-work-y1,5.24288,33.20758272,30.63013376,35.89490688,70.86706688,52.75496768,49.89165984,118.7667082,,,,,,,
CC1-work-x2,0.202177609536364,0.101088804768182,0.050544402384091,0.0252722011920455,0.0126361005960227,0.00631805029801137,0.00315902514900569,0.00157951257450284,,,,,,,
CC1-work-y2,3.93216,32.82026496,39.06863104,34.81261056,70.43177984,50.44477184,123.34996944,143.63722988,,,,,,,
CC1-work-x3,0.201930850299856,0.100965425149928,0.050482712574964,0.025241356287482,0.012620678143741,0.0063103390718705,0.00315516953593525,0.00157758476796763,,,,,,,
CC1-work-y3,3.93216,9.08476416,14.85279232,22.40024576,56.86619904,45.60306176,84.15979904,219.8378236,,,,,,,
CC1-work-x4,0.202139373671382,0.101069686835691,0.0505348434178455,0.0252674217089227,0.0126337108544614,0.00631685542723068,0.00315842771361534,0.00157921385680767,,,,,,,
CC1-work-y4,5.24288,12.02520064,10.0907008,10.35536384,10.43993856,11.0785056,13.487532,19.91880208,,,,,,,
CC1-work-x5,0.202183235896997,0.101091617948499,0.0505458089742494,0.0252729044871247,0.0126364522435623,0.00631822612178117,0.00315911306089059,0.00157955653044529,,,,,,,
CC1-work-y5,5.24288,46.41734656,14.02486784,13.125376,10.46693376,17.50732032,17.05923408,17.03078612,,,,,,,
CC1-time-y0,1.22286456832e-06,7.98106415104e-06,5.5665312768e-06,9.92377719296e-06,7.52419305472e-06,3.898453010528e-05,2.661535443816e-05,4.284364348691e-05,,,,,,,
CC1-time-y1,2.41371906048e-06,5.14581938176e-06,3.6761188352e-06,3.52105444864e-06,5.84191026176e-06,4.21547598992e-06,3.9242205478e-06,7.95341172744e-06,,,,,,,
CC1-time-y2,2.11005259776e-06,5.93765572608e-06,5.17012169728e-06,3.83528314112e-06,6.31528268672e-06,4.37819433248e-06,9.72598099092e-06,9.90937723341e-06,,,,,,,
CC1-time-y3,1.99708475392e-06,3.29620979712e-06,3.773057792e-06,4.09970697216e-06,8.75654325824e-06,6.84949537376e-06,1.184799235732e-05,2.672350053713e-05,,,,,,,
CC1-time-y4,2.46118219776e-06,3.51998742528e-06,2.27808319488e-06,1.7229934976e-06,1.51823122176e-06,1.44167028896e-06,1.64149924868e-06,2.24790991247e-06,,,,,,,
CC1-time-y5,2.37268598784e-06,6.50752421888e-06,2.1361464832e-06,1.91341949184e-06,1.4366832864e-06,2.05914456192e-06,1.89099930672e-06,1.8562889293e-06,,,,,,,
CC1-levels-y0,2,5,5,7,7,10,10,11,,,,,,,
CC1-levels-y1,2,5,6,7,9,9,9,11,,,,,,,
CC1-levels-y2,2,5,6,6,8,8,10,11,,,,,,,
CC1-levels-y3,2,4,5,6,8,8,9,11,,,,,,,
CC1-levels-y4,2,4,5,6,6,7,8,10,,,,,,,
CC1-levels-y5,2,5,5,6,6,8,9,9,,,,,,,
}\loadedtable
\pgfplotstabletranspose[colnames from=0, input colnames to={}]{\loadedtable}{\loadedtable}

\begin{figure}\centering
  \def\thesubfigure{}%
\begin{tikzpicture}

\begin{axis}[
xmode=log,ymode=log,log basis x={2},log basis y={2},
xlabel={\(\tau\)},
xmin=8.05363715071347e-05, xmax=0.757858283255199,
ylabel={\(\E*{\p*{\E*{\Delta P_{\ell} \given \mathcal F_{1-\tau}}}^{2}}\)},
ymin=8e-08, ymax=4e-5,
]

\addplot[mark=square*, legend entry={\(d{=}1\)}]
table [x=gbm1-Wl-x0,y=gbm1-tau-y0] {\loadedtable};

\addplot[mark=*, legend entry={\(d{=}2\)}]
table [x=gbm1-Wl-x0,y=gbm2-tau-y0] {\loadedtable};

\addplot[mark=triangle*, legend entry={\(d{=}3\)}] table
[x=gbm1-Wl-x0,y=gbm3-tau-y0] {\loadedtable};

\addref[text align=center, raise=-3ex]{\(\tau^{-1/2}\)}{5e-8 * x ^ (-0.5)};
\end{axis}

\end{tikzpicture}
   \caption{Numerical verification for \cref{eq:est-assumpt-cross} with
    \(p{=}1/2\) and \(h_{\ell} {=} 2^{-14}\) for the GBM example in \cref{eq:gbm}
    when using Euler-Maruyama.}
  \label{fig:tau-conv-gbm-1d}
\end{figure}
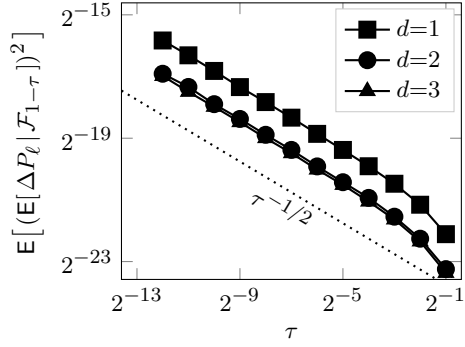

\begin{figure}\centering
  \begin{customlegend}[legend columns=3,legend
    style={align=left,draw=none,column sep=2ex}]
    \addlegendimage{only marks, mark=*};%
    \addlegendentry{\tradestlabel};%
    \addlegendimage{only marks, mark=square*};%
    \ifx\excludeTriangle\undefined
      \addlegendentry{\(\eta=1\)};%
      \addlegendimage{only marks, mark=triangle*};%
      \addlegendentry{\(\eta=4/3\)};
    \else%
      \addlegendentry{\newestlabel};%
    \fi
  \end{customlegend}

  \begin{subfigure}[t]{0.5\textwidth}\centering
    \subfigtag{{top-left}}\phantomsubcaption%
\begin{tikzpicture}

\begin{axis}[
xmode=log,ymode=log,log basis x={10},log basis y={10},
xlabel={\(h_{\ell}\)},
xmin=8.05363715071347e-05, xmax=0.757858283255199,
ylabel={\(\var{\Zest_\ell}\)},
ymin=4.67876288362468e-10, ymax=0.633454602338942,
]

\addplot [mark=*]
table [x=gbm1-Wl-x0,y=gbm1-Vl-y0] {\loadedtable};

\addplot [mark=square*]
table [x=gbm1-Wl-x0,y=gbm1-Vl-y1] {\loadedtable};

\ifx\excludeTriangle\undefined
\addplot [mark=triangle*]
table [x=gbm1-Wl-x0,y=gbm1-Vl-y2] {\loadedtable};
\fi

\addplot [dashed, mark=o]
table [x=gbm1-Wl-x0,y=gbm1-Vl-y3] {\loadedtable};

\addplot [dashed, mark=square]
table [x=gbm1-Wl-x0,y=gbm1-Vl-y4] {\loadedtable};

\ifx\excludeTriangle\undefined
\addplot [dashed, mark=triangle]
table [x=gbm1-Wl-x0,y=gbm1-Vl-y5] {\loadedtable};
\fi

\addref[text align=center, raise=1.5ex]{\(h_{\ell}^{1/2}\)}{0.1 * x ^ (0.5)};

\addref[text align=left, raise=1.ex]{\(\hskip 1cm h_{\ell}\)}{0.1 * x};

\addref[text align=center, raise=-2.5ex]{\(h_{\ell}^{2}\)}{1e-2 * x^2};

\end{axis}

\end{tikzpicture}
     \label{fig:Vl-gbm-1d}
  \end{subfigure}\hfill
  \begin{subfigure}[t]{0.5\textwidth}\centering
    \subfigtag{{top-right}}\phantomsubcaption%
\begin{tikzpicture}

\begin{axis}[
xmode=log,ymode=log,log basis x={10},log basis y={10},
xlabel={\(h_{\ell}\)},
xmin=8.05363715071347e-05, xmax=0.757858283255199,
ylabel={\(\textrm{Work}\p{\Zest_{\ell}}\)},
ymin=1.19717350359321, ymax=95798.9795595827,
]
\addplot [mark=*]
table [x=gbm1-Wl-x0,y=gbm1-Wl-y0] {\loadedtable};

\addplot [mark=square*]
table [x=gbm1-Wl-x0,y=gbm1-Wl-y1] {\loadedtable};

\ifx\excludeTriangle\undefined
\addplot [mark=triangle*]
table [x=gbm1-Wl-x0,y=gbm1-Wl-y2] {\loadedtable};
\fi

\addref[text align=center, raise=1.5ex]{\(h_{\ell}^{-1} \log_{2}\p{h_{\ell}}\)}{ 2
  * x^(-1) * abs(ln(x))};
\addref[text align=center, raise=-2.5ex]{\(h_{\ell}^{-1}\)}{ 0.5 * x^(-1) };

\end{axis}

\end{tikzpicture}
     \label{fig:Wl-gbm-1d}
  \end{subfigure}

  \begin{subfigure}[t]{0.5\textwidth}\centering
    \phantomsubcaption%
\begin{tikzpicture}

\begin{axis}[
xmode=log,ymode=log,log basis x={10},log basis y={10},
xlabel={\(h_{\ell}\)},
xmin=8.05363715071347e-05, xmax=0.757858283255199,
ylabel={\(\textrm{Kurt}\sq{\Zest_{\ell}}\)},
ymin=0.608838797741259, ymax=304063.972588452,
]
\addplot [mark=*]
table [x=gbm1-Wl-x0,y=gbm1-kurt-y0] {\loadedtable};

\addplot [mark=square*]
table [x=gbm1-Wl-x0,y=gbm1-kurt-y1] {\loadedtable};

\ifx\excludeTriangle\undefined
\addplot [mark=triangle*]
table [x=gbm1-Wl-x0,y=gbm1-kurt-y2] {\loadedtable};
\fi

\addplot [dashed, mark=o]
table [x=gbm1-Wl-x0,y=gbm1-kurt-y3] {\loadedtable};

\addplot [dashed, mark=square]
table [x=gbm1-Wl-x0,y=gbm1-kurt-y4] {\loadedtable};

\ifx\excludeTriangle\undefined
\addplot [dashed, mark=triangle]
table [x=gbm1-Wl-x0,y=gbm1-kurt-y5] {\loadedtable};
\fi

\addref[text align=center, raise=1ex]{\(h_{\ell}^{-1}\)}{ 8e1 * x^-1 };
\addref[text align=left, raise=1ex]{\(\hskip 0.5cm h_{\ell}^{-1/2}  \)}{ 4e1 * x^(-0.5) };
\end{axis}

\end{tikzpicture}
     \label{fig:kurt-gbm-1d}
  \end{subfigure}\hfill
  \begin{subfigure}[t]{0.5\textwidth}\centering
    \subfigtag{{bottom-left}}\phantomsubcaption%
\begin{tikzpicture}

\begin{axis}[
xmode=log,ymode=log,log basis x={10},log basis y={10},
xlabel={Relative \(\varepsilon\)},
xmin=0.000177324031383333, xmax=0.0408491256583033,
ylabel={Total work \(\times \varepsilon^{2}\)},
ymin=2.30584217974168, ymax=149.816251324174,
]
\addplot[mark=*]
table [x=gbm1-work-x0,y=gbm1-work-y0] {\loadedtable};

\addplot[mark=square*]
table [x=gbm1-work-x1,y=gbm1-work-y1] {\loadedtable};

\ifx\excludeTriangle\undefined
\addplot [mark=triangle*]
table [x=gbm1-work-x2,y=gbm1-work-y2] {\loadedtable};
\fi

\addplot[dashed, mark=o]
table [x=gbm1-work-x3,y=gbm1-work-y3] {\loadedtable};

\addplot[dashed, mark=square]
table [x=gbm1-work-x4,y=gbm1-work-y4] {\loadedtable};

\ifx\excludeTriangle\undefined
\addplot[dashed, mark=triangle]
table [x=gbm1-work-x5,y=gbm1-work-y5] {\loadedtable};
\fi

\addref[text align=center, raise=1.ex]{\(\hskip 0.9cm \varepsilon^{-5/2}\)} {2 * x^(-0.5)};

\addref[text align=left, raise=1.5ex]{\(\hskip 0.3cm \varepsilon^{-2} \abs{\log\p{\varepsilon}}^{3}\)} {0.065 * abs(ln(x))^3};

\end{axis}

\end{tikzpicture}
     \label{fig:total-work-gbm-1d}
  \end{subfigure}
  \caption{ The GBM example in \cref{eq:gbm} for \(d{=}1\) when using
    Euler-Maruyama (\emph{solid}) and Milstein (\emph{dashed})
    \ifx\excludeTriangle\undefined%
      and different values of \(\eta\). \else%
      in the traditional, \(\Delta P_{\ell}\), and branching, \(\Zest_{\ell}\),
        estimators. \fi%
    \subref{fig:Vl-gbm-1d}
    {shows numerical verification of the variance
      convergence of \(\Zest_{\ell}\)}.
    \subref{fig:Wl-gbm-1d} The work estimate per sample, based on the number
    of generated samples from the standard normal distribution. The work
    estimates when using the Milstein scheme are identical.
    \subref{fig:kurt-gbm-1d} The kurtosis of \(\Zest_{\ell}\).
    \subref{fig:total-work-gbm-1d} The total work estimate of MLMC for
    different tolerances.
    This figure illustrates the improved computational complexity of MLMC when using the new branching estimator.
  }
  \label{fig:gbm-1d}
\end{figure}
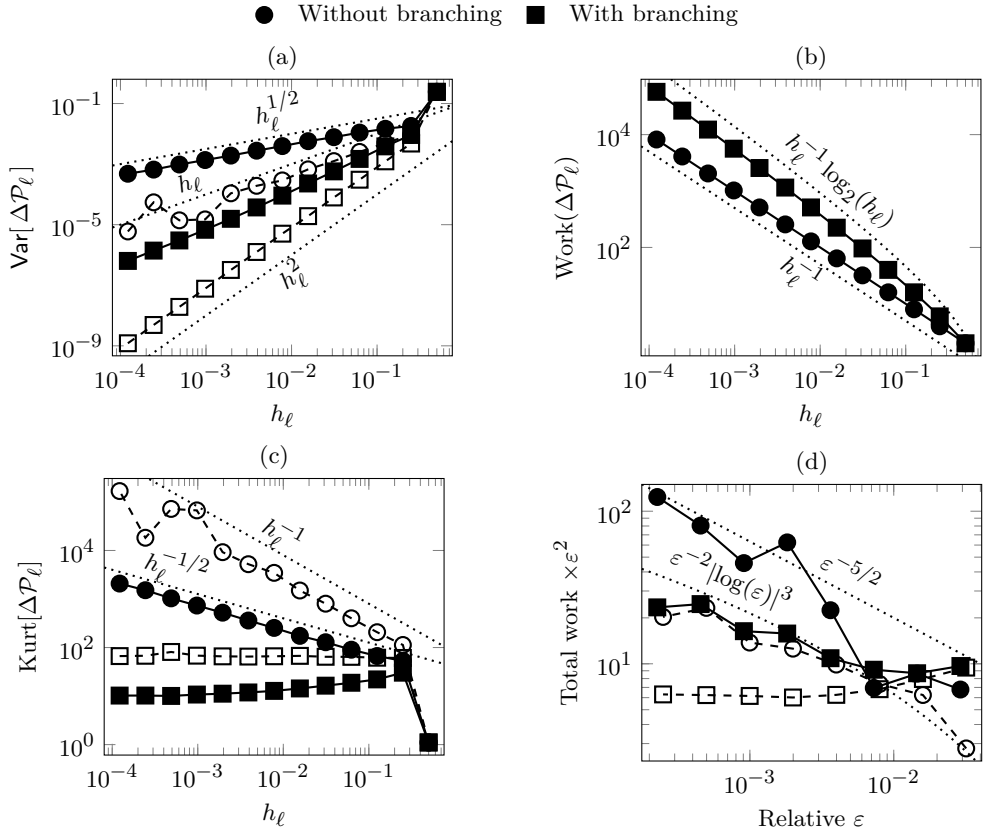

 \section{Antithetic Estimator}\label{sec:antithetic}%
As discussed in the previous section, the Milstein numerical scheme has faster
variance convergence than Euler-Maruyama leading to lower computational
complexity of MLMC estimators. However, for multi-dimensional SDEs, evaluating
the Milstein scheme requires expensive sampling of L\'{e}vy areas in most
cases. In \cite{giles:antithetic}, an antithetic estimator was introduced
which has the same MLMC variance convergence rate as a Milstein approximation
for smooth payoff functions \(f(x)\), but without requiring sampling of
L\'{e}vy areas. In this section, we analyse the corresponding branching
sampler for such an antithetic estimator.

Let \(\br{\p{\X[\ell]{t}}_{t=0}^{1}, \p{\Xa[\ell]{t}}_{t=0}^{1}}\) be an antithetic
pair which are identically distributed.
For example, \cite{giles:antithetic} presents such an antithetic estimator for
a Clark-Cameron SDE which is derived from a truncated Milstein discretization
by setting the L\'evy areas to zero. A similar branching estimator to
  \cref{def:est-main} can be defined by considering the triplet of approximate
  paths, \(\p{\X^{u}[\ell]{t}, \Xa^{u}[\ell]{t}, \X^{u}[\ell-1]{t}}_{t=0}^{1}\) for the
  same branch, \(\widetilde B^{u}\), of a Branching Brownian Motion. The
  antithetic estimator can then be defined as in \cref{eq:est-main} for
  \begin{equation}\label{eq:branching-antithetic-estimator}
  \Delta P_{\ell}^{u} =
  \begin{cases}
      \I{\X^{u}[0]{1} \in S} & \ell = 0,\\
    \frac{1}{2}\p[\big]{
      \I{\X^{u}[\ell]{1} \in S}+
      \I{\Xa^{u}[\ell]{1} \in S}} -
      \I{\X^{u}[\ell-1]{1} \in S} & \ell >0.
  \end{cases}
  \end{equation}
  Since the cost of sampling \(\Xa^{u}[\ell]{t}\) is the same as
  \(\X^{u}[\ell]{t}\), \cref{lem:estimator-workvar} still applies if
  \cref{ass:est-main} is satisfied for
  \cref{eq:branching-antithetic-estimator}.

In this section, we will impose well-motivated assumptions that allow us to
prove that \cref{ass:est-main} is satisfied for
\begin{equation}\label{eq:antithetic-estimator}
  \Delta P_{\ell} \defeq
  \begin{cases}
    \I{\X[\ell]{1} \in S} & \ell = 0,\\
    \frac{1}{2}\p[\big]{\I{\X[\ell]{1} \in S}+\I{\Xa[\ell]{1} \in S}} - \I{\X[\ell-1]{1} \in S}
                     & \ell >0.
  \end{cases}
\end{equation}
which has the same distribution as \(\Delta P_{\ell}^{u}\) in
\eqref{eq:branching-antithetic-estimator}.
In what follows, define
\[
  \begin{aligned}
    \g{t}\p{\xi} &\defeq \prob*{{\X{1} \in S} \given \X{t} = \xi}
    \\
    \text{and}\qquad \g[\ell]{t}\p{\xi} &\defeq \prob*{{\X[\ell]{1} \in S} \given \X[\ell]{t} = \xi}
    .
  \end{aligned}
\]
We will make the following assumptions

\begin{assumption}\label{ass:g-derv-bounds}%
  Let \(\nabla \g{t}\) and \(H_{\g{t}}\) be the Gradient and Hessian of \(\g{t}\),
  respectively. We assume that there exist constants %
  \(c_{1}, c_{2}>0\) such that for
  all \(\xi \in \rset^{d}\) and \(\tau \in \p*{0,1}\),
  \begin{subequations}\label{eq:g-derv-bound}
    \begin{align}
      \norm*{\nabla \g{1{-}\tau}\p{\xi}}  &\leq%
                                  \frac{c_1}{\tau^{1/2}} \exp\p*{ - c_2 \,
                                  \frac{
                                  \dist^2{\xi}
                                  }{\tau}}
                                  \label{eq:g-grad-bound}\\
      \text{and}\qquad \norm*{H_{\g{1{-}\tau}}\p{\xi}}  &\leq%
                                    \frac{c_1}{\tau} \exp\p*{ - c_2 \,
                                    \frac{\dist^2{\xi}}{\tau}},
                                    \label{eq:g-hessian-bound}
    \end{align}
  \end{subequations}
  and for any \(\tau \in (0,1]\),
  \begin{equation}\label{eq:dist-K-bound}
    \prob{ \dist{X_{1{-}\tau}} < \delta} \leq c_1 \delta.
  \end{equation}
\end{assumption}
This assumption is motivated by the case when \(\p{X_{t}}_{t \geq 0}\) is a
\(d\)-dimensional Wiener process, i.e., \(X_{1} \sim \mathcal{N}\p{\xi, \tau}\) and we prove it
in \cref{thm:cond-density-sde} for solutions of uniformly elliptic SDEs.

\begin{assumption}\label{ass:g-weak-conv}
  There exists a constant \(c_1 > 0\) such that for any \(\ell \in \nset\), \(\tau \in
  \p{0,1}\) and \(\xi \in \rset^{d}\),
  \begin{equation}\label{eq:g-weak-conv}
    \abs{{\g[\ell]{1{-}\tau}\p{\xi} - \g{1{-}\tau}\p{\xi}}} \leq
    \frac{c_{1} h_\ell}{\tau^{1/2}}.
  \end{equation}
\end{assumption}
This assumption is motivated by a result that was proved in
{{\cite[Theorem 2.3]{gobet:density}}} for an Euler-Maruyama scheme. In
particular,
letting \(\Gamma\p{\cdot, t; \xi, s}\) and \(\Gamma_{\ell}\p{\cdot, t; \xi, s}\) be the densities of
\(\X{t}\) and the Euler-Maruyama approximation, \(\X[\ell]{t}\), respectively,
given \(\X{s} {=} \X[\ell]{s} {=} \xi\), the authors in \cite{gobet:density} proved
that when the SDE is uniformly elliptic and the coefficients \(a, \sigma \in
C_{b}^{3,1}\) and \(\frac{\partial \sigma}{\partial t} \in C^{1,0}_{b}\) then for all \(\xi,x,y \in
\rset^{d}\), \(0 \leq s \leq t \leq 1\), and some constants \(c\) and \(C\)
\begin{equation*}
  \abs{\Gamma_{\ell}\p{x, t; \xi, s} - \Gamma\p{x, t; \xi, s}} \leq \frac{C \,
    h_{\ell}}{\p{t-s}^{\frac{d+1}{2}}} \exp\p*{- \frac{c \,\norm{x{-}\xi}^{2}}{t-s}}.
\end{equation*}
By setting \(t{=}1, s{=}1{-}\tau\) and integrating with respect to \(x\), \cref{eq:g-weak-conv} follows.
\detailed{\[
    \begin{aligned}
      &\abs*{\g[\ell]{1{-}\tau}\p{\xi} - \g{1{-}\tau}\p{\xi}} \\
      &\leq \int_{S} \; \abs*{ \Gamma_{\ell}\p{y,
        1; \xi, 1{-}\tau} -
        \Gamma\p{y, 1; \xi, 1{-}\tau}} \; \D y\\
      &\leq C \frac{h_{\ell}}{\tau^{\frac{d+1}{2}}} \int_{\rset^{d}}
        \exp\p*{-\frac{c \,\norm{\xi{-}y}^{2}}{\tau}} \; \D y\\
      &= C\: \p*{2\, c\, \pi}^{d/2} \: {h_{\ell}} \big/ {\tau^{1/2}}
    \end{aligned}
  \]}%
Note however that the SDE in the numerical example below does not satisfy the
uniform ellipticity condition of {\cite[Theorem 2.3]{gobet:density}} and a
slightly different numerical scheme is used in our case, namely the truncated
Milstein scheme without L\'evy areas which was proposed in
\cite{giles:antithetic}. %

\begin{theorem}\label{thm:Zest-antithetic-rates}
  Let \cref{ass:g-derv-bounds,ass:g-weak-conv} hold and assume further that
  for some \(q\geq1\)
  \begin{subequations}\label{ass:antithetic-conv}
    \begin{align}
      \label{eq:antithetic-strong-conv}
      \E*{\norm*{\X{1} -
      \X[\ell]{1}}^{q}}^{1/q} &\lesssim h_{\ell}^{1/2}\\
      \label{eq:antithetic-conv}
      \textrm{and}\qquad \E*{\norm*{\frac{1}{2} \p{\X[\ell]{1} + \Xa[\ell]{1}} -
      \X[\ell-1]{1}}^{q}}^{1/q} &\lesssim h_{\ell}.
    \end{align}
  \end{subequations}
  Then for \(\Delta P_{\ell}\) in \cref{eq:antithetic-estimator} there exists a
  constant \(c_{2}\) such that for all \(\tau \in \p{0, 1}\),
  \begin{subequations}
    \begin{align}
      \label{eq:antithetic-diag}\E{ \p{\Delta P_{\ell}}^{2} } &\leq c_{2} h_{\ell}^{\p{1-1/\p{q+1}}/2}\\
      \label{eq:antithetic-cross} \textrm{and}\qquad \E*{ \p*{\E{\Delta P_{\ell} \given
      \mathcal{F}_{1{-}\tau}}}^{2}} & \leq c_{2} h_{\ell}^{2\p{1-5/\p{q+5}}} / \tau^{3/2}.
    \end{align}
  \end{subequations}
\end{theorem}
This theorem shows that \cref{ass:est-main} is satisfied with
\(\diagbeta {=} \p{1{-}1/\p{q{+}1}}/2\), \(\crossbeta = 2
\p{1{-}5/\p{q{+}5}}\) and \(p{=}3/2\). Note that
\cref{eq:antithetic-strong-conv} is the same as \cref{ass:strong-conv} for
\(\beta{=}1\). For example, \cite[Theorem 4.13]{giles:antithetic} shows that
\cref{ass:antithetic-conv} is satisfied for all \(q{\geq}2\) and an antithetic
pair of estimators of a Clark-Cameron SDE derived from the Milstein
discretization by setting the L\'evy areas to zeros.
\begin{proof}
  The first claim \cref{eq:antithetic-diag} follows from a similar proof to
  \cref{thm:strong-conv-var-rates} given \cref{{eq:antithetic-strong-conv}}
  and \cref{eq:dist-K-bound}. To prove \cref{eq:antithetic-cross}, we start by
  defining \(E\) for a given \(\tau\) and some \(0 < r < 1\), that we will choose
  later, to be the set of paths for which
\[
  \begin{aligned}
    \max\br[\Bigg]{&\frac{\norm{\X[\ell]{1{-}\tau}{-}X_{1{-}\tau}}}{h_{\ell}^{r/2}},%
      \frac{\norm{\Xa[\ell]{1{-}\tau}{-}X_{1{-}\tau}}}{h_{\ell}^{r/2}},%
      \\
      &\frac{\norm{\X[\ell-1]{1{-}\tau}{-}X_{1{-}\tau}}}{h_{\ell}^{r/2}},%
      \frac{ \norm{ \half \p{\X[\ell]{1{-}\tau}\!{+}\Xa[\ell]{1{-}\tau}} - \X[\ell-1]{1{-}\tau} }
      }{h_{\ell}^{r}}} \geq 1.
  \end{aligned}
  \]
  Then since \(\abs{\Delta P_{\ell}}\leq 1\),
\[
  \begin{aligned}
    \E{ \E{\Delta P_{\ell} \given {\mathcal{F}}_{1{-}\tau}}^2}
    &= \E{ \I{E} \E{\Delta P_{\ell} \given
      {\mathcal{F}}_{1{-}\tau}}^2} + \E{ \I{E^c} \E{\Delta P_{\ell} \given {\mathcal{F}}_{1{-}\tau}}^2}\\
    &\leq \prob{E}+ \E{ \I{E^c} \p*{\E{\Delta P_{\ell} \given {\mathcal{F}}_{1{-}\tau}}}^2}.
  \end{aligned}
\]
Due to \cref{eq:antithetic-strong-conv} and the Markov inequality,
\[
  \prob*{\norm{\X[\ell]{1{-}\tau}{-}\X{1{-}\tau}} \geq h_{\ell}^{r/2}} \leq C_{q} h_\ell^{\p{1-r}q/2},
\]
for some constant \(C_q\) with similar bounds for
\(\norm{\X[\ell]{1{-}\tau}{-}\X{1{-}\tau}}\) and \(\norm{\Xa[\ell]{1{-}\tau}{-}\X{1{-}\tau}}\), and
\[
  \prob*{ \norm{\half\p{\X[\ell]{1{-}\tau}\!{+}\Xa[\ell]{1{-}\tau}}- \X[\ell-1]{1{-}\tau}} \geq
    h_\ell^{r}} \leq C_{q} h_\ell^{\p{1- r} q}.
\]
Hence \(\prob{E}\leq 4 C_{q} h_\ell^{\p{1-r}q/2}\).
For the other term we have
\begin{align*}
  \E{\I{E^c} \p*{\E{\Delta P_{\ell} \given \mathcal{F}_{1{-}\tau}}}^{2}}
  &= \E*{\I{E^c} \p*{\frac{1}{2}\p{\g[\ell]{1{-}\tau}\p{\X[\ell]{1{-}\tau}} +
    \g[\ell]{1{-}\tau}\p{\Xa[\ell]{1{-}\tau}}} - \g[\ell-1]{1{-}\tau}\p{\X[\ell-1]{1{-}\tau}}
    }^{2}} \\
  &\leq ~~3 \, \E*{\I{E^c} \p*{\frac{1}{2}\p{\g{1{-}\tau}\p{\X[\ell]{1{-}\tau}} + \g{1{-}\tau}\p{\Xa[\ell]{1{-}\tau}}}
    - \g{1{-}\tau}\p{\X[\ell-1]{1{-}\tau}}
    }^{2}} \\
  &~~+ 3 \,
    \E*{\I{E^c} \p{\g[\ell]{1{-}\tau}\p{\X[\ell]{1{-}\tau}} - \g{1{-}\tau}\p{\X[\ell]{1{-}\tau}} }^{2}}\\
  &~~+ 3 \: \E{\I{E^c} \p{\g{1{-}\tau}\p{\X[\ell-1]{1{-}\tau}} - \g[\ell-1]{1{-}\tau}\p{\X[\ell-1]{1{-}\tau}}}^{2}}.
\end{align*}
Due to \cref{eq:g-weak-conv} in \cref{ass:g-weak-conv}, the second of the
final three terms of this inequality is bounded by \(2\, c_1^2 h_\ell^2/\tau\) and
the third by \(2\, c_1^2 h_{\ell-1}^2/\tau\) so both are \(\Order{h_{\ell}^{2}/\tau}\).

To bound the first term we perform a Taylor series expansion about
\(\X[\ell-1]{1{-}\tau}\) to obtain
\begin{equation*}
  \half\p{\g{1{-}\tau}\p{\X[\ell]{1{-}\tau}}+\g{1{-}\tau}\p{\Xa[\ell]{1{-}\tau}}}
  -\g{1{-}\tau}\p{\X[\ell-1]{1{-}\tau}}
  = Y_{1} + \frac{1}{4}Y_{2} + \frac{1}{4}Y_{3},
\end{equation*}
where
\begin{eqnarray*}
  Y_1 & \defeq & \p*{ \half\p{\X[\ell]{1{-}\tau}+\Xa[\ell]{1{-}\tau}} -\X[\ell-1]{1{-}\tau}}
            . \nabla \g{1{-}\tau}\p{\X[\ell-1]{1{-}\tau}}, \\
  Y_2 & \defeq & \p{\X[\ell]{1{-}\tau}{-}\X[\ell-1]{1{-}\tau}}^T H_{\g{1{-}\tau}}\p{\xi_1}  \p{\X[\ell]{1{-}\tau}{-}\X[\ell-1]{1{-}\tau}}, \\
  Y_3 & \defeq & \p{\Xa[\ell]{1{-}\tau}{-}\X[\ell-1]{1{-}\tau}}^T H_{\g{1{-}\tau}}\p{\xi_{2}}
  \p{\Xa[\ell]{1{-}\tau}{-}\X[\ell-1]{1{-}\tau}},
\end{eqnarray*}
and where \(\xi_1\) is a positively weighted average of \(\X[\ell]{1{-}\tau}\)
and \(\X[\ell-1]{1{-}\tau}\) and \(\xi_2\) is a positively weighted average of
\(\Xa[\ell]{1{-}\tau}\) and \(\X[\ell-1]{1{-}\tau}\). Hence,
\[
  \E*{  \I{E^c}\p*{ \half\p{\g{1{-}\tau}\p{\X[\ell]{1{-}\tau}}+\g{1{-}\tau}\p{\Xa[\ell]{1{-}\tau}}} -\g{1{-}\tau}\p{\X[\ell-1]{1{-}\tau}} }^2 }
   \leq  2\ \E{\I{E^c}\p{Y_1^2 + Y_2^2 + Y_3^2}}.
\]

Defining \(\delta \defeq 2 %
\, \p{\tau/\p{h_{\ell}}}^{1/2} \, h_{\ell}^{r/2} \), we now split \(\I{E^c}\) into
\[
  \I{E^c} = \I{E^c}\I{\dist{X_{1{-}\tau}}>\delta} + \I{E^c}\I{\dist{X_{1{-}\tau}}<\delta}.
\]
Since \(h_\ell{\leq}\tau\) it follows that \(\delta {>} 2 \, h_{\ell}^{r/2}\) and if
\(\I{E^c}\I{\dist{X_{1{-}\tau}}>\delta}=1\) then \( \dist{\X[\ell]{1{-}\tau}}{>} \half \delta\),
\( \dist{\Xa[\ell]{1{-}\tau}} {>} \half \delta\), \(\dist{\X[\ell-1]{1{-}\tau}} {>} \half \delta\),
and also \(\dist{\xi_1} {>} \half \delta\), \(\dist{\xi_2} {>} \half \delta\). Therefore, by
\cref{ass:g-derv-bounds}, there is a constant \(C\) such that
\(\norm{\nabla \g{t}\p{ \X[\ell-1]{1-\tau}}}{\leq}C\) and \(\norm{ H_{\g{t}}\p{\xi_{i}}} {\leq} C\) for
\(i{=}1,2\) and all \(h_{\ell} {\leq} \tau\).
\begin{details}
\[
  \begin{aligned}
    \norm{\nabla \g{t}\p{ \X[\ell-1]{1-\tau}}}
    &\leq \frac{c_{1}}{\tau^{1/2}} \exp\p*{ - c_{2}
      \frac{\dist^{2}{\X[\ell-1]{1-\tau}}}{\tau} }\\
    &\leq \frac{c_{1}}{h_{\ell}^{1/2}} \exp\p*{ - c_{2} \frac{ \delta^{2} }{4 \tau}}\\
    &= \frac{c_{1}}{h_{\ell}^{1/2}} \exp\p*{ - c_{2} \frac{ \, \p{\tau/h_{\ell}} \,
      h_{\ell}^{r}}{\tau}}\\
    &\leq c_{1} h_{\ell}^{-1/2} \exp\p*{ - c_{2} \, h_{\ell}^{-\p{1-r}}}\\
    &\leq c_{1} \times \p*{\frac{\exp\p{-1}}{2 c_{2} \p{1-r}}}^{1/\p{2 \p{1-r}}}
  \end{aligned}
\]
similarly \(\norm{H_{g}\p{\xi}} \leq c_{1} \times \p*{\frac{\exp\p{-1}}{c_{2}
    \p{1-r}}}^{1/\p{1-r}}\).
\end{details}
Hence,
\[
  \E{\I{E^c}\I{\dist{X_{1{-}\tau}}>\delta} \p{Y_1^2+Y_2^2+Y_3^2}} \lesssim h_{\ell}^{2}.
\]
On the other hand, when \(\I{E^c}\I{\dist{X_{1{-}\tau}}<\delta}=1\),
\cref{ass:g-derv-bounds} implies that \(\norm{\nabla \g{1-\tau}\p{\xi}}\leq c_1/\tau^{1/2}\)
and \(\norm{H_{\g{1-\tau}}\p{\xi}}\leq c_1/\tau\) for any \(\xi \in \rset^{d}\) and \(\tau \in
\p{0,1}\) leading to
\[
  \begin{aligned}
    \E{\I{E^c}\I{\dist{X_{1{-}\tau}}<\delta} Y_1^2}
    &\leq \frac{c^2_1}{\tau} h_\ell^{2 r}
      \prob{\dist{X_{1{-}\tau}}<\delta} \leq 2 \frac{c^3_1}{\tau^{1/2}} h_\ell^{\p{5r-1}/2},
    \\
    \E{\I{E^c}\I{\dist{X_{1{-}\tau}}<\delta} \p{Y_2^2+Y_3^2}}
    &\leq \frac{c^2_1}{\tau^{2}}\,
      h_\ell^{2 r}\, \prob{\dist{X_{1{-}\tau}}<\delta} \leq 2 \frac{c^3_1}{\tau^{3/2}} \,
      h_\ell^{\p{5r-1}/2},
  \end{aligned}
\]
for any \(r<1\). Picking \(r\) so that \(\prob{E} \simeq h_{\ell}^{\p{5r-1}/2}\)
yields \(r=\p{q+1}/\p{q+5}\) and the final result.

\end{proof}

\begin{corollary}[MLMC Computational Complexity]\label{thm:mlmc-complexity-eta3/2}
  Under the assumptions of \cref{thm:Zest-antithetic-rates}, the MLMC method
  with MSE \(\varepsilon^2\) based on \(\Zest_{\ell}\) with \(\eta \in\p{1/2, 2}\), \(h_{\ell} =
  h_{0} M^{-\ell}\) for \(M \in \nset_{+}\) and the antithetic estimator
  \cref{eq:antithetic-estimator} has a computational complexity
  \(\Order{\varepsilon^{-2}}\).
\end{corollary}

Note that even though the MLMC estimator has the same computational complexity
for all values of \(\eta \in \p{1/2, 2}\), in theory the value \(\eta \approx 4/5\)
minimizes the work and variance of \(\Zest_{\ell}\) as discussed in
\cref{rem:optimal-eta}.

\subsection*{Numerical Experiments}\label{sec:antithetic-num}
In this section, we consider the Clark-Cameron SDE
\begin{equation}
  \label{eq:CC}
  \begin{aligned}
    \D \X_{1}{t} &= \D W_{1, t},\\
    \D \X_{2}{t} &= X_{1, t} \D W_{2, t}.
  \end{aligned}
\end{equation}
Here \(\br*{\p{W_{{i}, t}}_{t \geq 0}}_{i=1}^{2}\) are independent Wiener
processes.
Note that we can sample paths of \(\X_{1}{t} {=} W_{1,t}\) exactly. To
approximate the paths of \(\X_2{t}\), we use the Euler-Maruyama numerical
scheme \cite{kloden:numsde} as follows
\[
  \X_{2}[\ell]{\p{n+1} h_{\ell}} = \X_{2}[\ell]{n h_{\ell}} + W_{1,n h_{\ell}} \; \Delta_{\ell, n} W_{2},
\]
for \(n=0, \ldots, \p{h_{\ell}^{-1}-1}\) and
\[
  \Delta_{\ell, n} W_{i} \defeq W_{i, \p{n+1}\, h_{\ell}} - W_{i, n \, h_{\ell}}
\] is the Brownian increment. We again set the time step size at level \(\ell\)
as \(h_{\ell} {=} 2^{-\ell-1}\) and use the new estimator in \cref{def:est-main}
with \(\eta=1\).
We also test the antithetic estimators outlined in \cite{giles:antithetic}
which is obtained by setting the L\'evy area term in a Milstein discretization
to zero. In particular, the \(\nth{{\ell}}\) level approximation is defined as
\begin{equation}\label{eq:milstein-CC}
  \X_{2}[\ell]{\p{n+1} h_{\ell}} =
  \X_{2}[\ell]{n h_{\ell}} +
  W_{1,{n} h_{\ell}} \, \Delta_{\ell,n}W_{2} +
  \frac{1}{2} \, \Delta_{\ell,n}W_{1} \, \Delta_{\ell,n}W_{2}.
\end{equation}
When computing \cref{eq:antithetic-estimator} for a given Brownian path, the
coarse, \(\X_{i}[\ell-1]{\cdot}\), and fine, \(\X_{i}[\ell]{\cdot}\), approximations are
constructed according to \cref{eq:milstein-CC}. On the other hand, the
antithetic approximation at \(\p{2n+1} h_{\ell}\) and \(\p{2n+2} h_{\ell}\) for
\(n=0, 1, \ldots, \p{h_{\ell}^{-1}/2 -1}\) is defined as
\[
  \begin{aligned}
    \Xa_{2}[\ell]{\p{2n+1} h_{\ell}}
    &= \Xa_{2}[\ell]{2n h_{\ell}} + W_{1,2n h_{\ell}} \,
      \Delta_{\ell,2n+1}W_{2} +
    \frac{1}{2} %
      \, \Delta_{\ell,2n+1}W_{1} \, \Delta_{\ell,2n+1}W_{2},\\
    \Xa_{2}[\ell]{\p{2n+2} h_{\ell}} &=
                                 \Xa_{2}[\ell]{\p{2n+1} h_{\ell}} +
                                 \p*{W_{1,2n h_{\ell}} + \Delta_{\ell,2n+1}W_{1}} \,
                                 \Delta_{\ell,2n}W_{2}+ \frac{1}{2}
    \, \Delta_{\ell,2n}W_{1} \, \Delta_{\ell,2n}W_{2}.
  \end{aligned}
\]
In other words, the roles of \(\Delta_{\ell,n}W_{j}\) and \(\Delta_{\ell,n+1}W_{j}\) for
\(j{=}1,2\) are swapped when computing \(\br{\Xa_{2}[\ell]{\p{2n+1} h_{\ell}},
  \Xa_{2}[\ell]{\p{2n+2} h_{\ell}}}\) compared to \(\br{\X_{2}[\ell]{\p{2n+1} h_{\ell}},
  \X_{2}[\ell]{\p{2n+2} h_{\ell}}}\).

Under certain conditions on the coefficients of \cref{eq:sde}, the assumption
\Cref{eq:antithetic-strong-conv} is satisfied when using the Euler-Maruyama
scheme \cite{kloden:numsde} and both
\cref{eq:antithetic-strong-conv,eq:antithetic-conv} are satisfied for the
antithetic, truncated Milstein estimator \cite[Theorem
4.13]{giles:antithetic}. %
However, we emphasize that %
the diffusion coefficient in \cref{eq:CC} is not bounded, and more
importantly, is not elliptic. Hence the results of {{\cite[Theorem
    2.3]{gobet:density}}} showing \cref{eq:g-weak-conv} in
\cref{ass:g-weak-conv} are not applicable. Nevertheless, we first consider an
example where we compute \(\prob{X_{1}{ \in} S}\) where \(S = \br{x \in \rset^{d}
  : \minp{x_{1},x_{2}} {\geq} 1}\). For this example, the SDE in \cref{eq:CC} is
locally elliptic at the boundary of \(S\).

\cref{fig:tau-conv-CC-el1d} shows the convergence of \(\E{\p{\E*{\Delta P_{\ell}
      \given \mathcal{F}_{1-\tau}}}^{2}}\) for an Euler-Maruyama scheme, which as
predicted by \cref{thm:cond-density-sde} increases in proportion to
\(\tau^{-1/2}\), and an antithetic approximation, which as predicted by
\cref{thm:Zest-antithetic-rates} increases in proportion to
\(\tau^{-3/2}\) approximately%
.
\cref{fig:Vl-CC-el1d,fig:Wl-CC-el1d} confirm the claims of
\cref{lem:estimator-workvar}. %
\Cref{fig:total-work-CC-el1d} shows the total work estimate of an MLMC sampler
based on \(\Zest\) when using Euler-Maruyama or the antithetic estimator. The
computational complexities of MLMC based on the branching estimator
\(\Zest_{\ell}\) using Euler-Maruyama and the antithetic estimators are
consistent with \cref{thm:mlmc-complexity-eta1/2,thm:mlmc-complexity-eta3/2},
respectively. Recall that in this case, the optimal value of \(\eta\) is \(4/3\)
when using Euler-Maruyama, %
and \(4/5\) when using an antithetic approximation. However, similar to
\cref{sec:strong}, we did not observe a better computational cost when using
\(\eta \neq 1\) for the considered tolerances because of the additional cost of
branching when the branching points do not align with the time-stepping
scheme. For MLMC based on \(\Delta P_{\ell}\), labelled ``\tradestlabel'', the
computational complexity of MLMC is \(\Order{\varepsilon^{-5/2}}\) for both
Euler-Maruyama and the antithetic estimators since \(\var{\Delta P_{\ell}} \lesssim
h_{\ell}^{1/2}\); see \cref{thm:strong-conv-var-rates,thm:Zest-antithetic-rates}.
\cref{fig:kurt-CC-el1d} again illustrates that our branching estimator has
bounded kurtosis while the kurtosis of \(\Delta P_{\ell}\) grows approximately in
proportion to \(h_{\ell}^{-1/2}\). Hence an MLMC algorithm that relies on
variance estimates is more stable when using the branching estimator.

As a second test, we consider \(S = \br{x \in \rset^{d} : x_{1} \geq 1}\) for which
the diffusion coefficient is \emph{not} locally elliptic at the boundary.
\cref{fig:tau-conv-CC-1d} shows the convergence of \(\E{\p{\E*{\Delta P_{\ell} \given
      \mathcal{F}_{1-\tau}}}^{2}}\) and \cref{fig:Vl-CC-1d} shows the convergence of
\(\var{\Zest_{\ell}}\). The observed convergence rates are slightly worse than
those observed for the previous example. Nevertheless, recalling that the work
of \(\Zest\) increases in proportion to \(h_{\ell}^{-1}\abs{\log\p{h_{\ell}}}\), the
computational complexity of a MLMC estimator is still \(\Order{\varepsilon^{-2}}\) when
using the antithetic estimator, as confirmed in \cref{fig:total-work-CC-1d}.
This is a more difficult problem as \cref{fig:kurt-CC-1d} illustrates and the
branching estimator has the same increasing kurtosis as \(\Delta P_{\ell}\).

\begin{figure}\centering
  \begin{subfigure}[t]{0.5\textwidth}\centering
    \subfigtag{left}\phantomsubcaption%
\begin{tikzpicture}

\begin{axis}[
xmode=log,ymode=log,log basis x={10},log basis y={10},
xlabel={\(\tau\)},
xmin=0.000144164143240909, xmax=0.737134608645551,
ylabel={\(\E*{\p*{\E*{\Delta P_{\ell} \given \mathcal F_{1-\tau}}}^{2}}\)},
ymin=1.06998860503355e-12, ymax=1e-03,
]
\addplot [mark=square*]
table [x=gbm1-Wl-x0,y=CC_el1-tau-y0] {\loadedtable};

\addplot [mark=square, dashed]
table [x=gbm1-Wl-x0,y=CC_el1-tau-y2] {\loadedtable};

\addref[text align=center, raise=1ex]{\(\tau^{-1/2}\)}{1e-6 * x ^ (-0.5)};%
\addref[text align=center, raise=-3ex]{\(\tau^{-3/2}\)}{1e-12 * x ^ (-1.5)};%

\end{axis}

\end{tikzpicture}
     \label{fig:tau-conv-CC-el1d}
  \end{subfigure}\hfill
  \begin{subfigure}[t]{0.48\textwidth}\centering
    \subfigtag{right}\phantomsubcaption%
\begin{tikzpicture}

\begin{axis}[
xmode=log,ymode=log,log basis x={10},log basis y={10},
xlabel={\(\tau\)},
xmin=0.000144164143240909, xmax=0.737134608645551,
ylabel={\(\E*{\p*{\E*{\Delta P_{\ell} \given \mathcal F_{1-\tau}}}^{2}}\)},
ymin=1.48242338996018e-10, ymax=0.001,
]
\addplot [mark=square*]
table [x=gbm1-Wl-x0,y=CC1-tau-y0] {\loadedtable};

\addplot [mark=square, dashed]
table [x=gbm1-Wl-x0,y=CC1-tau-y2] {\loadedtable};

\addref[text align=center, raise=1ex]{\(\tau^{-1/2}\)}{1e-5 * x ^ (-0.5)};

\end{axis}

\end{tikzpicture}
     \label{fig:tau-conv-CC-1d}
  \end{subfigure}
  \caption{Numerical verification for \cref{eq:est-assumpt-cross} with \(h_{\ell}
    = 2^{-14}\) for the Clark-Cameron example in \cref{eq:CC} when using
    Euler-Maruyama (\emph{solid}) and antithetic Milstein (\emph{dashed}). For
    \subref{fig:tau-conv-CC-el1d}, \(S = \br{x \in \rset^{2} :
      \minp{x_{1},x_{2}} \geq 1}\) while for \subref{fig:tau-conv-CC-1d} we
    choose \(S = \br{x \in \rset^{2} : x_{2} \geq 1}\).}
  \label{fig:tau-conv-CC}
\end{figure}

\begin{figure}\centering
  \begin{customlegend}[legend columns=4,legend
    style={align=left,draw=none,column sep=2ex}]%
    \addlegendimage{only marks, mark=*};%
    \addlegendentry{\tradestlabel};%
    \addlegendimage{only marks, mark=square*};
    \ifx\excludeTriangle\undefined%
      \addlegendentry{\(\eta=1\)};%
      \addlegendimage{only marks, mark=triangle};%
      \addlegendentry{\(\eta=4/5\)};%
      \addlegendimage{only marks, mark=triangle*};%
      \addlegendentry{\(\eta=3/2\)};
    \else
      \addlegendentry{\newestlabel};
    \fi
  \end{customlegend}

  \begin{subfigure}[t]{0.5\textwidth}\centering
    \subfigtag{top-left}\phantomsubcaption%
\begin{tikzpicture}

\begin{axis}[
xmode=log,ymode=log,log basis x={10},log basis y={10},
xlabel={\(h_{\ell}\)},
xmin=8.05363715071347e-05, xmax=0.757858283255199,
ylabel={\(\var{\Zest_\ell}\)},
ymin=1.02280825454178e-09, ymax=0.0412356805960853,
]
\addplot [mark=*]
table [x=gbm1-Wl-x0,y=CC_el1-Vl-y0] {\loadedtable};

\addplot [mark=square*]
table [x=gbm1-Wl-x0,y=CC_el1-Vl-y1] {\loadedtable};

\ifx\excludeTriangle\undefined
\addplot [mark=triangle*]
table [x=gbm1-Wl-x0,y=CC_el1-Vl-y2] {\loadedtable};
\fi

\addplot [dashed, mark=o]
table [x=gbm1-Wl-x0,y=CC_el1-Vl-y3] {\loadedtable};

\addplot [dashed, mark=square]
table [x=gbm1-Wl-x0,y=CC_el1-Vl-y4] {\loadedtable};

\ifx\excludeTriangle\undefined
\addplot [dashed, mark=triangle]
table [x=gbm1-Wl-x0,y=CC_el1-Vl-y5] {\loadedtable};
\fi

\addref[text align=left, raise=1.5ex]{\(\hskip 0.7cm h_{\ell}^{1/2}\)}{5e-2 * x ^
  (0.5)};

\addref[text align=center, raise=-4ex]{\(h_{\ell}^{3/2}\)}{8e-4 * x^(1.5) };

\addref[text align=left, raise=1.3ex]{\(\hskip 0.3cm h_{\ell}\)}{2e-2 * x };
\end{axis}

\end{tikzpicture}
     \label{fig:Vl-CC-el1d}
  \end{subfigure}\hfill
  \begin{subfigure}[t]{0.5\textwidth}\centering
    \subfigtag{top-right}\phantomsubcaption%
\begin{tikzpicture}

\begin{axis}[
xmode=log,ymode=log,log basis x={2},log basis y={2},
xlabel={\(h_\ell\)},
xmin=8.05363715071347e-05, xmax=0.757858283255199,
ylabel={Work},
ymin=2.22446873808167, ymax=898601.974385946,
]
\addplot [mark=*]
table [x=gbm1-Wl-x0,y=CC_el1-Wl-y0] {\loadedtable};

\addplot [mark=square*]
table [x=gbm1-Wl-x0,y=CC_el1-Wl-y1] {\loadedtable};

\ifx\excludeTriangle\undefined
\addplot [mark=triangle*]
table [x=gbm1-Wl-x0,y=CC_el1-Wl-y2] {\loadedtable};
\fi

\addplot [dashed, mark=o]
table [x=gbm1-Wl-x0,y=CC_el1-Wl-y0] {\loadedtable};

\addplot [dashed, mark=square]
table [x=gbm1-Wl-x0,y=CC_el1-Wl-y1] {\loadedtable};

\ifx\excludeTriangle\undefined
\addplot [dashed, mark=triangle]
table [x=gbm1-Wl-x0,y=CC_el1-Wl-y3] {\loadedtable};
\fi

\addref[text align=center, raise=1.5ex]{\(h_{\ell}^{-1} \log_{2}\p{h_{\ell}}\)}
{ 9 * x^(-1) * abs(ln(x))};%
\addref[text align=center, raise=-2ex]{\(h_{\ell}^{-1}\)}{  x^(-1) };

\end{axis}

\end{tikzpicture}
     \label{fig:Wl-CC-el1d}
  \end{subfigure}

  \begin{subfigure}[t]{0.5\textwidth}\centering
    \phantomsubcaption%
\begin{tikzpicture}

\begin{axis}[
xmode=log,ymode=log,log basis x={10},log basis y={10},
xlabel={\(h_{\ell}\)},
xmin=8.05363715071347e-05, xmax=0.757858283255199,
ylabel={\(\textrm{Kurt}\sq{\Zest_{\ell}}\)},
ymin=21.9444711335824, ymax=6718.38594219222,
]
\addplot [mark=*]
table [x=gbm1-Wl-x0,y=CC_el1-kurt-y0] {\loadedtable};

\addplot [mark=square*]
table [x=gbm1-Wl-x0,y=CC_el1-kurt-y1] {\loadedtable};

\ifx\excludeTriangle\undefined
\addplot [mark=triangle*]
table [x=gbm1-Wl-x0,y=CC_el1-kurt-y2] {\loadedtable};
\fi

\addplot [dashed, mark=o]
table [x=gbm1-Wl-x0,y=CC_el1-kurt-y3] {\loadedtable};

\addplot [dashed, mark=square]
table [x=gbm1-Wl-x0,y=CC_el1-kurt-y4] {\loadedtable};

\ifx\excludeTriangle\undefined
\addplot [dashed, mark=triangle]
table [x=gbm1-Wl-x0,y=CC_el1-kurt-y5] {\loadedtable};
\fi

\addref[text align=center, raise=1ex]{\(h_{\ell}^{-1/2}\)}{ 8e1 * x^-(1/2) };

\end{axis}

\end{tikzpicture}
     \label{fig:kurt-CC-el1d}
  \end{subfigure}\hfill
  \begin{subfigure}[t]{0.5\textwidth}\centering
    \subfigtag{bottom-left}\phantomsubcaption%
\begin{tikzpicture}

\begin{axis}[
xmode=log,ymode=log,log basis x={10},log basis y={10},
xlabel={Relative \(\varepsilon\)},
xmin=0.00362763982489814, xmax=0.76479362918972,
ylabel={Total work \(\times \varepsilon^{2}\)},
ymin=0.433766234791354, ymax=121.965134936392,
]
\addplot [mark=*]
table [x=CC_el1-work-x0,y=CC_el1-work-y0] {\loadedtable};

\addplot [mark=square*]
table [x=CC_el1-work-x1,y=CC_el1-work-y1] {\loadedtable};

\ifx\excludeTriangle\undefined
\addplot [mark=triangle*]
table [x=CC_el1-work-x2,y=CC_el1-work-y2] {\loadedtable};
\fi

\addplot [dashed, mark=o]
table [x=CC_el1-work-x3,y=CC_el1-work-y3] {\loadedtable};

\addplot [dashed, mark=square]
table [x=CC_el1-work-x4,y=CC_el1-work-y4] {\loadedtable};

\ifx\excludeTriangle\undefined
\addplot [dashed, mark=triangle]
table [x=CC_el1-work-x5,y=CC_el1-work-y5] {\loadedtable};
\fi

\addref[text align=center, raise=1.1ex]{\(\hskip 1cm \varepsilon^{-5/2}\)} {5 * x^(-0.5)};

\end{axis}

\end{tikzpicture}
     \label{fig:total-work-CC-el1d}
  \end{subfigure}
  \caption{%
    Numerical results for Clark-Cameron example in \cref{eq:CC}
    when \(S = \br{x \in \rset^{2} :
      \minp{x_{1},x_{2}} \geq 1}\)
    and using Euler-Maruyama (\emph{solid}) or antithetic Milstein
    (\emph{dashed})
    \ifx\excludeTriangle\undefined%
      and different values of \(\eta\).
    \else
      in the traditional, \(\Delta P_{\ell}\), and branching, \(\Zest_{\ell}\), estimators.
    \fi
    \subref{fig:Vl-CC-el1d} {shows numerical verification of the variance
      convergence of \(\Zest_{\ell}\)} %
    \subref{fig:Wl-CC-el1d} {The work estimate per sample based on the number
      of generated samples from the standard normal distribution. The work
      estimates when using the Milstein scheme are identical.} %
    \subref{fig:kurt-CC-el1d} The kurtosis of \(\Zest_{\ell}\).
    \subref{fig:total-work-CC-el1d} {The total work estimate of MLMC for
      different tolerances.} %
  }
  \label{fig:CC-el1d}
\end{figure}

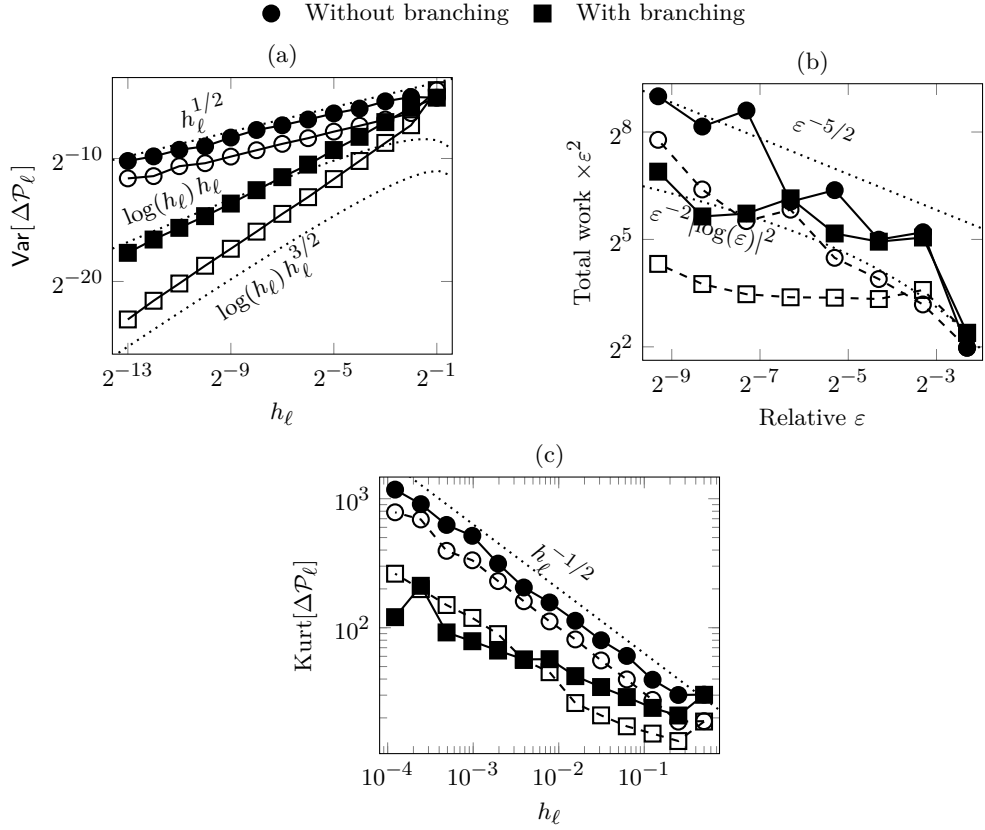
\begin{figure}
  \centering
  \begin{customlegend}[legend columns=4,legend
    style={align=left,draw=none,column sep=2ex}]%
    \addlegendimage{only marks, mark=*};%
    \addlegendentry{\tradestlabel};%
    \addlegendimage{only marks, mark=square*};%
    \ifx\excludeTriangle\undefined
      \addlegendentry{\(\eta=1\)};%
      \addlegendimage{only marks, mark=triangle*};%
      \addlegendentry{\(\eta=4/5\)};
      \addlegendimage{only marks, mark=triangle};%
      \addlegendentry{\(\eta=3/2\)};
    \else
      \addlegendentry{\newestlabel};
    \fi
  \end{customlegend}

\begin{subfigure}[t]{0.5\textwidth}\centering
  \subfigtag{left}\phantomsubcaption%
\begin{tikzpicture}

\begin{axis}[
xmode=log,ymode=log,log basis x={2},log basis y={2},
xlabel={\(h_{\ell}\)},
xmin=8.05363715071347e-05, xmax=0.757858283255199,
ylabel={\(\var{\Zest_\ell}\)},
ymin=1.59065363193936e-08, ymax=0.0935516343434528,
]
\addplot [mark=*]
table [x=gbm1-Wl-x0,y=CC1-Vl-y0] {\loadedtable};

\addplot [mark=square*]
table [x=gbm1-Wl-x0,y=CC1-Vl-y1] {\loadedtable};

\ifx\excludeTriangle\undefined
\addplot [mark=triangle*]
table [x=gbm1-Wl-x0,y=CC1-Vl-y2] {\loadedtable};
\fi

\addplot[mark=o]
table [x=gbm1-Wl-x0,y=CC1-Vl-y3] {\loadedtable};

\addplot [mark=square]
table [x=gbm1-Wl-x0,y=CC1-Vl-y4] {\loadedtable};

\ifx\excludeTriangle\undefined
\addplot [mark=triangle]
table [x=gbm1-Wl-x0,y=CC1-Vl-y5] {\loadedtable};
\fi

\addref[text align=left, raise=1.5ex]{\(\hskip 1.cm h_{\ell}^{1/2}\)}{1e-1 * x ^
  (0.5)};

\addref[text align=left, raise=-3.0ex]{\(\hskip 1.5cm \log\p{h_{\ell}} \,
  h_{\ell}^{3/2}\)}{2e-3 * x^(1.5) * abs(ln(x))};

\addref[text align=left, raise=1.5ex]{\(\hskip 0.3cm \log\p{h_{\ell}} \,
h_{\ell}\)}{8e-3 * x * abs(ln(x))};
\end{axis}

\end{tikzpicture}
   \label{fig:Vl-CC-1d}
\end{subfigure}\hfill
\begin{subfigure}[t]{0.5\textwidth}\centering
  \subfigtag{right}\phantomsubcaption%
\begin{tikzpicture}

\begin{axis}[
xmode=log,ymode=log,log basis x={2},log basis y={2},
xlabel={Relative \(\varepsilon\)},
xmin=0.00123767062206974, xmax=0.257710886565372,
ylabel={Total work \(\times \varepsilon^{2}\)},
ymin=3.08282640155086, ymax=651.561657323945,
]
\addplot [mark=*]
table [x=CC1-work-x0,y=CC1-work-y0] {\loadedtable};

\addplot [mark=square*]
table [x=CC1-work-x1,y=CC1-work-y1] {\loadedtable};

\ifx\excludeTriangle\undefined
\addplot [mark=triangle*]
table [x=CC1-work-x2,y=CC1-work-y2] {\loadedtable};
\fi

\addplot [mark=o, dashed]
table [x=CC1-work-x3,y=CC1-work-y3] {\loadedtable};

\addplot [mark=square, dashed]
table [x=CC1-work-x4,y=CC1-work-y4] {\loadedtable};

\ifx\excludeTriangle\undefined
\addplot [mark=triangle, dashed]
table [x=CC1-work-x5,y=CC1-work-y5] {\loadedtable};
\fi

\addref[text align=center, raise=2ex]{\(\varepsilon^{-5/2}\)} {20 * x^(-0.5)};

\addref[text align=left, raise=-3ex]{\(\hskip 0.2cm \varepsilon^{-2} \abs{\log\p{\varepsilon}}^{2}\)} {2 * abs(ln(x))^2};

\end{axis}

\end{tikzpicture}
   \label{fig:total-work-CC-1d}
\end{subfigure}

  \begin{subfigure}[t]{0.5\textwidth}\centering
    \phantomsubcaption%
\begin{tikzpicture}

\begin{axis}[
xmode=log,ymode=log,log basis x={10},log basis y={10},
xlabel={\(h_{\ell}\)},
xmin=8.05363715071347e-05, xmax=0.757858283255199,
ylabel={\(\textrm{Kurt}\sq{\Zest_{\ell}}\)},
ymin=10.5687022178315, ymax=1472.78023298378,
]
\addplot [mark=*]
table [x=gbm1-Wl-x0,y=CC1-kurt-y0] {\loadedtable};

\addplot [mark=square*]
table [x=gbm1-Wl-x0,y=CC1-kurt-y1] {\loadedtable};

\ifx\excludeTriangle\undefined
\addplot [mark=triangle*]
table [x=gbm1-Wl-x0,y=CC1-kurt-y2] {\loadedtable};
\fi

\addplot [dashed, mark=o]
table [x=gbm1-Wl-x0,y=CC1-kurt-y3] {\loadedtable};

\addplot [dashed, mark=square]
table [x=gbm1-Wl-x0,y=CC1-kurt-y4] {\loadedtable};

\ifx\excludeTriangle\undefined
\addplot [dashed, mark=triangle]
table [x=gbm1-Wl-x0,y=CC1-kurt-y5] {\loadedtable};
\fi

\addref[text align=center, raise=1ex]{\(h_{\ell}^{-1/2}\)}{ 20 * x^-0.5 };

\end{axis}

\end{tikzpicture}
     \label{fig:kurt-CC-1d}
  \end{subfigure}
  \caption{{ The Clark-Cameron example in \cref{eq:CC} for \(S = \br{x \in
        \rset^{2} : x_{2} \geq 1}\) when using Euler-Maruyama (\emph{solid}) or
      Milstein (\emph{dashed})
      \ifx\excludeTriangle\undefined%
        and different values of \(\eta\).
      \else
        in the traditional, \(\Delta P_{\ell}\), and branching, \(\Zest_{\ell}\), estimators.
      \fi
      \subref{fig:Vl-CC-1d} {shows numerical verification of the variance
        convergence of \(\Zest_{\ell}\)}. \subref{fig:total-work-CC-1d} shows the
      total work estimate of MLMC for different tolerances.
      \subref{fig:kurt-CC-1d} The kurtosis of \(\Zest_{\ell}\).}}
  \label{fig:CC-1d}
\end{figure}
 \clearpage
\section{Bounds on Solutions of Elliptic SDEs}\label{sec:sde-bounds}%
In this section, we prove \cref{ass:cond-density,ass:g-derv-bounds} for
solutions to SDEs with certain conditions on the SDE coefficients and the
boundary \(K\).
For any \(x \equiv \p{x_{i}}_{i=1}^{d} \in \rset^{d}\), define \(x_{-j} \equiv
\p{x_{i}}_{i=1, i\neq j}^{d} \in \rset^{d-1}\) and define \(\exp{x}\), \(\log{x}\)
and \(x^{-1}\) component-wise. For \(m \in \nset^{d}\), define \(\abs{m} = m_{1}
+ \ldots + m_{d}\) and \(\Diff{\xi}^{m} \equiv \frac{\partial^{m_{1}}}{\partial \xi_{1}^{m_{1}}}\ldots
\frac{\partial^{m_{d}}}{\partial \xi_d^{m_{d}}}\).
 Given a set \(J \subset
\rset^{d}\), define
\begin{equation}
  \label{eq:dist-set-def}
  J^{\delta} \equiv \br{y \in \rset^{d} \::\: \dist[J]{y} \le \delta},
\end{equation}
and for a function \(u : \rset^{d} \to \rset^{d}\), let \(u\p{J}\) denote the
image of \(J\) under the mapping \(x \to u\p{x}\), i.e., \( u\p{J} \equiv \br{u\p{x}
  \: : \: x \in J}%
\). %

We now define a class of ``Simple'' sets which are a particularly simple form
of Lipschitz boundaries.
\begin{definition}[\texorpdfstring{\niceset{}}{(Si)} sets]\label{def:nicesets}
  We say that a set \(J \subset \rset^{d}\) is an \niceset{} set if it is the graph
  of a Lipschitz function. In other words, there exists an index \(j\in\br{1, \ldots,
    d}\) and a Lipschitz function \(f:\rset^{d-1} \to \rset\) such that
  \[
    J = \br{x \in \rset^{d} \: : \: x_j = f\p*{x_{-j}}}.
  \]
\end{definition}

\begin{lemma}\label{thm:bounded-K-prob}
  Let \(\br{Z_{i}}_{i=1}^{d}\) be a set of i.i.d. Gaussian random variables
  with \(\var{Z_{i}} = \tau\) for all \(i\) and denote \(Z=\p{Z_{i}}_{i=1}^{d}\).
  If \(J \subset \rset^{d}\) is an \niceset{} set, then there exists a constant
  \(C>0\) such that
  \[
    \prob{\dist[J]{Z} \leq \delta} \leq C \times \frac{\delta}{\tau^{1/2}}.
  \]
\end{lemma}
\begin{proof}
  For the \niceset{} set \(J\) with corresponding index \(j\) and Lipschitz function \(f\) with
  Lipschitz constant \(L\), we first show that \(J^{\delta} \subseteq \widetilde{J^\delta} \)
  where %
  \[
    \widetilde{J^\delta} = \br{ x \in \rset^d : \abs{x_j - f\p{x_{-j}}} \leq \p{L{+}1} \delta
    }.
  \]
  Letting \(y \in J\) and \(x \in J^{\delta}\) such that \(y_{j} = f\p{y_{-j}}\)
  and \(\norm{x{-}y} \leq \delta\). It follows that
  \[
    \begin{aligned}
      \abs{f\p{ x_{-j}} - x_{j}}
      &\leq \abs{f\p{x_{-j}} - f\p{y_{-j}}} +
        \abs{x_{j}- y_{j}} \\
      &\leq L \norm{ x_{-j} - y_{-j}} + \abs{x_{j}- y_{j}}\\
      &\leq \p{L{+}1}\, \norm{x{-}y}\\
      &\leq \p{L{+}1} \, \delta.
    \end{aligned}
  \]
  Then%
  \[
      \prob{Z \in J^{\delta}} \leq \prob{Z \in {\widetilde{J^\delta}}}
      =
      \E*{\prob{\abs{Z_{j} - f\p{Z_{-j}}} \leq \p{L{+}1}\, \delta \given\, Z_{-j}}}.
  \]
  Using standard 1D results on \(Z_{j}\) yields
  \[
    \prob*{\abs{Z_{j} {-} f\p{Z_{-j}}} \leq \p{L{+}1}\, \delta \given Z_{-j}} \ \leq\
    \frac{2\p{L{+}1}}{\p{2\pi}^{1/2}} \times \frac{\delta}{\tau^{1/2}},
  \]
  and the result follows.
\end{proof}

The previous lemma can be used to show that \cref{ass:cond-density} is
satisfied for a Wiener process, i.e., \(X_{t}=W_{t}\), and a set \(S\) whose
boundary \(\partial S \equiv K\) is \niceset{}. We next prove a more general result
showing both \cref{ass:cond-density,ass:g-derv-bounds} for sets whose boundary
can be covered by \niceset{} sets and SDEs whose coefficients satisfy certain
smoothness conditions.

\begin{theorem}\label{thm:cond-density-sde}
  For \(S \subset \rset^{d}\), assume that \(\partial S \equiv K \subseteq \bigcup_{j=1}^{n} J_{j}\) for some
  finite \(n\) and \niceset{} sets \(\br{J_{j}}_{j=1}^{n}\). Assume that the
  SDE \eqref{eq:sde} is uniformly elliptic and that \(a, \sigma \sigma^{T}\) are
  \(\lambda\)-H\"older continuous in space uniformly with respect to time and let
  \(\p{X_{t}}_{t \in \sq{0,1}}\) satisfy the SDE. Then, there exist \(c_{1}>0\)
  such that for all \(0 < s < 1\) and all \(\delta > 0\) the following holds
  \begin{equation}
    \E*{\p[\big]{\prob{\dist{\X{1}} \leq \delta \given \mathcal{F}_{s}}}^{2}} \leq c_{1} \: \frac{\delta^{2}}{\p{1{-}s}^{1/2}}.
  \end{equation}
  Assume further that \(a, \sigma \sigma^{T} \in C_{b}^{2,0}\), then there exists \(c_{1},
  c_{2}>0\) such that for all \(\xi \in \rset^{d}\), \(0 < s < 1\) and \(m \in
  \nset^{d}, 0 \le \abs{m} \leq 2\),
  \begin{equation}\label{eq:derv-cond-sde-general}
    \abs{ \Diff{\xi}^{m} \prob*{X_{1} {\in} S \given X_{s}{=}\xi} } \leq \frac{c_{1}}{\p{1{-}s}^{\abs{m}/2}}.
    \exp\p*{-c_{2} \frac{\dist^2{\xi}}{\p{1{-}s}}}.
  \end{equation}
\end{theorem}
The assumption on \(S\) (or \(K\)) is illustrated in
\cref{fig:K-assumption-illustration}. \cref{thm:cond-density-sde} shows that
\cref{ass:cond-density,ass:g-derv-bounds} are satisfied for a solution to a
uniformly elliptic SDE assuming that the set \(K\) is covered by a finite
number of \niceset{} sets.
\begin{proof}
  We have the following bound on \(\Gamma\p{\cdot, 1 ; \xi, s}\), the density of
  \(X_{1}\) given \(X_{s}=\xi\) for \(m \in \nset^{d}\) and some \(C_{m}, c_{m}>0\),
  \begin{equation}\label{eq:density-bound}
    \abs{\Diff{\xi}^{m} \Gamma\p{x, 1 ; \xi, s}} \leq \frac{C_{m}}{\p{1{-}s}^{\p{d+\abs{m}}/2}}
    \exp\p*{- c_{m} \frac{\norm{x-\xi}^{2}}{1-s}},
  \end{equation}
  when the SDE \eqref{eq:sde} is uniformly elliptic and, for \(\abs{m}=0\),
  when \(a, \sigma \sigma^{T}\) are \(\lambda\)-H\"older continuous in space uniformly with
  respect to times \cite[Chapter 9, Theorem 2%
  ]{friedman:partial} and for \(0 \leq \abs{m} \leq 2\) when \(a, \sigma \sigma^{T} \in
  C_{b}^{2, 0}\) \cite[Chapter 9, Theorem 7%
  ]{friedman:partial}.
  Hence
  \[
    \begin{aligned}
      \prob{\dist{\X{1}} \leq \delta \given \mathcal{F}_{s}}
      &= \prob{\dist{\X{1}} \leq
        \delta
        \given X_{s}}\\
      &= \int_{\rset^{d}} \I{\dist{x} \leq \delta} \, \Gamma\p{x, 1; X_{s}, s} \D x\\
      &\leq \int_{\rset^{d}} \I{\dist{x} \leq \delta} \, \p*{\frac{C_{0}}{\p{1{-}s}^{d/2}} \exp\p*{-c_{0} \frac{\norm{x-X_{s}}^{2}}{1-s}}} \D x \\
      &= C_{0} \; \p{2 \pi/c_{0}}^{d/2} \; \prob{\dist{Z} \leq \delta \given X_{s}},
    \end{aligned}
  \]
  where \(Z\) is a multivariate Normal random variable with mean \(X_{s}\)
  with variance \(\p{\p{1{-}s}/c_{0}} I_{d}\) where \(I_{d}\) is the \(d \times d\)
  identity matrix. Then noting
  \[
    \prob{\dist{Z} \leq \delta \given X_{s}} \leq \sum_{j=1}^{n} \prob{\dist[J_j]{Z} \leq \delta
      \given X_{s}},
  \]
  and using \cref{thm:bounded-K-prob} we can conclude that there is a constant
  \(\widetilde C\)
  \[
    \prob{\dist{\X{1}} \leq \delta \given \mathcal{F}_{s}} \leq n\, \widetilde C\,
    \frac{\delta}{\p{1{-}s}^{1/2}}.
  \]
  Hence
  \[
    \begin{aligned}
      \E{\p*{\prob{\dist{\X{1}} \leq \delta \given \mathcal{F}_{s}}}^{2}}
      &\leq
        n\, \widetilde C \frac{\delta}{\p{1{-}s}^{1/2}} \, \E{\prob{\dist{\X{1}}
        \leq \delta \given \mathcal{F}_{s}}} \\
      &\leq n^{2}\, \widetilde C^{2} \frac{\delta^{2}}{\p{1{-}s}^{1/2}}.
    \end{aligned}
  \]
  To prove \cref{eq:derv-cond-sde-general}, we distinguish between two cases
  \begin{enumerate}
  \item \(\xi \notin S\), then
    \[
      \begin{aligned}
        &\abs{\Diff{\xi}^{m} \prob*{X_{1} \in S \given X_{s} = \xi} }\\
        &\leq \int_{S} \abs{\eval{\Diff{\xi}^{m}\Gamma\p{x, 1 ; \xi, s}}} \D x\\
        &\leq \int_{S} \frac{C_{m}}{\p{1{-}s}^{\p{d+\abs{m}}/2}} \times \exp\p*{- c_{m}
          \frac{\norm{x-\xi}^{2}}{1{-}s}} \D x \\
        &\leq
          \frac{1}{\p{1{-}s}^{\abs{m}/2}}
          \exp\p*{- c_{m} \frac{\inf_{x \in S} \norm{x-\xi}^{2}}{2 \p{1{-}s}}} \int_{S}
          \frac{C_{m}}{\p{1{-}s}^{d/2}} \times \exp\p*{- c_{m}
          \frac{\norm{x-\xi}^{2}}{2 \p{1{-}s}}} \D x \\%
        &\leq
          \frac{1}{\p{1-s}^{\abs{m}/2}}
          \exp\p*{- c_{m} \frac{\inf_{x \in S} \norm{x-\xi}^{2}}{2 \p{1{-}s}}}
          \p{\p{4 \pi}^{d/2}C_{m}},
      \end{aligned}
    \]
    and we conclude with
    \[
      \inf_{x \in S} \norm{x-\xi}^{2} =
      \inf_{x \in \partial S} \norm{x-\xi}^{2} = \dist^2{\xi},
    \]
    since \(\xi \notin S\).
  \item \(\xi \in S\), then, for \(S^{c}\) being the compliment of \(S\), we have
    \[
      \begin{aligned}
        \abs{\Diff{\xi}^{m} \prob*{X_{1} {\in} S \given X_{s}{=}\xi}} \detailed{&=
                                                                  \abs{\Diff{\xi}^{m} \p{1 - \prob*{X_{1} {\in} S^{c} \given X_{s}{=}\xi}}}\\%
        } &=
            \abs{\Diff{\xi}^{m} \p{\prob*{X_{1} {\in} S^{c} \given X_{s}{=}\xi}}},
      \end{aligned}
    \]
    and we can use the first step since \(\xi \notin S^{c}\) and \(\partial S^{c} = \partial S\).
  \end{enumerate}
\end{proof}

\begin{figure}
  \centering
  \begin{subfigure}[t]{0.5\textwidth}\centering
    \subfigtag{{top-left}}\phantomsubcaption
\begin{tikzpicture}
  \begin{axis}[xmin=-1.1,xmax=1.1,ymin=-1.1,ymax=1.1, axis x line=center,
    axis y line=center, axis equal image,xticklabels={},yticklabels={},
    legend style={draw=none,fill=none,at={(0,1)},anchor=north east}]

    \draw [fill=lightgray, opacity=0.5] (0,0) circle (1);

    \draw[dashed, very thick] (-1,0 |- 135:1) -- (135:1) arc (135:45:1) --
    (45:1 -| 1,0); \addlegendentry{\(f_{1}\p{x}\)};%
    \addlegendimage{dashed, very thick};

    \draw[dotted, very thick] (0,1 -| 135:1) -- (135:1) arc (135:225:1) --
    (225:1 |- 0,-1); \addlegendentry{\(f_{2}\p{y}\)};%
    \addlegendimage{dotted, very thick};%
  \end{axis}
\end{tikzpicture}
     \label{fig:valid-K}
  \end{subfigure}\hfill
  \begin{subfigure}[t]{0.5\textwidth}\centering
    \subfigtag{{top-left}}\phantomsubcaption
    \pgfplotsset{width=5.1cm}
\begin{tikzpicture}
  \begin{polaraxis}[ticks=none, xticklabel={ \pgfmathparse{\tick/180}
      \pgfmathifisint{\pgfmathresult}{$\pgfmathprintnumber[int
        detect]{\pgfmathresult}\pi$}%
      {$\pgfmathprintnumber[frac,frac denom=6,frac
        whole=false]{\pgfmathresult}\pi$} }]
    \begin{scope}[domain=0:(36*pi), samples=600]
      \addplot [name path=B, data cs = polarrad] { (1+x/pi)^(-0.9)};%
      \addplot [name path=A, data cs = polarrad] { (2+x/pi)^(-0.9)};
    \end{scope}
    \tikzfillbetween[of=A and B,on layer=]{lightgray,opacity=0.5};
  \end{polaraxis}
\end{tikzpicture}
     \pgfplotsset{width=6.1cm}%
    \label{fig:invalid-K}
  \end{subfigure}

  \caption{\subref{fig:valid-K} The set \(S = \br{\p{x,y} \in \rset^{2} \: : \:
      x^{2}+y^{2} \leq 1}\) whose boundary \(\partial S \equiv K\) satisfies the assumptions
    of \cref{thm:cond-density-sde}. We split the circle on the boundary into
    four parts and we show two of them here. %
    \subref{fig:invalid-K} The set \(S = \br{\p{r,\theta} \in
      \rset_{+}^{2} : \p{2+\theta/\pi}^{-0.9}\leq r \leq \p{1+\theta/\pi}^{-0.9}}\), in polar
    coordinates, whose boundary does not satisfy the assumptions of
    \cref{thm:cond-density-sde}.}
  \label{fig:K-assumption-illustration}
\end{figure}

As an example of a set that does not satisfy the assumptions of
\cref{thm:cond-density-sde}, consider \( K = \br{1/n^{b} : n \in \nset} \) for
\(0 < b \leq 1\). Then we can show that \cref{ass:cond-density} is not satisfied
for a standard Normal random variable \(Z\) and any \(\delta \leq b\),
\[
  \begin{aligned}
    \prob*{\min_{n \in \nset} \abs*{Z - {n^{-b}}} \leq \delta} &= 2 \int_{0}^{\infty} \I{\min_{n
                                                       \in \nset} \abs{y - {n^{-b}}} \leq \delta} \
                                                       \phi\p*{y} \D y\\
    \detailed{&\geq \frac{2 \exp\p{-1/2}}{\p{2 \pi}^{1/2}} \int_{0}^{1} \I{\min_{n \in
                \nset} \abs{y - {n^{-b}}} \leq \delta} \D y\\%
    } &\geq \frac{2
        \exp\p{-1/2}}{\p{2 \pi}^{1/2}} \int_{0}^{\p{\delta/b}^{{b}/\p{b+1}}} \I{\min_{n \in
        \nset} \abs{y
        - {n^{-b}}} \leq \delta} \D y\\
                                                     &= \frac{2 \exp\p{-1/2}}{\p{2 \pi}^{1/2}} \, \p{\delta/b}^{{b}/\p{b+1}},
  \end{aligned}
\]
where \(\phi\p{\cdot}\) is the density of standard normal random variable.
\detailed{To justify the last inequality first note that for some \(n\), the
  distance between two points on \(K\) is \(n^{-b} - \p{n+1}^{-b}\). When
  \(\delta\) is larger than that distance, and since subsequent terms have smaller
  distances, the indicator from \(0\) to \(n^{-b}\) is 1. Hence we need to
  find the largest \(n^{-b}\) (or smallest \(n\)) such that
  \[
    n^{-b} - \p{n+1}^{-b} \leq \delta
  \]
  Recall
  \[
    \p*{{n^{-b}} - {\p{n+1}^{-b}} }^{\frac{b}{b+1}} \leq
    {b^{\frac{b}{b+1}}} \ {n^{-b}}
  \]
  To show this, simplify the inequality to
  \[
    n \p*{1 - {\p{1 + 1/n}^{-b}}} \leq b
  \]
  then letting \(x=1/n \in [0,1]\), the function \( x^{-1} \p{1 - \p{1 +
      x}^{-b}} \) is decreasing over \(x \in [0,1]\). For \(x=0\) the limit is
  \(b\). Hence we pick the smallest \(n\) for which
  \[
    n^{-b} \leq \p{\delta/b}^{\frac{b}{b+1}}
  \]
}
Similarly, consider the two dimensional set in polar
coordinates %
\(K = \br{\p{r,\theta} \in \rset_{+} {\times} \sq{0, 2\pi} \: : \: r {=}
  \p*{n{+}{\theta/\pi}}^{-b}, n \in \nset}\) (see \cref{fig:invalid-K}) for some \(b \in
\p{0,1}\). Using a similar calculation to before we can show that for a 2D
standard normal random variable, \(Z\), any \(\delta<b\)
\[
  \prob*{\min_{y \in K} \norm*{Z - y} \leq \delta} \geq \frac{1}{2}\exp\p{-1/2} \,
  \p{\delta/b}^{2b/\p{b+1}}.
\]
\detailed{%
  \[
    \begin{aligned}
      \prob*{\min_{y \in K} \norm*{Z - y} \leq \delta} &= \int_{0}^{2\pi} \int_{0}^{\infty} \I{\square \leq
        \delta}
      \phi\p{r}\, r \,\D r \D \theta\\
      &\geq \frac{\exp\p{-1/2}}{2 \pi} \int_{0}^{2\pi} \int_{0}^{1} \I{\square \leq \delta} \, r \,\D r
      \D \theta
    \end{aligned}
  \]
  For a fixed \(\theta\), the distance between two points is \(\p{\theta/\pi + n}^{-b} -
  \p{\theta/\pi + n+1}^{-b}\). When \(\delta\) is larger than that distance for some
  \(n\), and since subsequent terms have smaller distances, the indicator from
  \(r=0\) to \(r=\p{\theta/\pi + n}^{-b}\) for that
  \(n\) %
  is 1. Hence we need to find the largest \(\p{n+\theta/\pi}^{-b}\) (or smallest
  \(n\)) such that
  \[
    \p{\theta/\pi + n}^{-b} - \p{\theta/\pi + n+1}^{-b} \leq \delta
  \]
  Like before, we impose (for \(\theta \in \sq{0,2 \pi}\))
  \[
    {\p{\theta/\pi + n}^{-b}} \leq \p{\delta/b}^{\frac{b}{b+1}}
  \]
}

The Geometric Brownian Motion (GBM) does not satisfy the conditions of
\cref{thm:cond-density-sde} since the diffusion coefficient of the SDE of a
GBM is not bounded nor uniformly elliptic. To deal with this important case,
we first prove a similar result to \cref{thm:bounded-K-prob} for log-normal
random variables.

\begin{lemma}\label{thm:bounded-prob-lognormal}
  Let \(\br{Z_{i}, Y_{i}}_{i=1}^{d}\) be two sets of independent Gaussian
  random variables with \(\var{Z_{i}} = \tau\) for all \(i\) and denote \(Z=
  \p{Z_{i}}_{i=1}^{d}, Y= \p{Y_{i}}_{i=1}^{d}\). Let \(J \subset \rset^{d}\) be an
  \niceset{} set. There exists \(\delta_0{>}0\) and \(C{>}0\) such that for all
  \(0{<}\delta {<}\delta_0\)
  \[
    \E{\, \prob{ \dist[\exp\!{J}]{\exp \p{Z + Y}} < \delta \ \given \ Y}^2 } \leq C \,
    \frac{\delta^2}{\tau^{1/2}}.
  \]
\end{lemma}
\begin{proof}
  If \(v \equiv \exp x \in \p{\exp\!{J}}^\delta\) then there exists
  \(w \equiv \exp y \in \exp\!J\) such that
  \(\norm{ v {-} w } \leq \delta\). If, in addition, \(\norm{x} < R_\delta \equiv \abs{\log
    \delta}^{3/4}\) then for sufficiently small \(\delta_0\), \(2\delta {<}v_{i} {<}
  (2\delta)^{-1}\) for all \(\delta {<}\delta_0\) and all \(i = 1,\ldots, d\), and hence
  \(\norm{w^{-1}} < 2\, \norm{v^{-1} }\),
from which it follows that
\[
  \norm{ x {-} y } = \norm{ \log v {-} \log w }
  \leq 2 \, \norm{v^{-1}}\, \delta.
\]

  \detailed{[To explain that last line, for each component index \(i\), we have
    \(\abs{x_i} \leq \abs{\log \delta}^{3/4}\) hence \(-\abs{\log \delta}^{3/4} \leq x_{i} \le
    \abs{\log \delta}^{3/4}\) hence \(\exp\p{-\abs{\log \delta}^{3/4}} \leq v_{i} \le
    \exp\p{\abs{\log \delta}^{3/4}}\) or \(\delta^{1/\abs{\log \delta}^{1/4}} \leq v_{i} \le
    \delta^{-1/\abs{\log \delta}^{1/4}}\). Finally, we argue that \(\delta^{1/\abs{\log
        \delta}^{1/4}} \geq 2\delta\) for sufficiently small \(\delta < \delta_{0}\). Then \(2\delta \leq
    v_{i} \leq \p{2\delta}^{-1}\). Additionally by the reverse triangle inequality
    \(w_i \geq v_i - \abs{w_{i}-v_{i}} \geq v_i - \delta \geq v_{i}/2\).]}

If \(J\) has corresponding index \(j\) and Lipschitz function \(f\) with Lipschitz constant \(L\),
then
\begin{equation}
  \abs{ x_{j} - f\p{x_{-j}} }
  \ \leq\ \abs{ f\p{x_{-j}} - f\p{y_{-j}} } + \abs{x_{j} - y_{j}}
  \ \leq\ \p{L{+}1}\, \norm{ x {-} y }
  \ \leq\ 2\, \p{L{+}1} \norm{v^{-1}}\, \delta.
 \label{eq:mike}
\end{equation}
Since \(\norm{x} < R_\delta\), \(\exp\p{\norm{x}} \, \delta \leq \half\), and
so \(\abs{ x_{j} - f\p{x_{-j}} } \leq L{+}1\), and therefore
\[
  \abs{x_{j}} \ \leq\ L+1 + \abs{f\p{x_{-j}}} \ \leq\ L+1 + \abs{f\p{0}} + L \norm{x_{-j}}.
\]
Hence, since \(\norm{x} \leq \norm{x_{-j}} + \abs{x_{j}}\) and \( \norm{v^{-1}} \leq \exp\p{\norm{x}} \),
\begin{eqnarray*}
  \abs{ x_{j} - f\p{x_{-j}} }
  &\leq&  2\, \p{L{+}1}\, \exp (\norm{x_{-j}} + \abs{x_{j}})\, \delta \\
  &\leq&  2\, \p{L{+}1}\, \exp\p{L{+}1 {+} \abs{f\p{0}}}\, \exp(\p{L{+}1}\norm{x_{-j}})\ \delta.
\end{eqnarray*}
The conclusion is that if \(\exp{x} \in \p{\exp J}^\delta\) and \(\norm{x} < R_\delta\)
then \(x \in \widetilde{J^\delta}\) where \(\widetilde{J^\delta}\) is defined as
\[
  \widetilde{J^\delta} = \br{ x \in \rset^{d} : \abs{ x_{j} - f\p{x_{-j}} } \leq c
    \exp(\p{L{+}1}\norm{x_{-j}})\ \delta },
\]
with \(c\equiv 2\, \p{L{+}1}\, \exp\p{L{+}1 + \abs{f\p{0}}}\).

\vspace{0.1in}

Letting \(U \equiv Z+Y\), to bound \(
\E{\, \p{\E{\I{\exp\!U \in \p{\exp\!J}^\delta} \,\given\, Y}}^2 } \), we start by
noting that\\ \(\displaystyle \I{\exp\!U \in \p{\exp\!J}^\delta} \leq \I{\exp\!U \in
  \p{\exp\!J}^\delta} \I{\norm{U}< R_\delta} + \I{\norm{U}\geq R_\delta},\) and also that
\(\E{\I{\norm{U}\geq R_\delta} \,\given\, Y }<1\). Hence,
\begin{eqnarray*}
  \E{\, \E{\I{\exp\!U \in \p{\exp\!J}^\delta} \,\given\, Y }^2}
  &\leq& 2\, \E{\, \E{\I{\exp\!U \in \p{\exp\!J}^\delta} \I{\norm{U}<R_\delta} \,\given\, Y }^2} \\
  &&  +\ 2\, \E{\, \E{\I{\norm{U}\geq R_\delta} \,\given\, Y }^2}\\&\leq
  & 2\, \E{\,\E{\I{ U \in \widetilde{J^\delta}} \,\given\, Y }^2}
    + 2\, \prob{\norm{U} {\geq} R_\delta}.
\end{eqnarray*}
Then, by normality of \(U\) and
\[
  \begin{aligned}
    \E{\,\E{\I{U \in \widetilde{J^\delta}} \,\given\, Y }^2}
    \detailed{&=
                \E*{\,\p*{\E{\,\E{\I{U \in \widetilde{J^\delta}} \,\given\, Y, Z_{-j}}
                \given \, Y}}^2}\\
              &\leq\E*{\,{\E*{\,\p*{\E{\I{U \in \widetilde{J^\delta}} \,\given\, Y, Z_{-j}}}^{2}
                \given \, Y}}}\\
              &\leq\E*{\,{\E*{\,\p*{\E{\I{U \in \widetilde{J^\delta}} \,\given\, Y, Z_{-j}}}^{2}
                \given \, Y_{-j}, Z_{-j}}}}\\
    }
              &\leq \E{\,\E{\E{\I{U \in \widetilde{J^\delta}} \,\given\, Y, Z_{-j} }^2
                \given Y_{-j}, Z_{-j} }}.
  \end{aligned}
\]
Using standard 1D results on \(Z_{j}\),
\[
  \E*{\E*{\I{U \in \widetilde{J^\delta}} \,\given\, Y, Z_{-j} }^2 \given Y_{-j},
    Z_{-j} } \leq c^{2} \, \exp\p*{2 \p{L{+}1}\norm{U_{-j}}}\
  \frac{\delta^{2}}{\tau^{1/2}},
\]
and we can conclude
\[
  \E*{\,\E{\I{U \in \widetilde{J^\delta}} \,\given\, Y }^2} \leq c^{2} \,
  \E*{\exp\p{2\p*{L{+}1}\norm{U_{-j}}}}\ \frac{\delta^{2}}{\tau^{1/2}},
\]
where \(U_{-j}\) is a \(\p{d{-}1}\)-dimensional Normal random variable for
which \(\E{\exp\p{2\p{L{+}1}\norm{U_{-j}}}}\) is finite. The final result is
obtained by noting that \(\prob{\norm{U} \geq R_\delta}=\order{\delta^2}\) due to the
definition of \(R_\delta\) and standard asymptotic results for a \(d\)-dimensional
Normal random variable.
\end{proof}

The previous lemma can be used to show that \cref{ass:cond-density} is
satisfied for a process \(X_{t} = \exp\p{W_{t}}\) where \(\p*{W_{t}}_{t \geq 0}\)
is a Wiener process and a set \(S\) whose boundary is \(K\) and the set
\(\log\p{K}\) is \niceset{}. We next prove a more general result showing
\cref{ass:cond-density} for sets whose boundary can be covered by exponentials
of \niceset{} sets and processes that can be written as exponentials of
solutions of uniformly elliptic SDEs.

\begin{theorem}\label{thm:cond-density-sde-gbm}
  For \(S \subset \rset^{d}\) assume that \(\partial S \equiv K \subseteq \bigcup_{j=1}^{n} \exp {J_{j}}\) for
  some finite \(n\) and \niceset{} sets \(\br{J_{j}}_{j=1}^{n}\). Assume that
  the SDE \cref{eq:sde} is uniformly elliptic and \(a,\sigma\sigma^{T}\) are
  \(\lambda\)-H\"older continuous in space uniformly with respect to time and let
  \(\br{Y_{t}}_{t \in \sq{0,1}}\) satisfy the SDE. Then, there exists \(C>0\)
  such that for all \(0 < s < 1\) the following holds
  \begin{equation}
    \E*{\p[\big]{\prob{\dist{\exp{Y_{1}}} \leq \delta \given \mathcal{F}_{s}}}^{2}} \leq C \: \frac{\delta^{2}}{\p{1-s}^{1/2}}.
  \end{equation}
\end{theorem}
\begin{proof}
  By the assumptions on the coefficients of \cref{eq:sde} and \cite[Chapter 9,
  Theorem 2]{friedman:partial} the density \(\Gamma\p{\cdot, 1; \xi, s}\) of \(Y_{1}\)
  given \(Y_{s}=\xi\) for \(s \leq 1\) exists and the upper bound
  \cref{eq:density-bound} holds for \(\abs{m} = 0\). Hence,
  \[
    \begin{aligned}
      \prob{\dist{\exp{Y_{1}}} \leq \delta \given \mathcal{F}_{s}}
      &= \prob{\dist{\exp{Y_{1}}} \leq \delta \given Y_{s}}\\
      &= \int_{\rset^{d}} \I{\dist{\exp {y}} \leq \delta} \, \Gamma\p{y, 1; Y_{s}, s} \D y\\
      &\leq \int_{\rset^{d}} \I{\dist{\exp {y}} \leq \delta} \, \p*{\frac{C_{0}}{\p{1-s}^{d/2}}
        \exp\p*{-c_{0} \frac{\abs{y-Y_{s}}^{2}}{1-s}}} \D y \\
      &= C_{0} \; \p{2 \pi/c_{0}}^{d/2} \; \prob{\dist{\exp\p{Z + Y_{s}}} \leq \delta \given
        Y_{s}},
    \end{aligned}
  \]
  where \(Z\) is a multivariate Normal random variable with zero mean and
  variance \(\p{\p{1-s}/c_{0}} I_{d}\) where \(I_{d}\) is the \(d \times d\)
  identity matrix. Similarly, using \cref{eq:density-bound} on the density of
  \(Y_{s}\),
  \begin{eqnarray*}
    \lefteqn{\E{\p{\prob{\dist{\exp\p{Z + Y_{s}}} \leq \delta \given
    Y_{s}}}^{2}}}
    \\& \leq C_{0} \; \p{2 \pi/c_{0}}^{d/2} \, \E{\p{\prob{\dist{\exp\p{Z + Y}} \leq \delta \given
    Y}}^{2}},
  \end{eqnarray*}
  where \(Y\) is a multivariate Normal random variable with variance
  \(\p{s/c_{0}} I_{d}\). Then noting
  \begin{eqnarray*}
    \lefteqn{\E{\p{\prob{\dist{\exp\p{Z + Y}} \leq \delta \given Y}}^{2}}} \\
    &\leq& n \sum_{j=1}^{n} \E{\p{\prob{\dist[\exp\!{J_{j}}]{\exp\p{Z + Y}} \leq
        \delta \given Y}}^{2}},
  \end{eqnarray*}
  and using \cref{thm:bounded-prob-lognormal} we obtain the result.
\end{proof}

 \section{Conclusion}%
In this article we have developed a new Monte Carlo estimator based on the
branching of approximate solution paths of the underlying stochastic
differential equation. Under certain assumptions, the new estimator, when
combined with MLMC, can be used to compute digital options with an improved
computational complexity. %
{Future directions for analysis could include
  extending \cref{thm:Zest-antithetic-rates} to the case of exponentials of
  solutions to uniformly elliptic SDEs, bounding higher moments of the error,
  particularly for the case of the antithetic estimators similar to
  \cref{sec:kurtosis} and extending the analysis to the case of solutions of
  locally elliptic SDEs to justify the numerical results in
  \cref{sec:antithetic}.}

There are also many applications that could benefit from the new estimator and the
branching ideas presented above.
First, instead of computing a single probability, the new estimator can be
used to compute multiple probabilities to reconstruct the cumulative (and
probability) density functions. This would provide an alternative approach to
the smoothing approach used in \cite{giles:smoothing}.

When MLMC is used together with the pathwise sensitivity approach (also known
as IPA, Infinitesimal Perturbation Analysis) to evaluate financial sensitivities
known collectively as Greeks, the loss of smoothness due to differentiation of
the payoff function affects the computational complexity \cite{burgos:greeks};
the branching estimator could significantly alleviate this.  Similarly, the
branching estimator could be used in combination with the finite difference
(or ``bumping'') approach to computing Greeks to counteract the increase in
the variance that results when decreasing the bump magnitude.

A final observation is that branching could also be used when the underlying
model is a parabolic stochastic PDE instead of an SDE.

\begin{appendix}
  \section{Bounding the kurtosis of the branching estimator}\label{sec:kurtosis}
\def\Q{Q}

The objective of this appendix is to prove that the kurtosis of the branching
estimator for the Euler-Maruyama and Milstein discretisations is
\(\order{h^{-\nu}}\) for any \(\nu>0\), for an elliptic SDE with a boundary set
\(K\) for which there exists a constant \(C\) such that
\[
  \prob{\dist{\X{1}}\leq\delta \given \mathcal F_{1-\tau}} \leq C\, \delta / \tau^{1/2},
\]
see also \cref{ass:cond-density,thm:cond-density-sde}. If we define \(\Q =
2^{\hat{\ell}}\), and number the particles as indicated in
\cref{fig:branching-estimator}, then noting that \(\abs{\Delta
  P_{\ell}^{\p{i}}}^n=\abs{\Delta P_{\ell}^{\p{i}}}\) for \(n{=}2,3,4\), the fourth
moment of the branching estimator is bounded by
\begin{equation}
\begin{aligned}
  \E{\p{\Zest_{\ell}}^{4}} &\leq \Q^{-4} \sum_{i,j,k,m=1}^{\Q}
                          \E{\abs{\Delta P_{\ell}^{\p{i}}}
                          \abs{\Delta P_{\ell}^{\p{j}}}
                          \abs{\Delta P_{\ell}^{\p{k}}}
                          \abs{\Delta P_{\ell}^{\p{m}}} }\\
                        &= 12 \, \Q^{-4} \sum_{i<j<k<m}
                          \E{\abs{\Delta P_{\ell}^{\p{i}}}
                          \abs{\Delta P_{\ell}^{\p{j}}}
                          \abs{\Delta P_{\ell}^{\p{k}}}
                          \abs{\Delta P_{\ell}^{\p{m}}} }\\
                        & ~ + 36 \, \Q^{-4} \sum_{i<j<k}
                          \E{\abs{\Delta P_{\ell}^{\p{i}}}
                          \abs{\Delta P_{\ell}^{\p{j}}}
                          \abs{\Delta P_{\ell}^{\p{k}}} } \\
                        & ~ + 14 \, \Q^{-4}\, \sum_{i<j}
                          \E{\abs{\Delta P_{\ell}^{\p{i}}}
                          \abs{\Delta P_{\ell}^{\p{j}}} } \\
                        & ~ + \Q^{-4}\ \sum_{i}
                          \E{\abs{\Delta P_{\ell}^{\p{i}}} }.
\end{aligned}
\label{eq:split}
\end{equation}

\begin{details}
  In more details: $\Q\p{\Q{-}1}(\Q{-}2)\p{\Q{-}2}$ quads all different,
  $6\Q\p{\Q{-}1}(\Q{-}2)$ with 3 different, $3\Q\p{\Q{-}1}(\Q{-}2)$ with 2
  pairs, $4\Q\p{\Q{-}1}(\Q{-}2)$ with 3 same, $\Q$ all the same
\end{details}

To begin with, we focus attention on the case with four distinct indices \(i{<}j{<}k{<}m\),
as this is the most common case.  There are 5 different branching patters among these,
but in each case through repeated use of
\[
\E{\I{\dist{\X^{\p{i}}{1}}\leq\delta} \,
    \I{\dist{\X^{\p{j}}{1}}\leq\delta} \given \mathcal F_{1-\tau_{i,j}}}
     \leq
C\, \delta\, \tau_{i,j}^{-1/2} \, \E{\I{\dist{\X^{\p{i}}{1}}\leq\delta}
       \given \mathcal F_{1-\tau_{i,j}}},
\]
where \(1{-}\tau_{i,j}\) is the time at which the particles
  \(\X^{\p{i}}{1}\) and \(\X^{\p{j}}{1}\) separate, we obtain
\[
  \E{\I{\dist{\X^{\p{i}}{1}}\leq\delta} \, \I{\dist{\X^{\p{j}}{1}}\leq\delta} \,
    \I{\dist{\X^{\p{k}}{1}}\leq\delta} \, \I{\dist{\X^{\p{m}}{1}}\leq\delta}} \leq C^4\, \delta^4\,
  \tau_{i,j}^{-1/2} \tau_{j,k}^{-1/2} \tau_{k,m}^{-1/2}.
\]

If we define the extreme set \(E\) to be those cases for which
\[
  \max_{n\in\{i,j,k,m\}} \max\br*{ \norm*{ \X^{\p{n}}{1}-\X^{\p{n}}[\ell]{1} } ,
    \norm*{\X^{\p{n}}{1}-\X^{\p{n}}[\ell-1]{1} } } \geq \delta,
\]
for some \(\delta{>}0\), then
\[
  \prob{E} \leq 4 \p*{ \prob*{\norm*{\X{1}-\X[\ell]{1}}\geq \delta}
    +\prob*{\norm*{\X{1}-\X[\ell-1]{1}}\geq\delta} } \lesssim h_\ell^{\beta q/2}\, \delta^{-q},
\]
due to the usual Markov inequality based on the \(q\)-th moment of the strong error being bounded.
We also have
\[
\abs{\Delta P_{\ell}^{\p{i}}}
\abs{\Delta P_{\ell}^{\p{j}}}
\abs{\Delta P_{\ell}^{\p{k}}}
\abs{\Delta P_{\ell}^{\p{m}}}
\leq
\I{\dist{\X^{\p{i}}{1}}\leq\delta}
       \I{\dist{\X^{\p{j}}{1}}\leq\delta}
       \I{\dist{\X^{\p{k}}{1}}\leq\delta}
       \I{\dist{\X^{\p{m}}{1}}\leq\delta}
+ \I{E}.
\]
So it follows that
\[
  \begin{aligned}
    \E{\abs{\Delta P_{\ell}^{\p{i}}}
    \abs{\Delta P_{\ell}^{\p{j}}}
    \abs{\Delta P_{\ell}^{\p{k}}}
    \abs{\Delta P_{\ell}^{\p{m}}}}
    ~ &\lesssim ~
    \delta^4 \tau_{i,j}^{-1/2} \tau_{j,k}^{-1/2} \tau_{k,m}^{-1/2}
    +
    h_\ell^{\beta q/2}\, \delta^{-q}\\
    ~ &\lesssim ~
    h_\ell^{2\beta q / \p{4+q}} \tau_{i,j}^{-1/2} \tau_{j,k}^{-1/2} \tau_{k,m}^{-1/2},
  \end{aligned}
\]
by choosing \(\delta^{2} \simeq h_\ell^{\beta q / \p{4+q}}\). For any fixed \(i\), provided \(\eta
< 2 \),
\[
 \sum_{j\neq i}\tau_{i,j}^{-1/2} = \tau_0^{-1/2} \sum_{\ell'=0}^{\hat{\ell}-1} 2^{\hat{\ell}-1-\ell'} 2^{\eta\ell'/2}
 \simeq 2^{\hat{\ell}} = \Q.
\]
This gives us
\begin{multline*}
  \Q^{-4} \sum_{i<j<k<m}
  \E{\abs{\Delta P_{\ell}^{\p{i}}}
    \abs{\Delta P_{\ell}^{\p{j}}}
    \abs{\Delta P_{\ell}^{\p{k}}}
    \abs{\Delta P_{\ell}^{\p{m}}} }
  \lesssim
  h_\ell^{2\beta q / \p{4+q}}  \Q^{-4} \sum_{i<j<k<m} \tau_{i,j}^{-1/2} \tau_{j,k}^{-1/2} \tau_{k,m}^{-1/2}
  \\
  \begin{aligned}
    &\lesssim h_\ell^{\beta q / \p{4+q}} \Q^{-4}  \sum_{i<j<k} \tau_{i,j}^{-1/2} \tau_{j,k}^{-1/2}
      \p*{ \sum_{m\neq k}\tau_{k,m}^{-1/2}}
    \\   &\lesssim h_\ell^{2\beta q / \p{4+q}} \Q^{-3}  \sum_{i<j<k} \tau_{i,j}^{-1/2} \tau_{j,k}^{-1/2}
    \\   &\lesssim h_\ell^{2\beta q / \p{4+q}} \Q^{-3}  \sum_{i<j} \tau_{i,j}^{-1/2}
           \p*{ \sum_{k\neq j}\tau_{j,k}^{-1/2} }
    \\   &\lesssim h_\ell^{2\beta q / \p{4+q}} \Q^{-2}  \sum_{i<j} \tau_{i,j}^{-1/2}
    \\   &\lesssim h_\ell^{2\beta q / \p{4+q}} \Q^{-2}  \sum_{i}
           \p*{ \sum_{j\neq i}\tau_{i,j}^{-1/2} }
    \\   &\lesssim h_\ell^{2\beta q / \p{4+q}}.
  \end{aligned}
\end{multline*}

Further analysis following the same approach proves that this is the
dominant contribution in \cref{eq:split}, and hence
\[
  \E{\p{\Zest_{\ell}}^{4}} \lesssim h_\ell^{2\beta q / \p{4+q}}.
\]
Similar analysis, or referring to
\cref{lem:estimator-workvar,thm:strong-conv-var-rates}, shows that the second
moment is bounded as follows
\[
  \E{\p{\Zest_{\ell}}^{2}} \lesssim h_\ell^{\beta q / \p{2+q}}.
\]
If we assume the second moment has a lower bound of \(\Order{h_\ell^{\beta}}\) then it follows that the kurtosis is bounded by
\[
  \text{Kurt}\sq{\Zest_{\ell}} \lesssim h_\ell^{- 8 \beta / \p{4+q} },
\]
for all \(q{>}2\), and therefore is \(\order{h_\ell^{-\nu}}\) for any \(\nu{>}0\).

The analysis can be extended to the exponential SDE case by first expanding the
extreme set \(E\) to include cases in which
\[
  \max_{n\in\{i,j,k,m\}} \norm{ \log\X^{\p{n}}{1} } \geq R_\delta,
\]
where \(R_\delta \defeq \abs{\log\delta}^{3/4} \) as defined previously in the proof of
Lemma 5.4. Equation \cref{eq:mike} in that proof gives us
\[
\E*{\I{\dist{\X{1}}\leq\delta}\I{\log\X{1}\leq R_\delta} \given \mathcal F_{1-\tau}}
\lesssim \exp\p{R_\delta} \, \delta / \tau^{1/2},
\]
so then we obtain
\detailed{for the non-extreme paths}
\[
  \begin{aligned}
    \E{\I{\dist{\X^{\p{i}}{1}}\leq\delta} \, \I{\dist{\X^{\p{j}}{1}}\leq\delta} \,
    \I{\dist{\X^{\p{k}}{1}}\leq\delta} \, \I{\dist{\X^{\p{m}}{1}}\leq\delta}\I{E^{c}}}
    \begin{details}
      &\lesssim \delta^4\, \exp\p{3 R_{\delta}} \, \tau_{i,j}^{-1/2} \tau_{j,k}^{-1/2} \tau_{k,m}^{-1/2}\\
  \end{details}
    & \lesssim \delta^{4-r}
      \, \tau_{i,j}^{-1/2} \tau_{j,k}^{-1/2} \tau_{k,m}^{-1/2},
  \end{aligned}
\]
for any \(r{>}0\) and \(\prob{E}\) remains \(\Order{h_\ell^{\beta q/2}\delta^{-q}} \) as
before when \(\delta\) is as previously chosen. Therefore the final conclusion
remains that the kurtosis is \(\order{h_\ell^{-\nu}}\) for any \(\nu{>}0\).

 \end{appendix}

\begin{acks}
We thank Soeren Wolfers for the helpful discussion regarding
\cref{thm:bounded-K-prob}. MBG gratefully acknowledges research funding from
the UK EPSRC (ICONIC programme grant EP/P020720/1), and the Hong Kong
Innovation and Technology Commission (InnoHK Project CIMDA).
 \end{acks}

\bibliographystyle{imsart-number}

\end{document}